%% file: paper.tex
\documentclass[12pt]{amsart}

\usepackage{amsmath, amsthm, amssymb}
\usepackage{graphics}
\usepackage{pinlabel}
\usepackage{color}
\usepackage[colorlinks,citecolor=blue,linkcolor=red]{hyperref}
\usepackage{graphicx}
\usepackage{nicefrac}						
\usepackage{mathtools}
\usepackage{float}
\usepackage{tikz-cd}
\usepackage{enumerate}

\usepackage{secdot}
\usepackage{caption}

\usepackage{verbatim} 
 
\usepackage[margin=1in]{geometry} 

\newcommand{\N}{\mathcal{N}} 
\newcommand{\NS}{\mathcal{N}_1(S)} 

\newcommand{\C}{\mathcal{C}}

\newcommand{\qi}{\cong_{QI}}

\DeclareMathOperator{\Aut}{Aut}
\DeclareMathOperator{\Mod}{Mod}
\DeclareMathOperator{\diam}{diam}
\DeclareMathOperator{\link}{link}
 
\newtheorem{theorem}{Theorem}[section]
\newtheorem{lemma}[theorem]{Lemma}
\newtheorem{remark}[theorem]{Remark}
\newtheorem{definition}[theorem]{Definition}

\newtheorem{conjecture}[theorem]{Conjecture}
\newtheorem*{metaconjecture*}{Ivanov's metaconjecture}
\newtheorem{proposition}[theorem]{Proposition}

\newtheorem*{mainthm*}{Main Theorem}
\newtheorem{claim}{Case}

\begin{document}

\title{Automorphisms of the k-Curve Graph}

\author{Shuchi Agrawal, Tarik Aougab, Yassin Chandran, Marissa Loving, \\ J. Robert Oakley, Roberta Shapiro, and Yang Xiao}

 \begin{abstract}
Given a natural number $k$ and an orientable surface $S$ of finite type, define the \textit{k-curve graph} to be the graph with vertices corresponding to isotopy classes of essential simple closed curves on $S$ and with edges corresponding to pairs of such curves admitting representatives that intersect at most $k$ times. We prove that the automorphism group of the $k$-curve graph of a surface $S$ is isomorphic to the extended mapping class group for all $k$ satisfying $k \leq |\chi(S)|- 512$. We prove the same result for the so-called \textit{systolic complex}, a variant of the curve graph with many complete subgraphs coming from interesting collections of systoles with respect to a hyperbolic metric. This resolves a conjecture of Schmutz Schaller.
\end{abstract}

\maketitle

\section{Introduction}

Let $S$ be a connected, orientable surface of genus $g$ possibly with finitely many punctures $p$, and let $\Mod^{\pm}(S)$ denote the extended mapping class group. The \textit{curve complex}, $\C(S)$, is a flag simplicial complex whose vertices correspond to isotopy classes of essential simple closed curves and whose edges represent pairs of such classes that can be realized disjointly on $S$. A celebrated theorem of Ivanov \cite{Ivanov} identifies $\mbox{Aut}(\mathcal{C}(S))$ with $\mbox{Mod}^{\pm}(S)$ in all but finitely many cases. This result inspired a flurry of results in related contexts, where $\mbox{Mod}(S)$ acts by simplicial automorphisms on some graph whose vertices represent homotopy classes of curves and/or arcs \cite{BrendleMargalit, Disarlo, Irmak, IrmakKorkmaz, KorkmazPapa, Luo, Schaller} or finite collections of curves and arcs \cite{KorkmazPapa2, Margalit}, or subsurfaces \cite{BrendleMargalit, McCarthyPapa}. 

In many of these papers, the result is that the full automorphism group of the complex being considered is $\Mod^{\pm}(S)$, or at least virtually so in a finite number of sporadic cases, and the proofs all factor through Ivanov's original theorem by showing that any automorphism of a particular complex induces one of $\mathcal{C}(S)$. This led to Ivanov's metaconjecture, which is discussed further in Chapter 4: \emph{Fifteen problems about the mapping class group} of \cite[p.~77]{Farb}.

\begin{metaconjecture*} Any ``sufficiently rich" complex naturally associated to a surface should have $\Mod^{\pm}(S)$ as its group of automorphisms, and furthermore, a proof of this exists which factors through Ivanov's original theorem.  \end{metaconjecture*} 

The focus of this paper is to verify the metaconjecture for an infinite family of curve graphs whose edges represent bounded intersection. In particular, we will consider the following natural generalization of $\C(S)$. For any $k \in \mathbb{N}$, the $k$-\textit{curve graph} is defined to be the graph whose vertices are those of $\mathcal{C}(S)$ and whose edges represent homotopy classes of curves with geometric intersection number at most $k$. Our main result characterizes $\Aut(\mathcal{C}_k(S))$ when $|\chi(S)|$ is sufficiently large relative to $k$.

\begin{theorem} \label{k-curve} Suppose $|\chi(S)| \geq k+ 512$. Then the natural map 
\[ \mbox{Mod}^{\pm}(S) \rightarrow \mbox{Aut}(\mathcal{C}_k(S)) \]
is an isomorphism. 
\end{theorem}

When $k = 1$, Theorem \ref{k-curve} holds without the restriction on the Euler characteristic of $S$. The proof has been omitted for the sake of clarity since it is nearly identical to the proof of Theorem \ref{non-separating}. Furthermore the lower bound of $k + 512$ on $|\chi(S)|$ is not sharp. See the appendix for details on how this bound is derived. 

Theorem \ref{k-curve} addresses Part (3) of Question $7.4$ of Margalit's collection of open problems \cite{Margalit}. It represents a first step towards resolving Ivanov's metaconjecture in the cases where edges do not represent disjointness. In addition, to the authors' knowledge it is only the third result in the literature which resolves Ivanov's conjecture for an infinite family of simplicial complexes. The first such result was the work of Brendle--Margalit for \textit{complexes of regions} \cite{BrendleMargalit} and the second is McLeay's extension of their work from closed surfaces to punctured surfaces (including those of genus $0$) \cite{McLeay}. 

In addition to results concerning simplicial automorphisms mentioned above, there are a number of theorems characterizing simplicial injections \cite{Aramayona, AramayonaLeininger, AramayonaKoberdaParlier, Irmak2, IrmakMcCarthy}, quasi-isometries \cite{RafiSchleimer},
and other types of structure-preserving maps of $\mathcal{C}(S)$ and related complexes. For example, in \cite{RafiSchleimer}, Rafi-Schleimer identify the group of quasi-isometries of $\mathcal{C}(S)$ with $\mbox{Mod}^{\pm}(S)$. We remark that even though $\mathcal{C}_{k}(S)$ is quasi-isometric to $\mathcal{C}(S)$, this result does not imply Theorem \ref{k-curve}. Indeed, a priori it is possible that an automorphism of $\mathcal{C}_{k}(S)$ moves every vertex a uniformly bounded distance, and would therefore be equivalent to the identity as a quasi-isometry.  

We also consider the following variant of the curve graph, which we denote $\mathcal{SC}(S)$. The vertices of this graph correspond to isotopy classes of essential curves which are either non-separating curves, or separating curves which bound a twice punctured disk on one side. The edges represent pairs of such curves that intersect minimally, that is at most once in the case that both vertices correspond to non-separating curves, and at most twice when at least one of those vertices is a separating curve. The notation $\mathcal{SC}(S)$ is due to Schmutz Schaller \cite{Schaller}, and stands for the \textit{systolic complex}, as interesting sets of systoles on a hyperbolic surface correspond to complete subgraphs of $\mathcal{SC}(S)$. However, Anderson--Parlier--Pettet \cite{APP} give examples of complete subgraphs of $\mathcal{SC}(S)$ which are not realizable as the set of systoles for any hyperbolic metric on $S$.

\begin{theorem} \label{Schaller} If $S$ is a closed surface with genus $g \geq 3$, then the natural map  
\[ \mbox{Mod}^{\pm}(S) \rightarrow \mbox{Aut}(\mathcal{SC}(S))  \]
is an isomorphism. If $g = 2$, then the above map is surjective with kernel $\mathbb{Z}/2\mathbb{Z}$ corresponding to the hyperelliptic involution. If $S$ is a surface of genus $g$ with $p > 0$ punctures and $\chi(S) < 0$, then the above map is an isomorphism for $(g,p) \neq (1,2), (1,3), (0,5)$.
\end{theorem}

Theorem \ref{Schaller} represents an almost complete resolution to the conjecture of \cite{Schaller}; we remark that our techniques do not cover the cases $(g,p)= (1,1), (1,2),(1,3),(0,4),(0,5)$. Following the outline of Ivanov's metaconjecture, our proof strategy relies on showing that any automorphism of $\mathcal{SC}(S)$ induces one of $\C(S)$. This fails when $(g, p) = (1, 2)$, since Luo proved in \cite{Luo} that the curve complex of the twice punctured torus admits automorphisms that are not induced by homeomorphisms. In the cases $(g,p)=(1,1)$ and $(0,4)$, the systolic complex is isomorphic to the $1$-skeleton of the Farey tesselation of the hyperbolic plane, whose automorphism group is $PGL(2, \mathbb{Z})$ and so the theorem is known. Thus, the only remaining cases are $(g,p)= (1,2), (1,3)$ and $(0,5)$.

When $g \neq 0$, one can also consider the subgraph of $\mathcal{SC}(S)$ consisting only of non-separating curves. Note that in the event that $S$ is closed, this is the entirety of $\mathcal{SC}(S)$. We denote this graph by $\mathcal{N}_{1}(S)$ and give the following characterization of its automorphisms.

\begin{theorem} \label{non-separating}
Suppose that $g \geq 1$ and that $(g,p) \neq (1,2)$. Then the natural map 
\[ \mbox{Mod}^{\pm}(S) \rightarrow \mbox{Aut}(\mathcal{N}_1(S))\]
is an isomorphism for $(g,p) \neq (2,0)$ and a surjection with kernel $\mathbb{Z}/2\mathbb{Z}$ otherwise.
\end{theorem}

\subsection{Idea of the proofs.} In both Theorems \ref{Schaller} and \ref{k-curve}, one needs to show that an automorphism of either $\mathcal{SC}(S)$ or of $\mathcal{C}_k(S)$ preserves edges that represent disjointness. In what follows, we let $\text{link} ( \cdot )$ denote the link of a vertex, the subgraph induced by the set of vertices adjacent to a given vertex. Given a pair of curves $\alpha, \beta$ connected by an edge, we study the subgraph $L(\alpha, \beta) = \text{link}(\alpha) \cap \text{link}(\beta)$; we refer to such a subgraph as the \textit{link of an edge} or \textit{edge link}.


In particular, we prove that the diameter of an edge link distinguishes between edges corresponding to disjoint curves and edges corresponding to curves intersecting once. For larger values of $k$, the diameters may not be sufficient to pick out the edges respresenting disjoint pairs, so we need a more careful analysis of the types of geodesics that exist in each edge link. We show that, under the additional hypothesis that the surface is sufficiently large, an edge link representing a pair of non-disjoint curves always has finite diameter. Furthermore, there must always exist a finite number of vertices, which we call \textit{shortcut curves}, so that for any two vertices $u,v$ in the link whose edge link distance is maximal, there is a geodesic from $u$ to $v$ that passes through a shortcut curve. This additional geometric property distinguishes edges representing disjoint curves from all other types of edges.

Throughout this paper we will employ both combinatorial and coarse-geometric techniques; for example, we use the technology of \textit{subsurface projections} to compute exact diameters of the edge links. Given a pair of curves $\alpha$ and $\beta$ intersecting $k$ times, a standard surgery argument going back to Lickorish \cite{Lickorish} yields a curve $\delta$ which intersects $\alpha$ at most once and $\beta$ at most $k/2$ times. In the proof of Theorem \ref{k-curve}, as opposed to such a $\delta$ we have need of a curve $\delta'$ which is disjoint from $\alpha$ and which intersects $\beta$ at most $k/4$ times. For this, we use a variant of a proposition due to the second author, used to prove that curve graphs are uniformly hyperbolic~\cite{Aougab2}. 

\subsection{Outline of the paper.} Section \ref{prelim} contains a brief introduction to curves on surfaces, several relevant graphs of curves associated to surfaces, the notion of subsurface projections, and some relevant coarse geometry. In Section \ref{section:one-intersection} and Section \ref{section:disjoint}, we compute diameters of edge links in $\N_1(S)$ and use these to prove Theorem \ref{non-separating} at the end of Section \ref{section:disjoint}. Section \ref{section:systolic} provides a proof of Theorem \ref{Schaller}, resolving the conjecture from \cite{Schaller}. Lastly, in Section \ref{section:k-curve} we prove our main result, Theorem \ref{k-curve}. We also include an appendix which contains the sketch of several known results which are needed in the proof of Theorem \ref{k-curve}. Appendix \ref{appendix} also provides an explanation of the restriction on the Euler characteristic of the surfaces required by Theorem~\ref{k-curve}.

\subsection{Acknowledgements.} Much of this paper is based on work supported by the National Science Foundation under Grant No. DMS-1439786 while the authors were in residence at the Institute for Computational and Experimental Research in Mathematics (ICERM) in Providence, RI, during the Summer@ICERM program. We sincerely thank ICERM for its hospitality. In addition, Aougab was partially supported by NSF grants DMS-1502623 and DMS-1807319 and Loving was supported by the NSF Graduate Research Fellowship under Grant No. DGE 1144245. Thank you to Moira Chas and Jonah Gaster for many helpful conversations and Nick Bell, Maxime Fortier Bourque, Dan Margalit, Alan McLeay, Athanase Papadopoulos, Joe Scull, and Davide Spriano for their comments on a draft of this paper. We would also like to thank the anonymous referee whose careful reading and thorough comments greatly improved the exposition of this paper. 

\section{Preliminaries} 
\label{prelim} 

Throughout the paper, unless otherwise noted, $S$ is an orientable surface of finite type, possibly with punctures and/or boundary components.

\subsection{Curves and arcs.} A \textit{simple closed curve} on $S$ is a homotopy class of maps $S^{1} \rightarrow S$ admitting a representative that is an embedding. We will often abuse notation and identify a simple closed curve with an embedded representative, and further identify this embedded representative with its image in $S$. A simple closed curve is \textit{essential} if it is not homotopically trivial, and if it is not homotopic to a map whose image  bounds a once-punctured disk on one side.

\begin{figure}[htb!]
\centering
\def\svgwidth{4in}
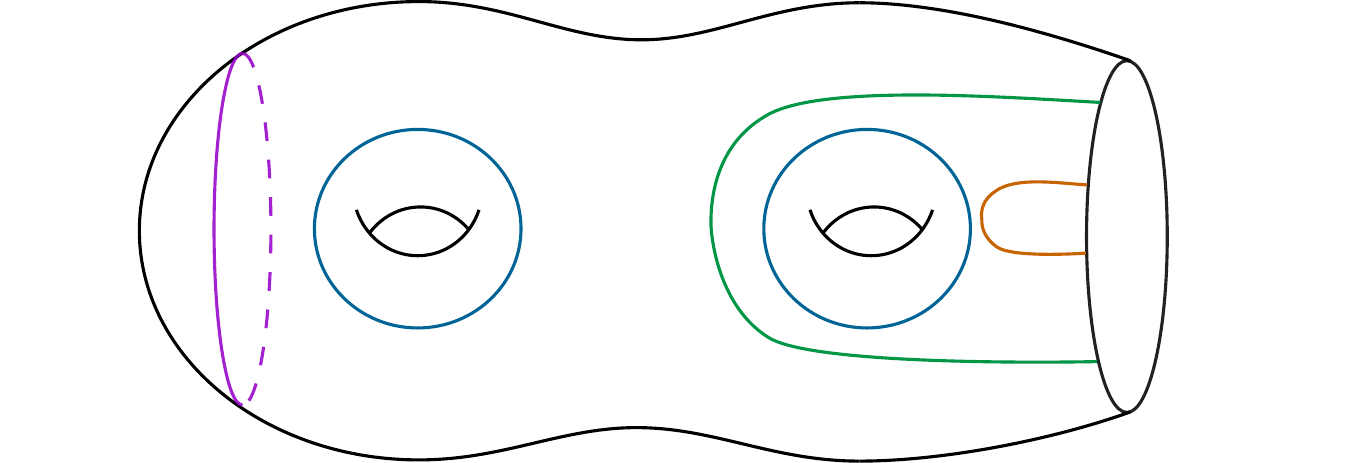
\caption{The boundary slide homotopy class of $\beta$ is an essential arc, whereas $\alpha$ is not since it co-bounds a disk with an arc of $\partial S$. The homotopy class of each $\gamma_i$ is essential and that of $\eta$ is not since $\eta$ bounds a disk on the surface. The collection $\{\gamma_1, \gamma_2\}$ is a multi-curve.}
\label{fig:essential-arcs-curves}
\end{figure}

Let $f, g: (0,1) \rightarrow S$ be two embeddings so that $\displaystyle \lim_{t \rightarrow 0,1} f(t)$ are either punctures or points on boundary components, and similarly for $g$. Then $f$ and $g$ are \textit{boundary-slide homotopic} if there is a homotopy $H:~[0,1] \times (0,1) \rightarrow S$ from $f$ to $g$ so that $\displaystyle \lim_{s \rightarrow 0}H(t,s)$ is either a puncture or on a fixed boundary component for all $t$, and similarly as $s \rightarrow 1$. Then an \textit{essential arc} will be a non-trivial boundary-slide homotopy class of such maps. 

\begin{figure}[htb!]
\centering
\def\svgwidth{2in}
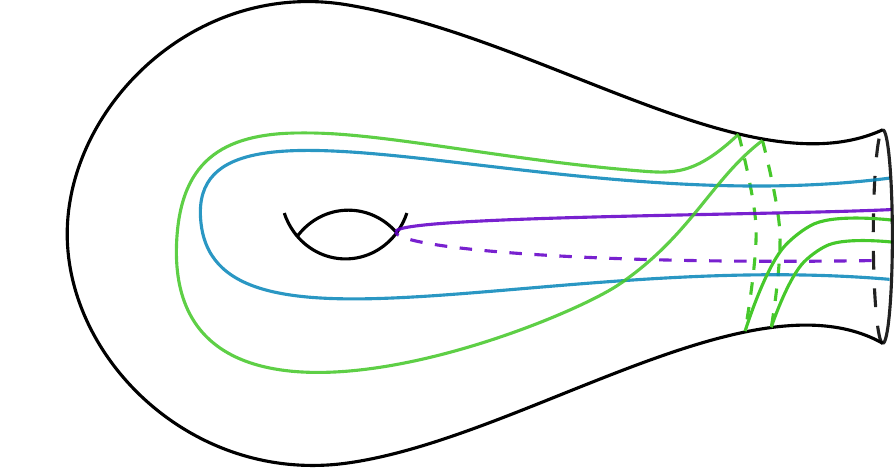
\caption{The green arc $\alpha$ and the blue arc $\beta$ are boundary-slide homotopic. Neither is boundary-slide homotopic to $\gamma$. The boundary-slide homotopy class of $\alpha$ and $\beta$, denoted $[\alpha]$, is an essential arc, as is the class of $[\gamma]$. Since $\gamma \cap \beta = \emptyset$ and $\beta \in [\alpha]$, then $\{[\alpha], [\gamma]\}$ is a multi-arc.}
\label{fig:boundary-slide}
\end{figure}

Given a pair of essential simple closed curves or arcs $\alpha, \beta$, their \textit{geometric intersection number} $i(\alpha, \beta)$ is the minimum of $|\alpha' \cap \beta'|$ taken over all representative images $\alpha' \subset S$ of $\alpha$ and $\beta' \subset S$ of $\beta$. Note that $\alpha$ can be a curve and $\beta$ an arc here. If $\alpha', \beta'$ realize the geometric intersection number for their respective homotopy classes, they are said to be in \textit{minimal position}.
	
A \textit{multi-curve} (resp. \textit{multi-arc}) is a collection of pairwise distinct essential simple closed curves (resp. arcs) whose pairwise geometric intersection numbers are all $0$. A collection of pairwise disjoint curves and arcs will, by convention, be referred to as a multi-curve. As is well known, for $(g,p) \neq (1,0)$, any multi-curve on $S$ consisting of curves contains at most $3g+p-3$ connected components and this bound is realizable. 

Lastly we introduce the notion of a \textit{weighted multi-arc}, which will be used in the proof of Proposition \ref{k jointly separating} and in Appendix \ref{appendix}. A \textit{weighted multi-arc} is a multi-arc with positive integer weights assigned to each arc. We use $|\alpha|$ to denote the number of arcs in a multi-arc $\alpha = \{a_1, a_2, \dots, a_n\}$, and $w(\alpha) = \displaystyle \sum_{i=1}^n w_i$ to denote the total weight, where $w_i$ is the weight assigned to arc $a_i \in \alpha$. 
	
\subsection{Relevant graphs and their automorphisms} 

In this section, we will introduce various graphs whose vertices will represent curves or arcs in $S$ and with edges corresponding to various constraints on the geometric intersection number. We call an edge connecting vertices $v$ and $w$ an \textit{$n$-edge} if the curves corresponding to $v$ and $w$ minimally intersect $n$ times. An edge is a \textit{non-zero} edge if $n\neq 0.$

Let $\mathcal{AC}(S)$, the \textit{arc and curve graph}, be the graph whose vertices correspond to essential simple closed curves and arcs on $S$, and whose edges correspond to disjoint pairs, so all edges are 0-edges. When $S$ is an annulus, $\mathcal{AC}(S)$ consists only of arcs connecting the two boundary components. We note that $\mathcal{AC}(S)$ is connected when $3g+n+b\geq 5$: Theorem~4.3 in \cite{primer} asserts that $\mathcal{C}(S)$ is connected, and any arc is disjoint from at least one curve.

Define $\mathcal{N}(S)$, the \textit{non-separating curve graph}, to be the graph whose vertices are non-separating simple closed curves with edges between classes that admit disjoint representatives. All edges in $\mathcal{N}(S)$ are 0-edges. It is known that $\mathcal{N}(S)$ is connected for genus $g \geq 2$ (see Theorem~4.4 in \cite{primer}). In adddition, $\mathcal N(S)$ is infinite diameter. The argument of Masur--Minsky \cite{MasurMinsky} for the infinite diameter of $\mathcal C(S)$ apply fairly directly to $\mathcal{N}(S)$. Irmak showed in \cite{Irmak} that for surfaces of with $g >1,$ $\Aut(\mathcal{N}(S)) \cong \Mod^\pm(S).$

The \textit{systolic complex}, denoted $\mathcal{SC}(S)$, is defined differently for closed and punctured surfaces. When $S$ is closed, $\mathcal{SC}(S)$ has vertices corresponding to isotopy classes of non-separating curves and whose edges represent pairs of such curves with geometric intersection number at most $1$. In this case, $\mathcal{SC}(S)$ contains both 0-edges and 1-edges. If $S$ is not closed, $\mathcal{SC}(S)$ has an additional vertex for each separating curve that bounds a twice punctured disk on one side; such vertices are connected to others by an edge when there are at most two intersections. We observe that for closed surfaces, $\mathcal{SC}(S)$ is equal to $\mathcal{N}_1(S).$

We define $\mathcal{N}_1(S)$, the \textit{non-separating 1-curve graph}, as the subgraph of $\mathcal{SC}(S)$ consisting only of non-separating curves. To the best of the authors' knowledge, this graph has not been otherwise named or studied extensively in the literature. Note that $\N(S)$ is a subgraph of $\N_1(S)$ with the same vertex set, and therefore $\N_1(S)$ is connected when $g \geq 2$ because $\N(S)$ is connected. When $g \geq 2$, $\N_1(S)$ is quasi-isometric to $\N(S)$ and therefore it is infinite diameter. See \ref{non-separating qi}.

Along with $\mathcal{SC}(S)$, Schmutz Schaller defined the graph $\mathcal{G}(S)$. When $g \geq 1$, $\mathcal{G}(S)$ has as its vertex set the collection of all non-separating curves and has edges corresponding to pairs of curves intersecting exactly once. This means that for closed surfaces of positive genus, $\mathcal{G}(S)$ is a subgraph of $\mathcal{SC}(S)$. When $g= 0$, $\mathcal{G}(S)$ has as its vertex set the collection of all curves bounding a twice punctured disk on one side and whose edges correspond to pairs of curves intersecting exactly twice. Again, this is a subgraph of $\mathcal{SC}(S)$. It is a result of Schaller from \cite{Schaller} that $\Aut(\mathcal{G}(S)) \cong \Mod^{\pm}(S).$ In the same paper, Schaller conjectured the following, which we resolve in Theorem \ref{Schaller} for all but $(g,p)= (1,3), (0,5)$.

\begin{conjecture}[Schmutz Schaller, \cite{Schaller}]
\label{SchallerC}
The automorphism group of $\mathcal{SC}(S)$ is isomorphic to $\Mod^\pm(S).$
\end{conjecture}

\begin{figure}[htb!]
\centering
\def\svgwidth{6in}
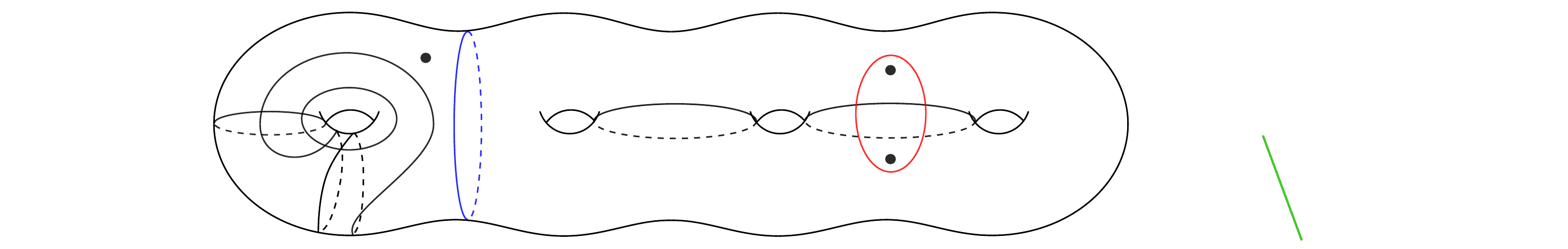
\caption{The black curves $\alpha, \beta, \varepsilon, \eta,$ and $\gamma$ are all vertices of $\mathcal N_1(S)$. The black edges in the graph record adjacencies in $\mathcal N_1(S)$. The curve $\rho$ is not in $\mathcal N_1(S)$, but it is in $\mathcal{SC}(S)$. The red edges represent adjacencies in $\mathcal{SC}(S)$. Note that the black edges are also in $\mathcal{SC}(S)$. The green edge represents adjacency in $\mathcal{G}(S)$. The curve $\omega$ is not a vertex in $\mathcal N_1(S), \mathcal{SC}(S)$, or $\mathcal G(S)$.} Note that not all the edges in this subgraph are shown to avoid cluttering the diagram.
\label{fig:section-2.2}
\end{figure}

For any one of the above mentioned graphs and for the curve complex $\mathcal{C}(S)$, we obtain a metric on the vertex set by identifying each edge with the unit interval and defining the distance between two vertices to be the minimum number of edges contained in any edge path between them. Given one of these graphs $G$, the distance function will be denoted by $d_{G}( , )$. All graphs mentioned above are infinite diameter; for all but finitely many surfaces this follows directly from work Masur--Minsky on the infinite diameter of $\mathcal C(S)$ \cite{MasurMinsky}. 

The following lemma establishes a quasi-isometry between $\mathcal{AC}_1(S)$ and $\mathcal{AC}(S)$, where $\mathcal{AC}_1(S)$ has the same vertex set as the standard arc and curve graph, but with edges when there is at most one intersection.

\begin{proposition}\label{qi}
If $S$ is a surface with punctures or with non-empty boundary and so that $(g,p) \neq (0,3)$, then \[ \mathcal{AC}_{1}(S) \qi \mathcal{AC}(S),\] where $\mathcal{AC} _1 (S)$ is the arc and curve graph of $S$ with both $0$-edges and $1$-edges. Moreover, both are infinite diameter. 
\end{proposition} 

\begin{proof}

Let $\phi: \mathcal{AC}_1(S)\longrightarrow \mathcal{AC}(S)$ be the identity map on the vertices. 

\medskip
For any vertices $u,v$ in $\mathcal{AC}_1(S))$, suppose $d_{\mathcal{AC}_1}(u,v) = l$ for some $l>0$. If $l = 1$ and $i(u,v) = 1$, either the genus $g$ is non-zero or at least one of $u,v$ is an arc. If $g>1$ or if the number of punctures $p>1$, there is an essential curve in $S\setminus (u \cup v)$, and thus $d_{\mathcal{AC}}(\phi(u), \phi(v)) = 2$. 

\medskip
If $(g,p)= (1,1)$ and $u,v$ are curves, then without loss of generality one is the $1/0$ and the other is the $0/1$ curve, and then there are disjoint arcs $\lambda_u, \lambda_v$ so that $i(\lambda_u, u)= i(\lambda_v,v)= 0$, and so $d_{\mathcal{AC}}(\phi(u), \phi(v)) \leq 3$. 

If one of $u,v$ is a curve and the other is an arc, there is an arc disjoint from both. These are shown in Figure~\ref{fig:qi1}. 

If both are arcs, Hatcher's original surgery argument for the contractibility of the arc complex (see the Main Theorem of \cite{Hatcher} or Theorem~5.5 of \cite{primer}) implies that the distance from $u$ to $v$ in the arc complex of $S$ is at most $2$. 

\begin{figure}[htb!]
\centering
\def\svgwidth{4.5in}
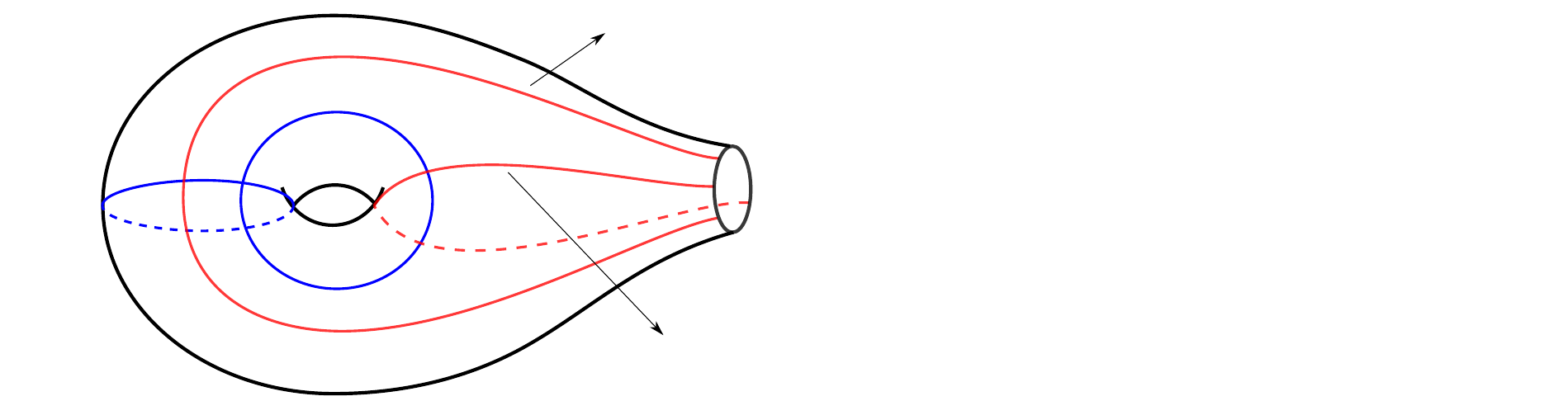
\caption{Two possible configurations when $(g,p) = (1, 1)$: on the left when $u$ and $v$ are curves and on the right when $u$ is a curve and $v$ is an arc.}
\label{fig:qi1}
\end{figure}

 \medskip
Finally, if $g= 0$, we will argue in such a way that the proof for $p= n$ implies a proof for $p >n$, and thus we can assume that $p=4$. 

If both $u$ and $v$ are arcs, at most one can be separating. In this case, it is easy to find an arc in the complement of $u$ and $v$: without loss of generality, $u$ separates two punctures from another, and $v$ only witnesses one of the two punctures on one side of $u$. Thus there is an arc $\lambda$ connecting the two punctures on one side of $u$, disjoint from $v$. If neither $u$ nor $v$ separate, cutting along one produces a $3$-holed sphere, in which the other arc becomes two arcs. Given any two disjoint arcs in a $3$-holed sphere, there is always a third essential arc disjoint from both, as shown in Figure~\ref{fig:qi2}). 

Lastly, if $u$ is an arc and $v$ a curve, $u$ can not be separating and so up to homeomorphism there is a unique configuration: $v$ separates two punctures from the other two, and $u$ connects one puncture on one side of $v$ to one on the other. It is then straightforward to find an arc disjoint from both. 
 
\begin{figure}[htb!] 
\centering
\def\svgwidth{3.5in}
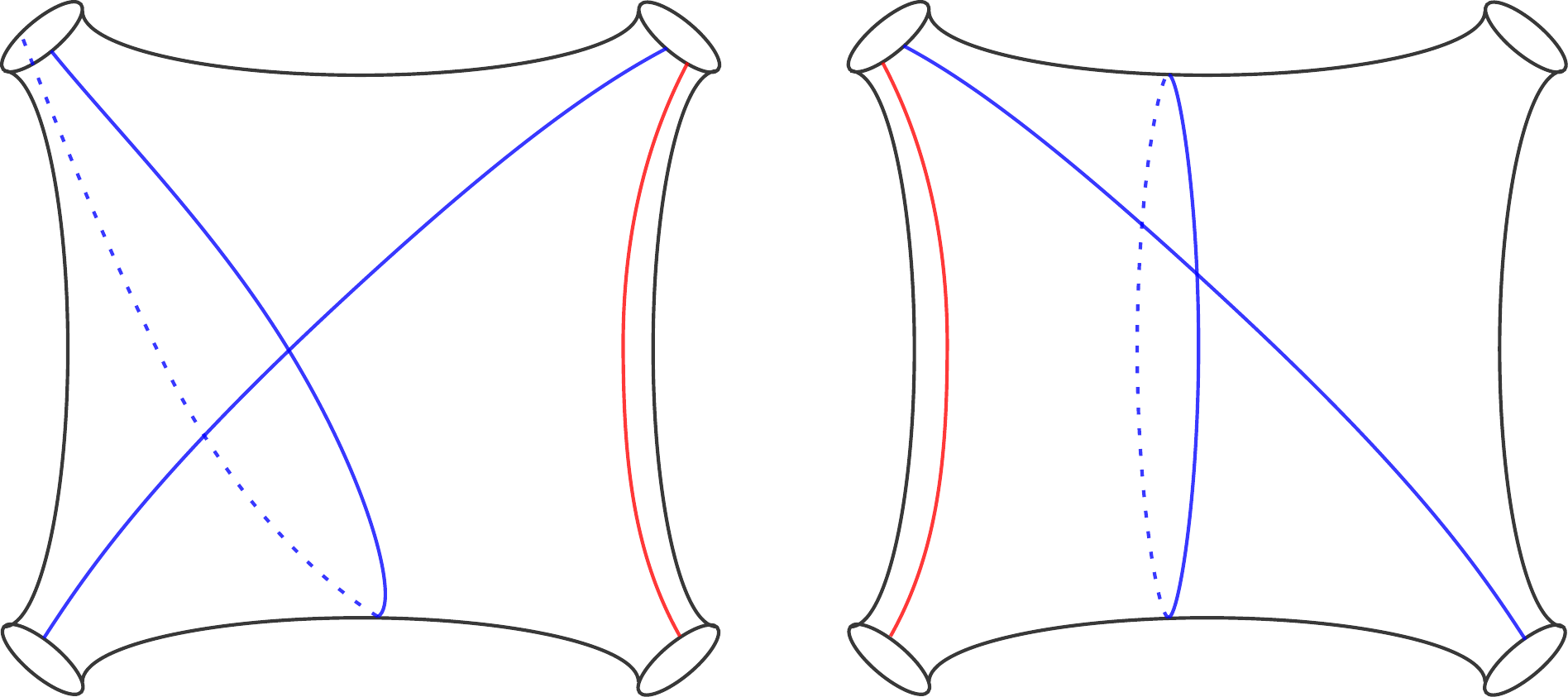
\caption{Two possible configurations when $(g,p) = (0, 4)$. On the left when $u$ and $v$ are both arcs and $u$ is separating. On the right when $u$ is an arc and $v$ is a curve.}
\label{fig:qi2}
\end{figure}
 
 If $l>1$ and if the shortest path between $u$ and $v$ contains no 1-edges, then $d_{\mathcal{AC}}(\phi(u),\phi(v))=l$ as well. If the shortest path contains a 1-edge $(\rho,\eta)$, then as above, we choose a path of length at most $3$ through $0$-edges from $\rho$ to $\eta$. Therefore $d_{\mathcal{AC}}(\phi(u),\phi(v)) \leq 3l$. 
 
 When $(g,p) \neq (1,1)$, that $\mathcal{AC}(S)$ has infinite diameter follows from the fact that $\mathcal{C}(S)$ has infinite diameter \cite{MasurMinsky} and from Theorem $1.3$ of \cite{KorkmazPapa} which states that $\mathcal{AC}(S) \qi \mathcal{C}(S)$. When $(g,p)=(1,1)$, the graph $\mathcal{AC}_1(S)$ is quasi-isometric to the Farey graph, which is infinite diameter. \end{proof}

\begin{remark} \label{loxodromic} 

Masur and Minsky \cite{MasurMinsky} not only show that $\mathcal{C}(S)$ is infinite diameter, but that the orbit of any pseudo-Anosov mapping class is as well. Since the quasi-isometries between $\mathcal{AC}(S), \mathcal{C}(S), $ and $\mathcal{AC}_1(S)$ are all $\mbox{Mod}(S)$-equivariant, the same is true for $\mathcal{AC}(S)$ and $\mathcal{AC}_1(S)$. 

\end{remark}

\begin{remark} \label{non-separating qi} 

When $g \geq 2$, a simplified version of the above argument demonstrates that $\N(S)$ and $\N_1(S)$ are quasi-isometric, via the identity map on the vertices. 
\end{remark}

We conclude this subsection with an argument, due originally to Hempel \cite{Hempel} (see also \cite{Schleimer}), which implies the connectedness of the curve complex. We reference the argument explicitly since it will be called upon in Section~\ref{section:k-curve}. Hempel's argument proves that given $\alpha, \beta$ simple closed curves on a surface $S$ representing vertices of $\mathcal{C}(S)$, 
\[ d_{\mathcal{C}(S)}(\alpha, \beta) \leq 2 \log_{2}(i(\alpha, \beta))+ 2. \]
 Hempel considers a simple surgery that replaces $\alpha$ with a curve $\alpha’$ intersecting $\beta$ at most half the number of times that $\alpha$ intersects $\beta$. The curve $\alpha’$ is obtained by following along $\alpha$ until the arrival at a chosen intersection with $\beta$, and then following along $\beta$ until the arrival at a subsequent intersection, at which point $\alpha’$ follows along $\alpha$ again until closing up.

\subsection{Subsurface projections.}
An \textit{essential subsurface} of $S$ is a pair $(Y, i_Y)$ where $Y$ is a surface (potentially with boundary), and $i_Y: Y \hookrightarrow S$ is a $\pi_{1}$-injective map and an embedding on the interior of $Y$, so that each component of $\partial Y$ maps to either an essential simple closed curve on $S$ or a component of $\partial S$. We will often identify an essential subsurface with its image in $S$. When $Y$ is an annulus we say that it is an \textit{annular} subsurface, and otherwise that $Y$ is \textit{non-annular}. An essential simple closed curve or arc is said to be in \textit{minimal position} with an essential subsurface $Y$ when it is in minimal position with all components of $\partial Y$. 

\begin{figure}[htb!]
\centering
\def\svgwidth{5in}
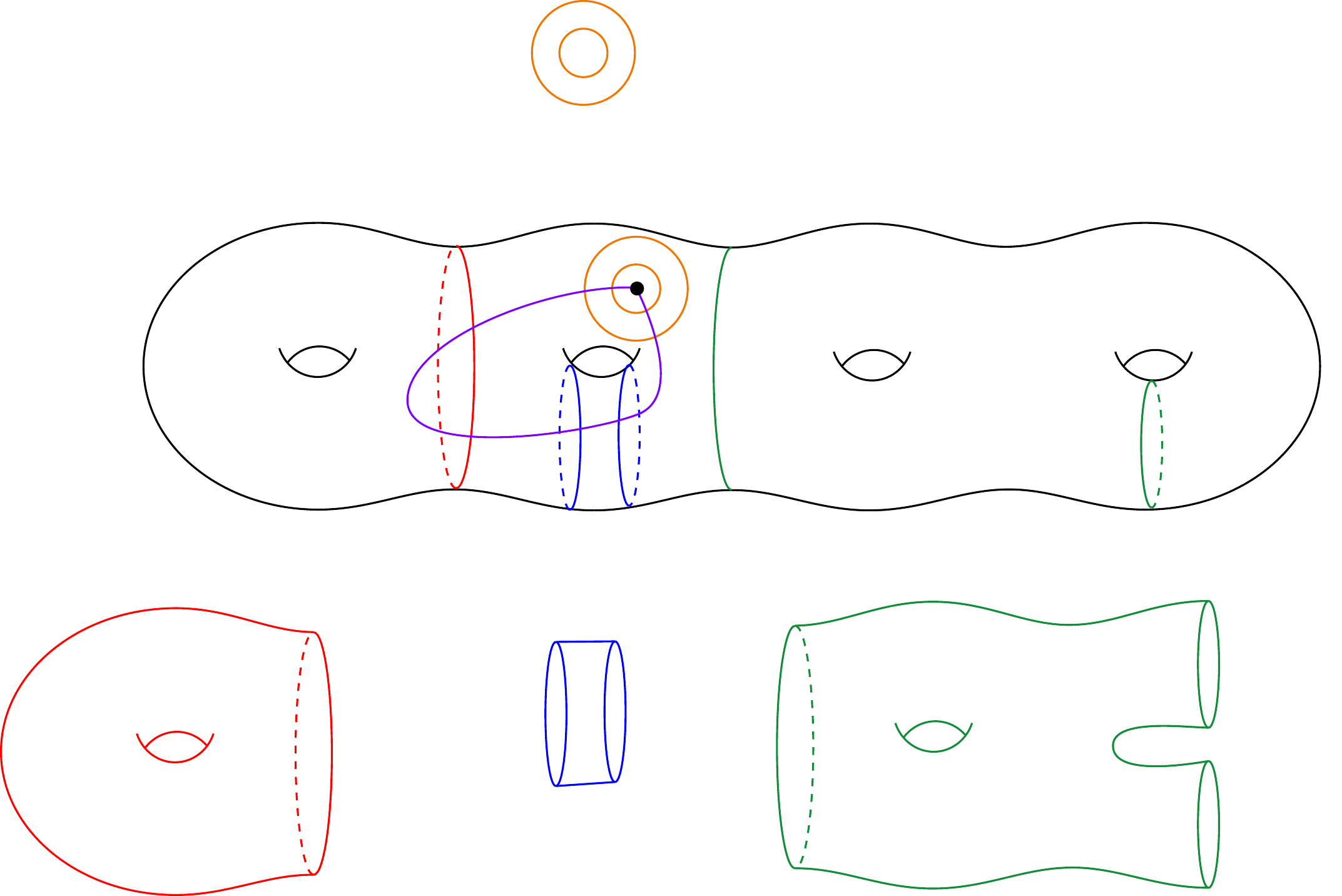
\caption{The pairs $(Y, i_Y)$ and $(W, i_W)$ are examples of essential non-annular subsurfaces, and $(Z, i_Z)$ is an essential annular subsurface. The arc $\alpha$ is in minimal position with $Z$ and $W$ but not with $Y$.}
\label{fig:subsurface-proj-1}
\end{figure}

Given $Y\subset S$ a non-annular essential subsurface, the \emph{subsurface projection} \[ \pi_Y : \mathcal{C}(S) \longrightarrow \mathcal{P}(\mathcal{AC}(Y))\]
takes a vertex $\alpha \in \mathcal{C}(S)$ to the multi-curve in $Y$ obtained by taking all distinct homotopy classes occurring in the intersection of $\alpha$ with $Y$, after $Y$ and $\alpha$ are put in minimal position (see \cite{MasurMinsky2} for more details).  

When $Y$ is an annulus, we first consider the cover $S_Y$ of $S$ corresponding to $\pi_1 Y$. This cover compactifies to an annulus, and we let $\pi_Y(\alpha)$ be  any lift of $\alpha$ to this annulus that connects its two boundary components (see Figure \ref{fig:subsurface-proj-2}).  
 
We define the $Y$-\textit{subsurface distance} as $d_Y(\alpha,\beta):=\mbox{diam}_{\mathcal{AC}(Y)}(\pi_Y(\alpha) \cup \pi_Y(\beta)).$

\begin{figure}[htb!]
\centering
\def\svgwidth{5in}
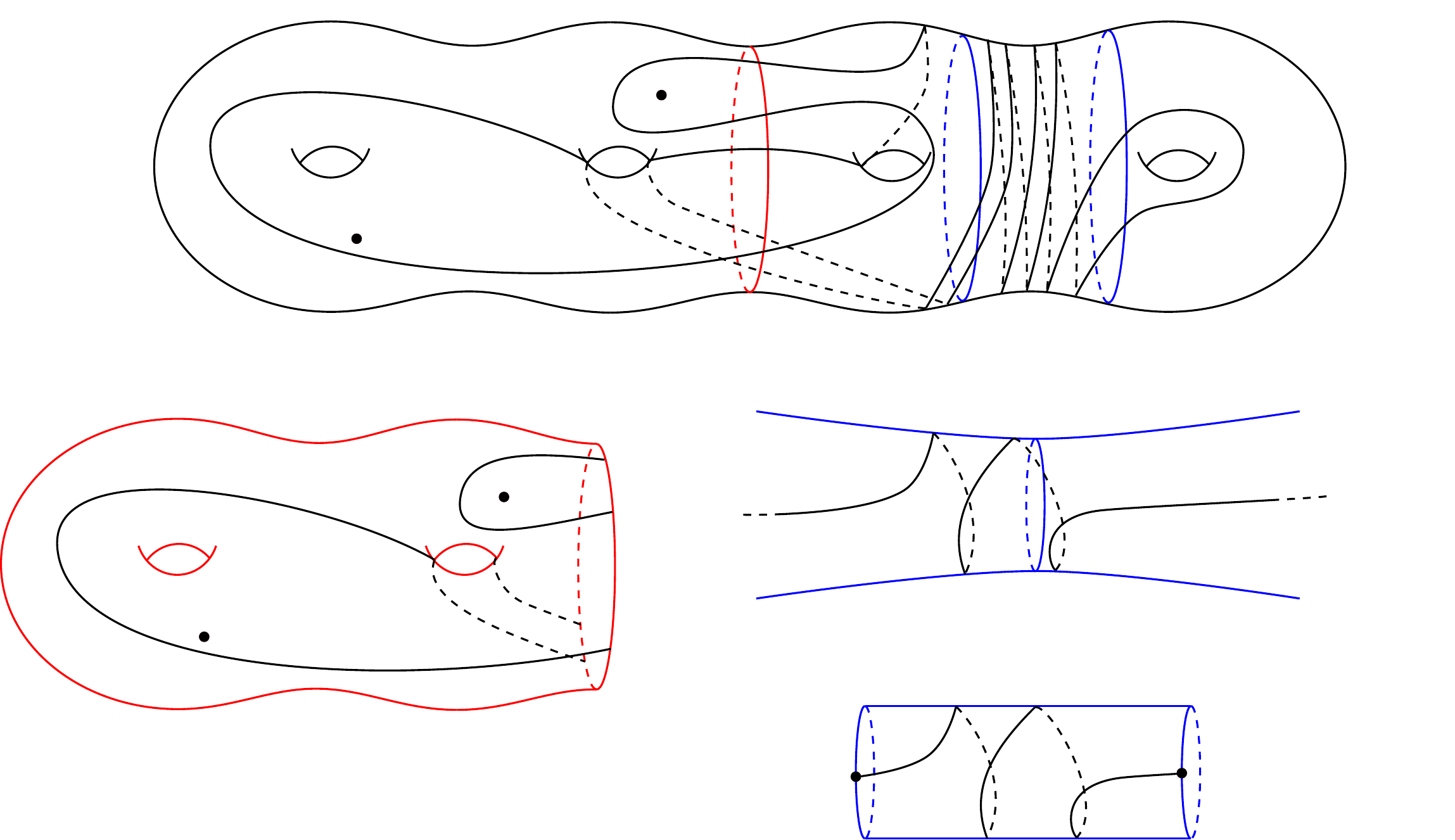
\caption{The subsurface projection of $\gamma$ to $Y$ consists of three pairwise disjoint essential arcs. The projection of $\gamma$ to $Z$ is obtained first by lifting $\gamma$, via the covering map $p_Z$, to the annular cover associated to $Z$, and then compactifying to obtain an essential arc.}
\label{fig:subsurface-proj-2}
\end{figure}


\section{Curves that intersect exactly once}
\label{section:one-intersection}

Let $G$ be a graph and $v$ a vertex of $G$. From now on, we will use $V(G)$ to refer to the vertex set of $G$. We define the \textit{link of a vertex} $v$, $\link(v)$, to be the induced subgraph of $G$ containing the set of all vertices adjacent to $v$. Note that $v \notin \link(v).$

Let $u,v$ be adjacent vertices in a graph. We will denote the edge between $u$ and $v$ by $(u,v)$. Define the \textit{link of an edge} $(u,v)$ to be 
\[L(u, v) := \link(u) \cap \link(v).\]

In this section, we focus on the graph $\mathcal{N}_1(S)$, which, when $S$ is closed, agrees with $\mathcal{SC}(S)$. There are two types of edges in $\mathcal{N}_1(S)$: we call edges that connect vertices admitting disjoint representatives \textit{0-edges} and those that minimally intersect once \textit{1-edges}. In order prove that an automorphism of $\N_1(S)$ induces an automorphism of $\mathcal C(S)$ we will give a graph theoretic criterion to distinguish between $0$-edges and $1$-edges.

The arguments in this section will be used to prove Theorem \ref{Schaller} and will also be useful for the proof of Theorem \ref{non-separating}. In particular, the first step of the proof of Theorem \ref{non-separating} is showing the following.

\begin{proposition} \label{prop:diam-4}
The diameter of a link of an edge in $\mathcal N_1(S)$ equals $4$ if and only if $e$ is a $1$-edge. 
\end{proposition}

This section is devoted to proving one direction of this statement, which we record as Lemma \ref{lemma:diam-1-edge}. We will prove the other direction in Section \ref{section:disjoint}.

\begin{lemma} \label{lemma:diam-1-edge}
The diameter of the link of a $1$-edge in $\mathcal N_1(S)$ is $4$.
\end{lemma}

The proof of Lemma \ref{lemma:diam-1-edge} will break into two pieces, Lemmas \ref{leq4} and \ref{4}, which we will prove in Subsections \ref{sub:leq4} and \ref{sub:geq4}, respectively. We begin by establishing some helpful language and notation.
 


Let $u,v \in V(\NS)$ such that $(u,v)$ is a $1$-edge. Let $S_1 = N (u \cup v)$ be the regular neighborhood of $u \cup v \subset S$, which is homeomorphic to a torus with one boundary component. 

Recall that the set of isotopy classes of essential simple closed curves on a torus with one boundary component is in bijection with the set $\mathbb Q \cup \left\{ \frac{1}{0} \right\}$, which we will call \emph{slopes}. In particular, the meridian curve is associated with $0/1$ and the longitude is associated with $1/0$. Moreover, the boundary-slide isotopy classes of essential simple arcs on a torus with one boundary component are in one-to-one correspondence with the isotopy classes of essential simple closed curves on the torus. We thus may refer to curves or arcs on $S_1$ by their associated slopes, which are of the form $p/q$, where $p, q$ are coprime integers. Note that an essential simple closed curve or simple arc with slope $p/q$ intersects the meridian $|p|$ times and the longitude $|q|$ times in minimal position. 
  
Up to a change of coordinates if necessary, we may assume that $u$ and $v$ are the $0/1$ and $1/0$ curves, respectively, as shown by the green curves in Figure \ref{figure:one-two-component}. We will also denote the $1/1$ and $-1/1$ curves in $S_1$ by $\gamma^+$ and $\gamma^-$. Observe that if $\alpha \subset S$ is a simple closed curve with $i(\alpha,u), i(\alpha,v) \leq 1$, then $\alpha \cap S_1$ has at most two nontrivial components: each  component of $\alpha \cap S_1$ is an essential arc in $S_1$, and any essential arc in $S_1$ must intersect the $0/1$ curve or the $1/0$ curve at least once. 

By a \textit{1-curve}, we will mean any curve $\alpha$ so that $i(\alpha, u)= i(\alpha,v)=1$ and so that $\alpha \cap S_1$ has a single component. If $\alpha$ is a $1$-curve, then $i(\alpha, u) = i(\alpha, v) = 1$ and $\alpha \cap S_1$ must be a $1/1$ arc or $-1/1$ arc. One of these two possible configurations is shown on the top of Figure \ref{figure:one-two-component}. Similarly, a \textit{2-curve} is any curve intersecting each of $u$ and $v$ exactly once, and so that its intersection with $S_1$ has two components. An illustration of a $2$-curve is shown on the bottom of Figure \ref{figure:one-two-component}.

\begin{figure}[htb!]
\centering
\begin{minipage}{0.8\textwidth}
    \resizebox{\textwidth}{!}{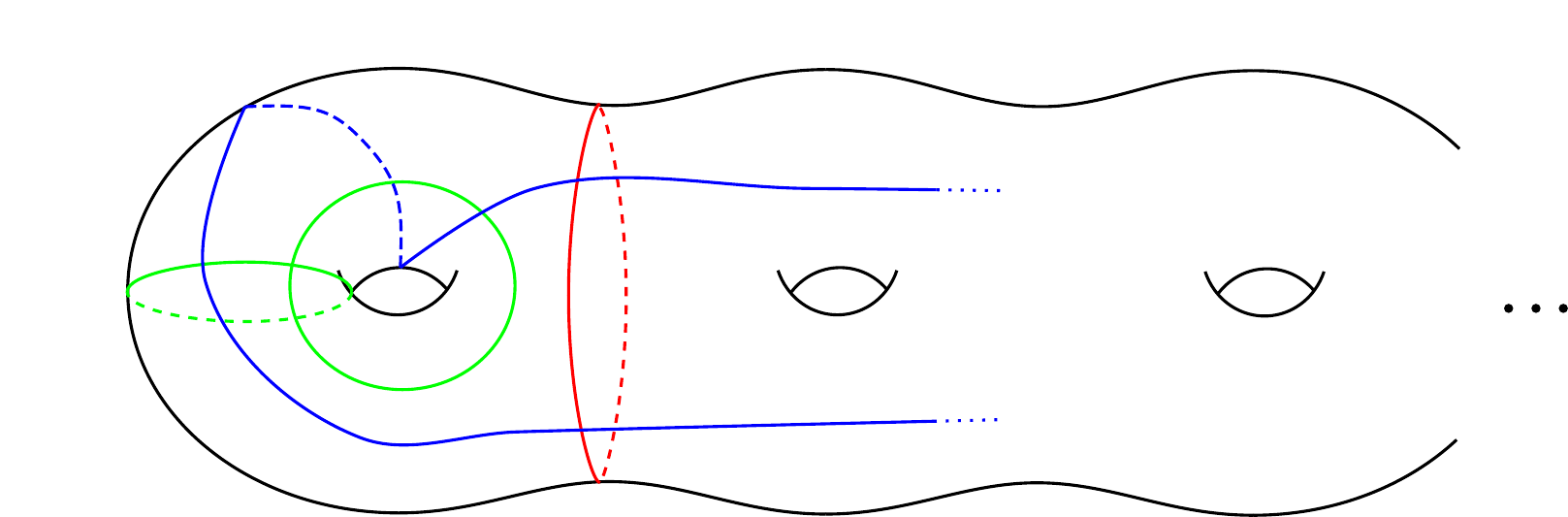}
\end{minipage}

\begin{minipage}{0.8\textwidth}
    \resizebox{\textwidth}{!}{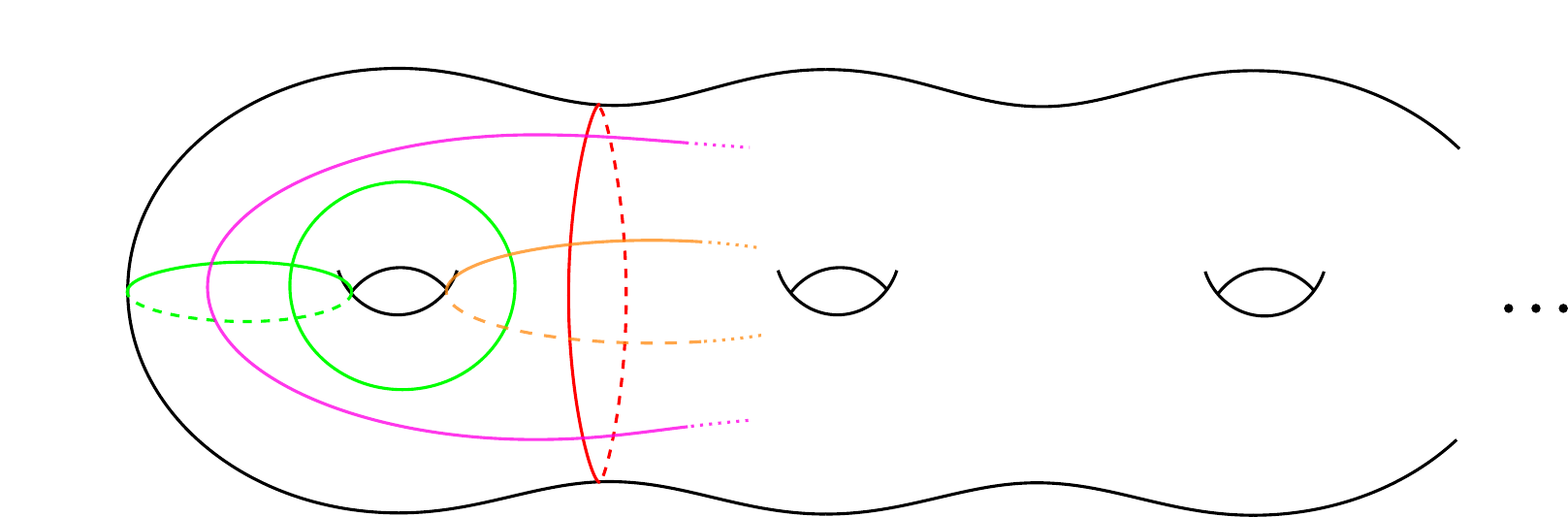}
\end{minipage}
\caption{On the top: a $1$-curve $\alpha$ in $S_1$. On the bottom: A $2$-curve $\alpha$ in $S_1.$}
\label{figure:one-two-component}
\end{figure}

\subsection{The link of a $1$-edge has diameter at most $4$} 
\label{sub:leq4}

We are now prepared to show that the link of a $1$-edge has diamater at most $4$.

\begin{lemma} \label{leq4}
If $(u,v)$ is a $1$-edge in $\N_1(S)$, then $\diam(L(u,v))\leq 4.$
\end{lemma}

\begin{proof} By assumption, $i (u,v) = 1$.  Let $S_2=\overline{S\setminus S_1}$. Up to a change of coordinates, we assume that $u$ and $v$ are the $0/1$ and $1/0$ curves in $S_1$, respectively.  


Let $\alpha \in L(u,v)$. Up to a homeomorphism of $S_1$ exchanging $u$ and $v$, the possible configurations of $\alpha$ relative to $u$ and $v$ are as follows:

\begin{enumerate} 
\item $i(u,\alpha)=i(v,\alpha)=1$
\item $i(u,\alpha) = i(v,\alpha)=0$
\item $i(u,\alpha)=0$ and $i(v,\alpha)=1.$
\end{enumerate}

We will refer to a curve $\alpha \in L(u,v)$ as being of either type (1), (2), or (3) depending on which of the above three holds. Recall that $\gamma^+$ and $\gamma^-$ are the $1/1$ and $-1/1$ curves on $S_1$, respectively.

\begin{claim} If $\alpha \in L(u,v)$ is type (1), then $d_L(\alpha, \{\gamma^+, \gamma^-\}) \leq 2$.  \end{claim}

\begin{proof} If $\alpha \subset S_1$, it must be either $\gamma^{+}$ or $\gamma^{-}$ and we are done.   

If $\alpha \not \subset S_1$, then $\alpha$ intersects $S_1$ in either one or two connected components, i.e., $\alpha$ is either a $1$- or $2$-curve. If $\alpha \not \subset S_1$ is a $1$-curve, then $\alpha \cap S_1$ must be either a $1/1$ or a $-1/1$ arc and therefore we have either $d_{L}(\alpha, \gamma^{+}) = 1$ or $d_{L}(\alpha, \gamma^{-}) = 1$. 

Otherwise, $\alpha \not \subset S_1$ is a $2$-curve and hence $\alpha$ has two connected components in $S_1$ and they can be arranged as in Figure \ref{figure:one-two-component}. Note that this implies $\alpha \cap S_2$ also has two disjoint connected components.

Observe that the endpoints of the two components of $\alpha\cap S_2$ are unlinked on $\partial S_2$, i.e, if one reads off the cyclic order of these endpoints on $\partial S_2$ clockwise or counter-clockwise, then the two endpoints of each component are adjacent to each other (see the two blue arcs in $S_2$ in Figure \ref{betaconstruction}).

We can now construct a curve $\beta$ as follows: starting at a point of intersection $\alpha \cap \partial S_1$, follow along one component of $\alpha \cap S_1$. Upon arriving back at $\partial S_1$, continue following along $\alpha$ through one component of $\alpha \cap S_2$. When arriving back at $\partial S_2$, the combinatorics of the points $\alpha \cap \partial S_2$ described in the previous paragraph implies that there is a choice of arc along $\partial S_2$ that ends where we began and so that the resulting closed curve is simple and intersects $\alpha$ exactly once. In particular, $\beta$ is non-separating; it also intersects each of $\gamma^+$ and $\gamma^-$ exactly once and thus $d_{L}(\alpha, \gamma^{+}), d_{L}(\alpha, \gamma^-) \leq 2$ (see Figure \ref{betaconstruction}).\end{proof}

\begin{figure}[htb!]
\centering
\begin{minipage}{0.4\textwidth}
    \resizebox{\textwidth}{!}{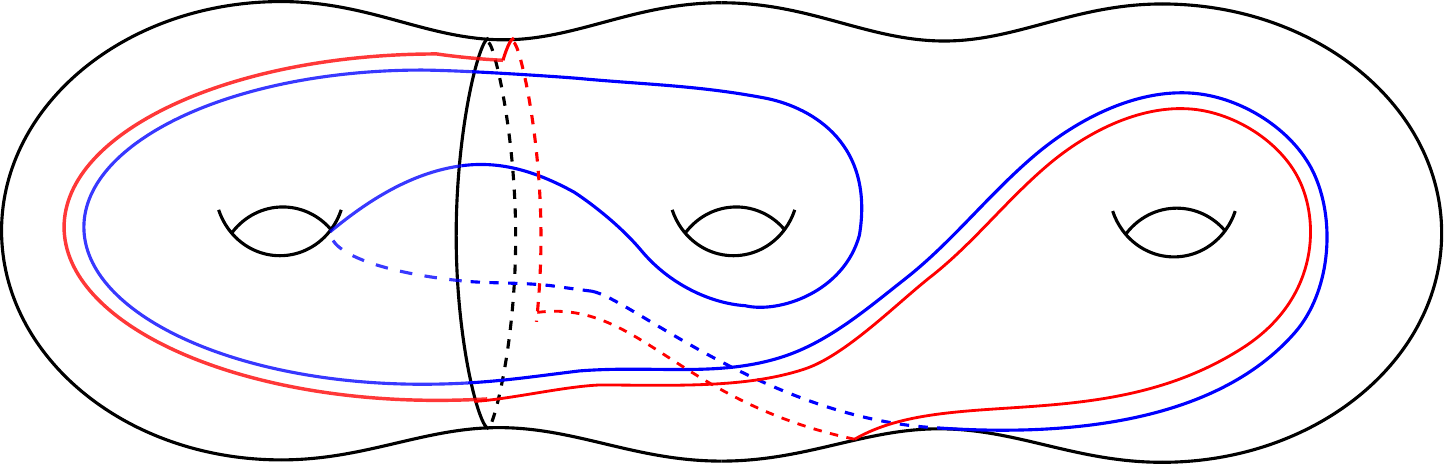}
\end{minipage}
\hspace{.5in}
\begin{minipage}{0.4\textwidth}
    \resizebox{\textwidth}{!}{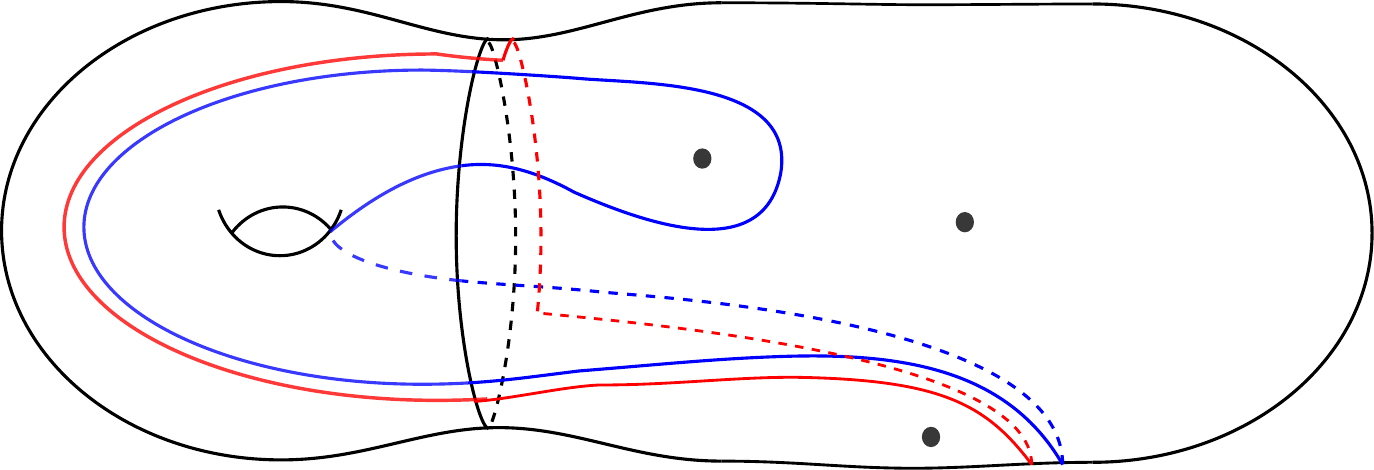}
\end{minipage}
\caption{The constructed curve $\beta$ is shown in red.}
\label{betaconstruction}
\end{figure}

\begin{claim} If $\alpha \in L(u,v)$ is of type (2), then  $d(\alpha, \gamma^+) = d(\alpha, \gamma^-) = 1$ \end{claim}

\begin{proof} In this case, $\alpha \subset S_2$. Thus, $d(\alpha, \gamma^+) = d(\alpha, \gamma^-) = 1$ and we are done. \end{proof}

\begin{claim} If $\alpha \in L(u,v)$ is of type (3), then $d_L(\alpha, \gamma^+) = d_L(\alpha, \gamma^-) = 1.$\end{claim} 

\begin{proof} In this case, $\alpha$ intersects both $\gamma^+$ and $\gamma^-$ exactly once. Hence, we have that $d_L(\alpha, \gamma^+) = d_L(\alpha, \gamma^-) = 1,$ as desired. \end{proof}

Now, let $\alpha, \beta \in L(u,v)$. Suppose both $\alpha$ and $\beta$ are two type (1) curves. If both $\alpha$ and $\beta$ are $2$-curves, then by  Case $1$, they are distance at most $2$ from both $\gamma^+$ and $\gamma^-$ and thus $d_L(\alpha, \beta) \leq 4$. If $\alpha$ is a $1$- curve and $\beta$ is a $2$-curve, then again we have that $d_{L}(\alpha, \beta) \leq 4$ because $d_L(\beta, \gamma^+), d_L(\beta, \gamma^-) \leq 2$ and thus $d_L(\alpha, \{\gamma^+, \gamma^-\}) = 1$. If both $\alpha$ and $\beta$ are $1$-curves, both are distance at most $1$ from either $\gamma^+$ or $\gamma^-$ and therefore are at a distance of at most $4$ from one another, since $d_L(\gamma^+, \gamma^-) = 2$. 

If $\alpha$ is a $1$-curve and $\beta \subset S_1$, then $d_L(\alpha, \beta) \leq 3$, since $\beta \in \{\gamma^+, \gamma^-\}$ and either $d_L(\alpha, \gamma^+) =1$ or $d_L(\alpha, \gamma^-) = 1$. Finally, if $\alpha$ is a 2-curve and $\beta \subset S_1$, then $d_L(\alpha, \beta) \leq 2$, since $\omega \in \{ \gamma^+, \gamma^-\}$ and $d_L(\alpha, \gamma^+), d_L(\alpha, \gamma^-) \leq 2$. 

It remains to consider curves of type (2) and (3). A type (2) curve intersects neither $u$ nor $v$, and therefore is disjoint from $S_1$. Thus it is distance $1$ from both $\gamma^{+}$ and $\gamma^{-}$, and therefore at distance at most $3$ from any type (1) curve and at a distance of at most $2$ from any other type (2) curve. A type (3) curve must also be distance $1$ from both $\gamma^{+}$ and $\gamma^{-}$ and so is also distance at most $3$ from any type (1) curve, and distance at most $2$ from any type (2) or type (3) curve. \end{proof}

\subsection{The link of a 1-edge has diameter at least 4}
\label{sub:geq4}

We end this section by showing that the diameter of the link of a 1-edge is at least 4.

\begin{lemma}\label{4}
Suppose $(g,p) \neq (1,2)$. If $(u,v)$ is a $1$-edge in $\N_1(S)$, then $\diam(L(u,v)) \geq 4.$
\end{lemma}

\begin{proof}
Let $u,$ $v,$ $S_1,$ and $S_2$ be as in Lemma \ref{leq4}. Again let $\gamma^{+}$ be the $1/1$ curve and $\gamma^{-}$ the $-1/1$ curve in $S_1$. It suffices to show that there exist two 2-curves, $\alpha$ and $\beta$, whose shortest connecting path passes through $\gamma^{+}$ or $\gamma^{-}$.

Since $S_2$ is not a $3$-holed sphere, the diameter of $\mathcal{AC}(S_2)$ is infinite. Hence there exist arcs $\delta, \eta  \in \mathcal{AC}(S_2)$ such that $d_{\mathcal{AC}(S_2)}(\delta, \eta) \geq 23$. There also exists arcs $\delta', \eta' \in \mathcal{AC}(S_2)$ such that $d_{\mathcal{AC}(S_2)}(\delta, \delta')=1=d_{\mathcal{AC}(S_2)}(\eta, \eta'),$ and so that the endpoints of $\delta$ and $\delta'$ (respectively, $\eta$ and $\eta'$) do not link along $\partial S_2.$ Note that $d_{\mathcal{AC}(S_2)}(\delta',\eta') \geq 21.$ 

Now, construct a 2-curve $\alpha$ from $\delta$ and $\delta'$ as follows: choose an endpoint from each of $\delta$ and $\delta'$ that are not consecutive on $\partial S_2$. Connect these through $S_1$ via the $1/0$ arc and connect the remaining two endpoints with the $0/1$ arc.  Construct $\beta$ in a similar manner from $\eta$ and $\eta'.$ Note that $\pi_{S_2}(\alpha) = \{ \delta, \delta'\}$ and $\pi_{S_2}(\beta) = \{ \eta, \eta'\}$. 

Let $\rho=\{\alpha, v_0,\ldots,v_n,\beta\}$ be any shortest path in $L(u,v)$ connecting $\alpha$ to $\beta.$ If each $v_i$ has a non-trivial projection to $S_2,$ by choosing one vertex from each of $\pi_{S_2}(v_i)$ one obtains a path of length $m \leq n + 2$ in $\mathcal{AC}_1(S_2)$ from  $\delta'$ to $\eta'.$ Note that $m$ may be strictly less than $n + 2$ if there exist $v_i$ and $v_j$ with the same projection to $S_2$. This, in turn, yields a path of length at most $3m$ in $\mathcal{AC}(S_2)$ between $\delta'$ and $\eta'$. The factor of $3$ comes from the quasi-isometry established in Proposition \ref{qi}. Since $d_{\mathcal{AC}(S_2)}(\delta', \eta') \geq 21$, then $3(n + 2) \geq 3m\geq 21,$ and so the length of $\rho$ is greater than or equal to 5, a contradiction to Lemma \ref{leq4}. 

This implies there exists some $v_i$ that projects trivially to $S_2$, and therefore $v_i \in \{\gamma^+, \gamma^- \}$. Therefore, we have that the length of $\rho$ is at least 4.\end{proof}


\section{Disjoint curves}
\label{section:disjoint}

In this section we will complete the proof of Proposition \ref{prop:diam-4}, and use it to prove Theorem \ref{non-separating}, by proving the following lemma. 

\begin{lemma} \label{lemma:diam-link-4}
If $(u,v)$ is a $0$-edge in $\mathcal N_1(S)$, then $\diam L(u,v) \neq 4$.
\end{lemma}

This lemma is proved by considering two cases: when the two vertices $u,v$ of the edge $(u,v)$ are jointly separating, as considered in Lemma~\ref{proposition:diam-3}, and when they are jointly non-separating, as considered in Lemma~\ref{proposition:diam-inf}.

\subsection{The jointly separating case.} In this section we will prove Lemma~\ref{lemma:diam-link-4} for a $0$-edge $(u,v)$ when $u$ and $v$ are jointly separating.

\begin{lemma}\label{proposition:diam-3}
Suppose $u, v \in \mathcal{N}_1(S)$ are disjoint curves that are jointly separating. If both components of $S \setminus u \cup v$ contain non-separating curves of $S$, then $L(u, v)$ has diameter at most $3$. Otherwise, $L(u,v)$ has infinite diameter.
\end{lemma}

\begin{proof}
Denote the two connected components of $S \setminus (u \cup v)$ by $S_1$ and $S_2$. Then any curve in $L(u,v)$ is contained in either $S_1$ or $S_2$ or has nontrivial intersection with both subsurfaces. Let $\alpha, \beta \in L(u,v)$. Assume first that both $S_1$ and $S_2$ contain non-separating curves of $S$.

If $\alpha$ and $\beta$ are contained in the same component of $S\setminus (u \cup v)$, say $S_1$, then $d_L(\alpha, \beta) \leq 2$ since by assumption there is a non-separating curve contained entirely in $S_2$. On the other hand, if $\alpha$ and $\beta$ are contained in different components of $S \setminus (u \cup v)$, then they are disjoint and thus $d_L(\alpha, \beta) = 1$.

Without loss of generality, suppose $\alpha$ is contained in $S_1$ and $\beta$ has non-trivial intersection with both $S_1$ and $S_2$. Then $\beta \cap S_2$ is a non-separating arc on $S_2$ with endpoints on distinct boundary components. Letting $\omega$ denote a non-separating curve of $S$ contained in $S_2$ (which exists by assumption), by the classification of surfaces there is a homeomorphism $f$ of $S_2$ fixing $u$ and $v$ which sends $\beta \cap S_2$ to an arc crossing $\omega$ at most once. If $\omega$ does not separate $u$ from $v$, we can choose $f$ so that $f(\beta \cap S_2)$ is disjoint from $\omega$. Otherwise, we can choose $f$ so that $f(\beta \cap S_2)$ crosses $\omega$ exactly once. Then $f^{-1}(\omega)$ is a non-separating curve on $S$ that is at an $L(u,v)$-distance of $1$ from both $\alpha$ and $\beta$. 

Finally, if both $\alpha$ and $\beta$ have non-trivial intersections with both components of $S \setminus (u \cap v)$, then each of them intersects $S_1$ and $S_2$ in a non-separating arc. Then by the same argument used in the previous paragraph there are non-separating curves $\omega_1 \subset S_1, \omega_2 \subset S_2$ so that $d_L(\alpha, \omega_1)= d_L(\beta, \omega_2) = 1$. Thus, $d_L(\alpha, \beta) \leq 3$, since $i(\omega_1, \omega_2) = 0$.   

It remains to consider the case where either $S_1$ or $S_2$ contains no non-separating curve of $S$. Since $S$ is not a twice punctured torus, at least one component of $S \setminus (u \cup v)$, say $S_1$, is not a $3$-holed sphere. Now for $n \in \mathbb{N}, n>2$, by Proposition \ref{qi}, choose arcs $\omega_1, \omega_n$ with the same endpoints, each with one endpoint on $u$ and the other on $v$ so that 
\[ d_{\mathcal{AC}_1(S_1)}(\omega_1, \omega_n) \geq n. \]

Then choose an arc $\lambda \subset S_2$ with the same endpoints as $\omega_1, \omega_n$ so that the concatenation of $\lambda$ with $\omega_1$ and with $\omega_n$ yields non-separating curves $\eta_1, \eta_n$ (these will be non-separating because they both cross each of $u$ and $v$ exactly once). As $S_2$ contains no non-separating curves, any curve in $L(u,v)$ must project non-trivially to $S_1$, and so a path in $L(u,v)$ from $\eta_1$ to $\eta_n$ gives rise to a path in $\mathcal{AC}_1(S_1)$ of length on the order of $n$. It follows that $d_{L}(\eta_1, \eta_n)$ is at least on the order of $n$. As $n$ was arbitrary, the diameter of $L(u,v)$ must be infinite. \end{proof}

We record the following remark as it will be useful in the proof of Conjecture \ref{SchallerC}:

\begin{remark} \label{separating curves}
If $u$ is a separating curve on $S$ that bounds a $3$-holed sphere on one side and $v$ is another curve representing a vertex of $\mathcal{SC}(S)$ so that $i(u,v)=0$, then the proof of Lemma \ref{proposition:diam-3} implies that the diameter of $L(u,v)$ in $\mathcal{SC}(S)$ is infinite. Indeed, in this case $S_1$ consists of either a single $3$-holed sphere or a disjoint union of two $3$-holed spheres. Thus, there is no curve in $L(u,v)$ that does not project to $S \setminus S_1$ which is the only assumption used in the last two paragraphs of the above proof. 
\end{remark}

\begin{remark} \label{genus 2}
If the genus of $S$ is $1$, then if $u,v$ are disjoint non-separating curves, they must be jointly separating. Indeed, cutting along $u$ (or $v$) produces a planar surface. Therefore, in the next subsection it will suffice to assume that the genus of $S$ is at least $2$. 
\end{remark}

\subsection{The jointly non-separating case}
Let $u,v$ be disjoint curves in $S$ such that $u$ and $v$ are jointly non-separating. In this case, we show that the diameter of $L(u,v)$ in $\N _1 (S)$ is infinite. This concludes the proof of Lemma~\ref{lemma:diam-link-4}, which in turn completes the proof of Proposition~\ref{prop:diam-4}.

\begin{lemma}\label{proposition:diam-inf}
Let $g \ge 2$. If $u, v \in \N_1(S)$ are disjoint curves that are jointly non-separating, then $L(u, v)$ has infinite diameter.
\end{lemma}

\begin{proof}
Consider $S'= S\setminus (u \cup v).$ Since the genus of $S$ is at least $2$, $S'$ is not a $3$-holed sphere and so by Proposition \ref{qi}, $\mathcal{AC}_1(S')$ has infinite diameter. Let $\lambda_u, \lambda_v$ denote simple closed curves on $S$ so that $\lambda_u$ (resp. $\lambda_v$) crosses $u$ (resp. $v$) exactly once. 
 
Choose a pseudo-Anosov mapping class $\phi \in \mbox{Mod}(S')$. By Remark \ref{loxodromic}, given $n \in \mathbb{N}$ there exists an $N \geq 1$ so that \[ d_{\mathcal{AC}_1(S')}(\pi_{S'}\lambda_u, \phi^{N}(\pi_{S'}\lambda_v)) > n.\]
 
Let $\lambda_v^{\phi^N}$ denote the simple closed curve on $S$ obtained by turning $\phi^{N}(\pi_{S'} \lambda_v)$ into a simple closed curve by including its intersection point with $u$. Then, since any essential simple closed curve on $S$ projects non-trivially to $S'$, an $L(u,v)$- path from $\lambda_u$ to $\lambda_v^{\phi^{N}}$ gives rise to a path in $\mathcal{AC}_1(S')$ of comparable length, between their projections. Thus, we have produced vertices in $L(u,v)$ which are arbitrarily far apart in $L(u,v)$ and so $L(u,v)$ has infinite diameter, as desired. \end{proof}

\subsection{The proof of Theorem \ref{non-separating}} 

We are now in a position to prove Theorem \ref{non-separating}.

\vspace{2 mm}

\noindent \textbf{Theorem \ref{non-separating}.}
\textit{Suppose that $g \geq 1$ and that $(g,p) \neq (1,2)$. Then the natural map 
\[ \mbox{Mod}^{\pm}(S) \rightarrow \mbox{Aut}(\mathcal{N}_1(S))\]
is an isomorphism for $g \neq (2,0)$ and a surjection with kernel $\mathbb{Z}/2\mathbb{Z}$ otherwise.}

\vspace{2 mm}

\begin{proof}
Let $f\in \Aut(\N_1(S))$. By Proposition~\ref{prop:diam-4}, we can conclude that graph automorphisms of $\N_1(S)$ preserve edge types. Thus, $f$ induces a graph automorphism of $\mathcal{N}(S)$ and of $\mathcal{G}(S)$ by restriction. Since the vertex sets of $\mathcal{N}(S)$ , $\mathcal{N}_1(S)$, and $\mathcal{G}(S)$ are the same,  $\mbox{Aut}(\mathcal{N}_1(S))$ injects into $\mbox{Aut}(\mathcal{G}(S))$.  Hence, $\Aut(\N_1(S)) \leq \Aut(\mathcal{G}(S))$. 

Then by Theorem A of \cite{Schaller}, $f$ is induced by a mapping class of $S$. Conversely when $(g,p) \neq (2,0)$ every mapping class gives rise to a distinct automorphism of $\mathcal{N}_1(S)$. When $(g,p)= (2,0)$, mapping classes give rise to distinct graph automorphisms exactly when they reside in distinct cosets of the centralizer of $\mbox{Mod}^{\pm}(S)$. \end{proof}

\begin{remark} When $g \geq 2$, Theorem \ref{non-separating} can also be proved by appealing to Theorems $1.1$ and $1.2$ of \cite{Irmak} together with Lemmas \ref{leq4}, \ref{4}, and \ref{proposition:diam-3}, and Proposition \ref{proposition:diam-inf}. 
\end{remark}

\section{Automorphisms of the Systolic Complex}
\label{section:systolic}

The purpose of this section is to prove Theorem \ref{Schaller} which we do in the final subsection. Recall that when a surface $S$ has multiple punctures, the graph $\mathcal{SC}(S)$ includes vertices representing separating curves that bound a $3$-holed sphere on one side. Such a vertex is connected to another vertex $v$ by an edge whenever the corresponding curves are disjoint or intersect exactly twice. 

Two main tools in the proof of Theorem~\ref{Schaller} are the following propositions. The first proposition characterizes the diameter of the link of a 2-edge $(u,v)$ when both $u$ and $v$ are separating.

\begin{proposition}\label{lemma:summarydiam=4}
Suppose $g\geq 1$ and $(g,p)\neq (0,5), (1,3).$ If $(u,v)$ is a 2-edge in $\mathcal{SC}(S)$ with $u$ and $v$ both separating then $\diam(L(u,v))=~4.$
\end{proposition}

The second proposition characterizes the diameter of the link of a 2-edges $(u,v)$ when exactly one of $u$ and $v$ is separating.

\begin{proposition}\label{prop:summarydiam=3}
Suppose $g\geq 1$ and $(g,p)\neq (0,5), (1,3).$ If $(u,v)$ is a 2-edge in $\mathcal{SC}(S)$ with $u$ nonseparating and $v$ separating, then $\diam(L(u,v)) = 3.$ Furthermore, whenever $d_{L}(\alpha, \beta)=3$, there exists a path of length $3$ in $L(u,v)$ from $\alpha$ to $\beta$ that passes through $\left\{b_1, b_2 \right\}$.
\end{proposition}

The proof of Proposition~\ref{lemma:summarydiam=4} will break into two pieces, Lemma~\ref{2-edge upper} and Lemma~\ref{2-edge}. Likewise, the proof of Proposition~\ref{prop:summarydiam=3} will break into Lemmas~\ref{lemma:leq3} and \ref{2-edge-3}. 

\subsection{Upper bounds on the diameter of the link of a 2-edge in $\mathcal{SC}(S)$} \label{sec:2-edge upper}

Let $(u,v)$ be an edge in $\mathcal{SC}(S)$ with $i(u,v)=2$; we refer to such an edge as a $2$-\textit{edge}. In this case, at least one of $u,v$ is separating and the subsurface obtained by thickening the union of $u$ and $v$ is necessarily a $4$-holed sphere. As in the case of the punctured torus, the simple closed curves on a $4$-holed sphere are also naturally parameterized by slopes in $\mathbb{Q} \cup \infty$, and without loss of generality we identify $u,v$ with the $1/0$ and $0/1$ curves. 

\begin{lemma} \label{2-edge upper}
Suppose $(g,p) \neq (0,5)$. If $(u,v)$ is a $2$-edge in $\mathcal{SC}(S)$ with both $u$ and $v$ separating, then $\diam(L(u,v)) \leq 4.$
\end{lemma}

\begin{proof}
Let $S_1$ be the $4$-holed sphere that forms the regular neighborhood of $u \cup v$. Denote the $1/1$ and $-1/1$ curves in $S_1$ by $\gamma^+$ and $\gamma^-$, respectively. Note that $\gamma^{\pm}$ both intersect $u$ and $v$ twice.

Since both $u$ and $v$ are separating, three of the four boundary components of $S_1$ necessarily correspond to punctures of $S$ which implies that both $\gamma^{+}$ and $\gamma^{-}$ are separating and bound a $3$-holed sphere on one side. Therefore, both are elements of $L(u,v)$.

Now consider $\alpha \in L(u,v)$. Note that $|\alpha \cap S_1| \leq 2$. This is because every component of $\alpha \cap S_1$ has its endpoints on a single boundary component (the one that is not a puncture of $S$) and $\alpha$ intersects both $u$ and $v$ either $0$ or $2$ times, since $u$ and $v$ are both separating.

If $|\alpha \cap S_1| = 1$, then $\alpha$ will be at an $L(u,v)$-distance of $1$ from $\gamma^{\pm}$.

If $|\alpha \cap S_1| = 2$, then $\alpha \cap S_1$ must be (up to homeomorphism) as pictured in Figure \ref{fig:Lemma4-1-Fig1}. 

\begin{figure}[htb!]
\centering
\def\svgwidth{1.5in}
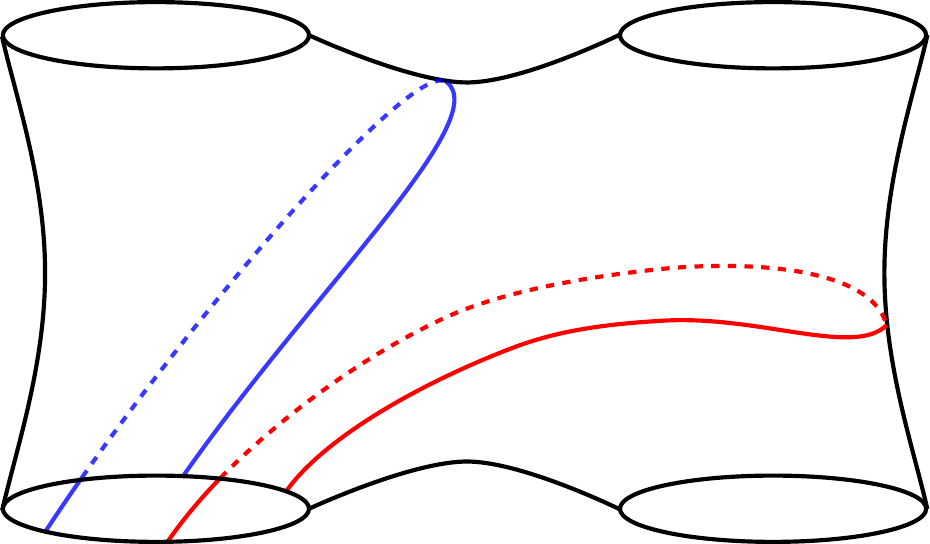
\caption{The intersection pattern of $\alpha \in L(u,v)$ with $S_1$ when $|\alpha \cap S_1| = 2$ and both $u$ and $v$ are separating.}
\label{fig:Lemma4-1-Fig1}
\end{figure}

First suppose that $\alpha$ is separating. Take one of the arcs of $\alpha$ contained in $S_2$. This can be concatenated with the blue arc shown in Figure \ref{fig:Lemma4-1-Fig1} to obtain a simple closed curve $\eta$, which is disjoint from $u$ and intersects each of $v$, $\alpha$, and (without loss of generality) $\gamma^{+}$ twice. Based on the topology of $S_2$ and the arcs of $\alpha \cap S_2$, $\eta$ can be chosen to be either non-separating, or a separating curve that bounds a $3$-holed sphere on one side. 

Thus, $\eta \in L(u,v)$, $d_L(\alpha, \eta) = 1$ and $d_L(\eta, \gamma^{+}) = 1$. The possible configurations for $\eta$ are shown in Figure~\ref{fig:section5fix}.

\begin{figure}
    \centering
    \def\svgwidth{3.5in}
    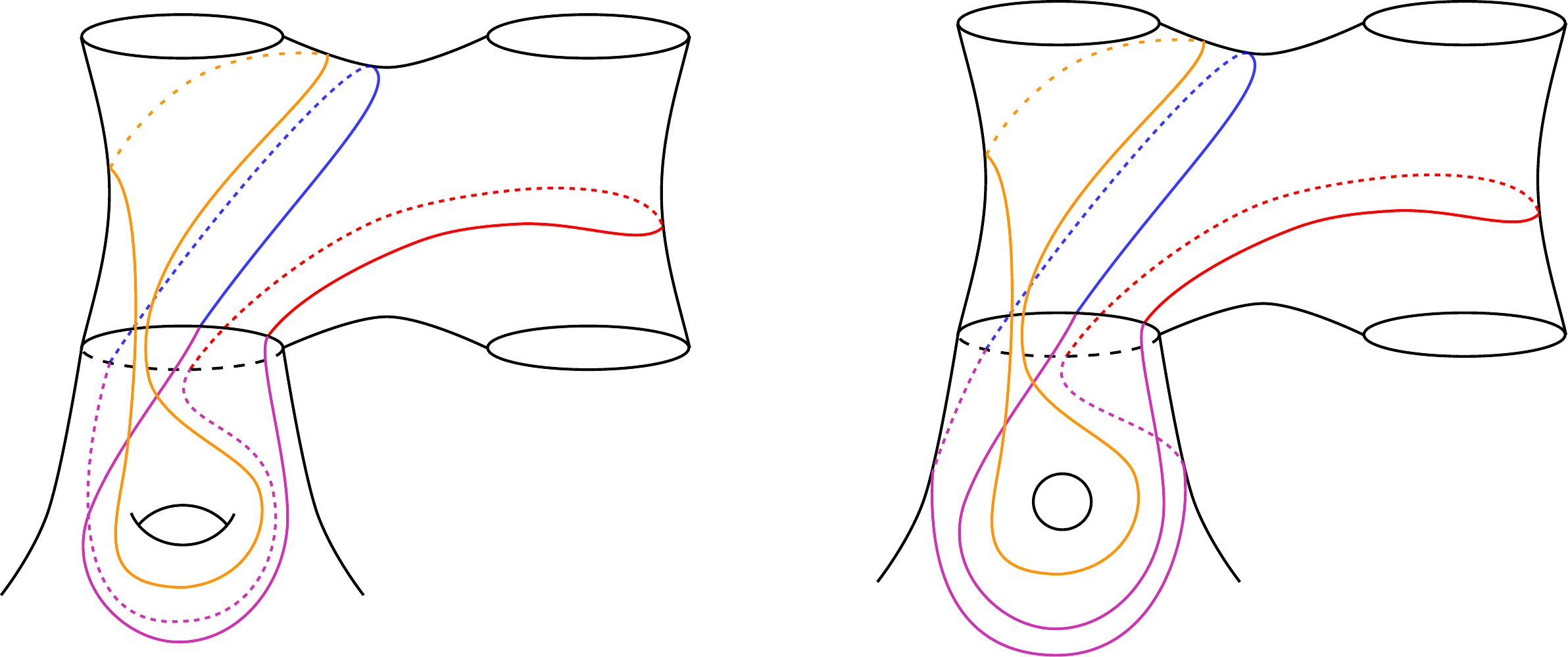
    \caption{The curve $\eta$ (drawn in orange) is either nonseparating (shown on the left) or separating (shown on the right) depending on the configuration of $\alpha \cap S_2$.}
    \label{fig:section5fix}
\end{figure}

Now suppose $\alpha$ is non-separating. Then we we claim that it is possible to build a simple closed curve $\eta$ contained entirely in $S_2$ which intersects $\alpha$ exactly once and is thus non-separating (and therefore in $L(u,v)$). Indeed, take an essential arc $\lambda$ in $S_2$ disjoint from $\alpha$ with endpoints on $b_1$ (the existence of $\lambda$ is guaranteed because $\alpha$ is non-separating). We can then concatenate $\lambda$ with a subarc of $b_1$ that intersects $\alpha$ exactly once. This concatenation gives us our desired curve $\eta$, illustrated in Figure~\ref{fig:final-eta}. It follows that $\alpha$ is at an $L(u,v)$-distance of at most $2$ from both $\gamma^+$ and $\gamma^-$.\end{proof}

\begin{figure}[htb!]
\centering
\def\svgwidth{1.5in}
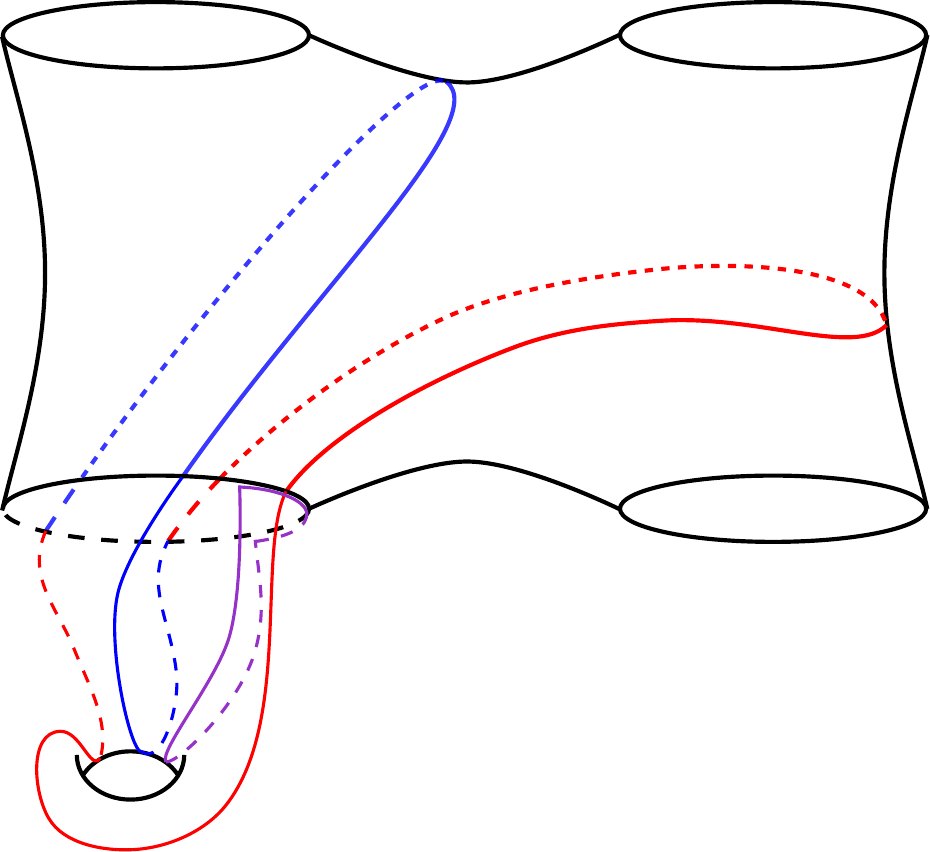
\caption{Constructing $\eta$ when $\alpha$ is non-separating.}
\label{fig:final-eta}
\end{figure}

\begin{lemma} \label{lemma:leq3}
Suppose $(g,p) \neq (0,5)$. If $(u,v)$ is a $2$-edge in $\mathcal{SC}(S)$ with $u$ nonseparating and $v$ separating, then $\diam(L(u,v)) \leq 3.$
\end{lemma}

\begin{proof} Let $\alpha, \beta \in L(u,v)$ with $u$ nonseparating and $v$ separating as shown in Figure~\ref{fig:labelled-u-v}. The boundary curves, $b_1$ and $b_2$, of $S_1$ will play a crucial role in each case of this proof. Note that $b_1$ and $b_2$ are necessarily nonseparating. If they were both separating, then $u$ would also be separating and if only one of them was separating, then the other would be a boundary component and again $u$ would be separating.

\begin{figure}[htb!] 
\centering
\def\svgwidth{2in}
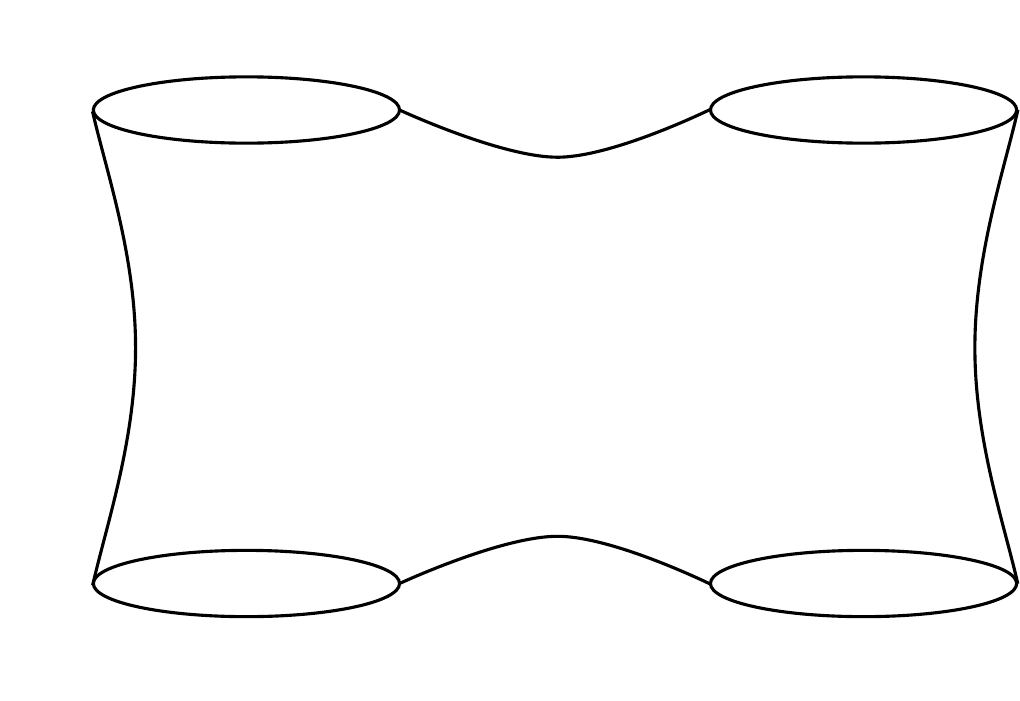
\caption{The regular neighborhood $S_1$ of $u$ and $v$ when $u$ is non-separating and $v$ is separating. }
\label{fig:labelled-u-v}
\end{figure}

We will make use of the following possible intersection patterns shown in Figure~\ref{fig:possibilities}. In each case, we will show that $\alpha$ has an $L(u,v)$- distance of $1$ from at least one of $\left\{b_{1}, b_{2} \right\}$. It is also possible for $\alpha$ to intersect $S_{1}$ in a sub-multi-arc of one of the multi-arcs pictured below, but this possibility is subsumed by the three arguments outlined below. 

\begin{figure}[htb!]
    \centering
    \def\svgwidth{5in}
    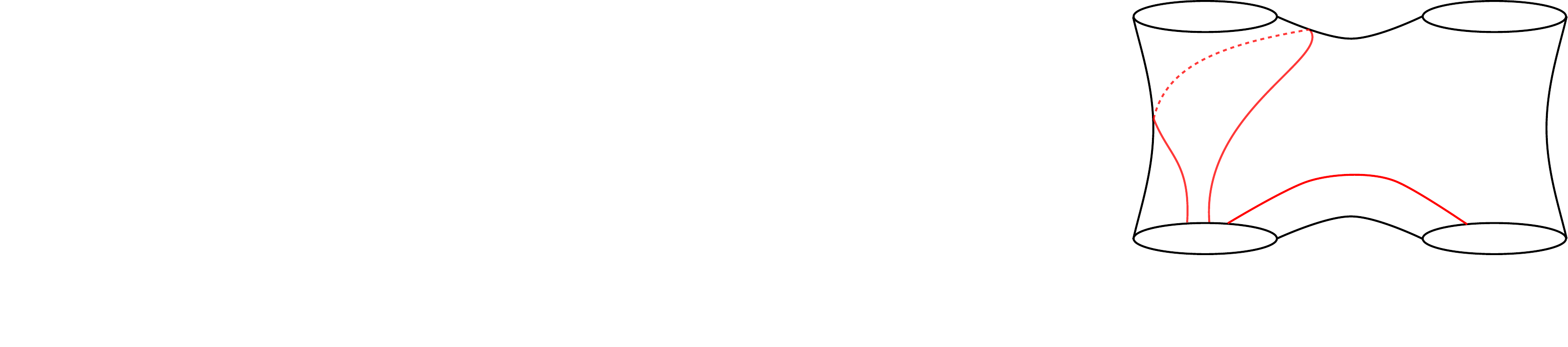
    \caption{Possible intersection patterns of a curve in $L(u,v)$ with $S_1$ when exactly one of $u$ and $v$ is separating. Note that it is also possible for $\alpha$ to intersect $S_1$ in a sub-multi-arc of one of the multi-arcs pictured above. The arguments outlined below will apply to this situation.}
    \label{fig:possibilities}
\end{figure}

\noindent {\bf Case 1:} Suppose $\alpha$ and $\beta$ are both separating. Since they are separating, they are both in either configuration $(i)$ or $(ii)$. This implies that there exists $i \in \{1, 2\}$ such that $\alpha$ intersects $b_i$ twice. The same is true for $\beta$, although not necessarily with the same value of $i$. 

Assume first that $d_L(\alpha, b_1) = d_L(\beta, b_1) = 1$. Then $d_L(\alpha, \beta) \leq 2$. 

Otherwise, $d_L(\alpha, b_1) = 1$ and $d_L(\beta, b_2) = 1$. Since $b_1$ and $b_2$ are disjoint, $d_L(b_1, b_2) = 1$. So $d_L(\alpha, \beta) \leq 3$.

\noindent {\bf Case 2:} Suppose $\alpha$ is separating and $\beta$ is nonseparating. Since $\beta$ is nonseparating it must be in configuration $(iii)$. So there exists $i \in \{1, 2\}$ such that $d_L(\beta, b_i) = 1$. Since $\alpha$ is separating it must be either in configuration $(i)$ or $(ii)$. By the same reasoning as in Case 1, $d_L(\alpha, b_i) = 1$ for some $i \in \{1, 2\}$. The argument then resolves in the same way as in Case 1.

\noindent{\bf Case 3:} Suppose $\alpha$ and $\beta$ are both nonseparating. Then $\alpha$ and $\beta$ must be in configuration $(iii)$. So there exists $i \in \{1, 2\}$ such that $d_L(\beta, b_i) = 1$. The same is true for $\beta$, although not necessarily with the same value of $i$. Once again, the argument resolves as in Case 1.
\end{proof}

\subsection{Lower bounds on the diameter of the link of a 2-edge in $\mathcal{SC}(S)$} \label{2-edge lower}

In this section we finish the proofs of Propositions~\ref{lemma:summarydiam=4} and \ref{prop:summarydiam=3} by proving the following two lemmas.

\begin{lemma} \label{2-edge}
Suppose $g \geq 1$ and that $(g,p) \neq (1,3)$. If $(u,v)$ is a $2$-edge in $\mathcal{SC}(S)$ with both $u$ and $v$ separating, then $\diam(L(u,v)) \geq 4.$
\end{lemma}

\begin{proof} 

Let $S_1$ denote the regular neighborhood of $u \cup v$. As mentioned in the proof of Lemma~\ref{2-edge upper}, $S_1$ must be a sphere with $4$ boundary components, exactly three of which are boundary components of $S$.

Since $(g,p) \neq (0,5)$ or $(1, 3)$, the complementary subsurface $S_{2}= S \setminus S_{1}$ is not a $3$-holed sphere, and thus the diameter of $\mathcal{AC}_1(S_{2})$ is infinite. As in the proof of Lemma \ref{4}, we will construct a pair of $2$-curves $\alpha$ and $\beta$ whose shortest path in $L(u,v)$ passes through $\{\gamma^+, \gamma^-\}$, both of which are in $L(u,v)$. Using the same notation as before, we will first specify the construction for $\alpha \cap S_2 = \{\eta, \eta' \}$. 

The assumptions in the statement of the proposition guarantee that $S_2$ contains a non-separating curve $c$ of $S$, and let $b$ be the essential boundary component of $S_1$ (that is, the boundary component of $S_1$ not corresponding to a boundary component of $S$). Let  $\lambda$ be an embedded arc connecting $b$ to $c$. We consider an arc $\eta \subset S_2$ with both endpoints on $b$, obtained by traveling along $\lambda$, then around $c$, and then back to $b$ along the inverse of $\lambda$. 

Note that $S_{2}$ has positive genus since we assumed that $S$ has genus at least $1$. This situation is illustrated in Figure \ref{fig:Section5-non-planar}. So on $S_{2}$ there must exist a second essential arc $\eta'$ disjoint from $\eta$ and whose endpoints link with those of $\eta$ along $b$. Furthermore, we can choose $\eta'$ so that it intersects $c$ exactly once. As shown in Figure \ref{fig:Section5-linking}, the existence of $\eta'$ is guaranteed so long as $S_{2}$ is not planar.

\begin{figure}[htb!]
\centering
\def\svgwidth{3in}
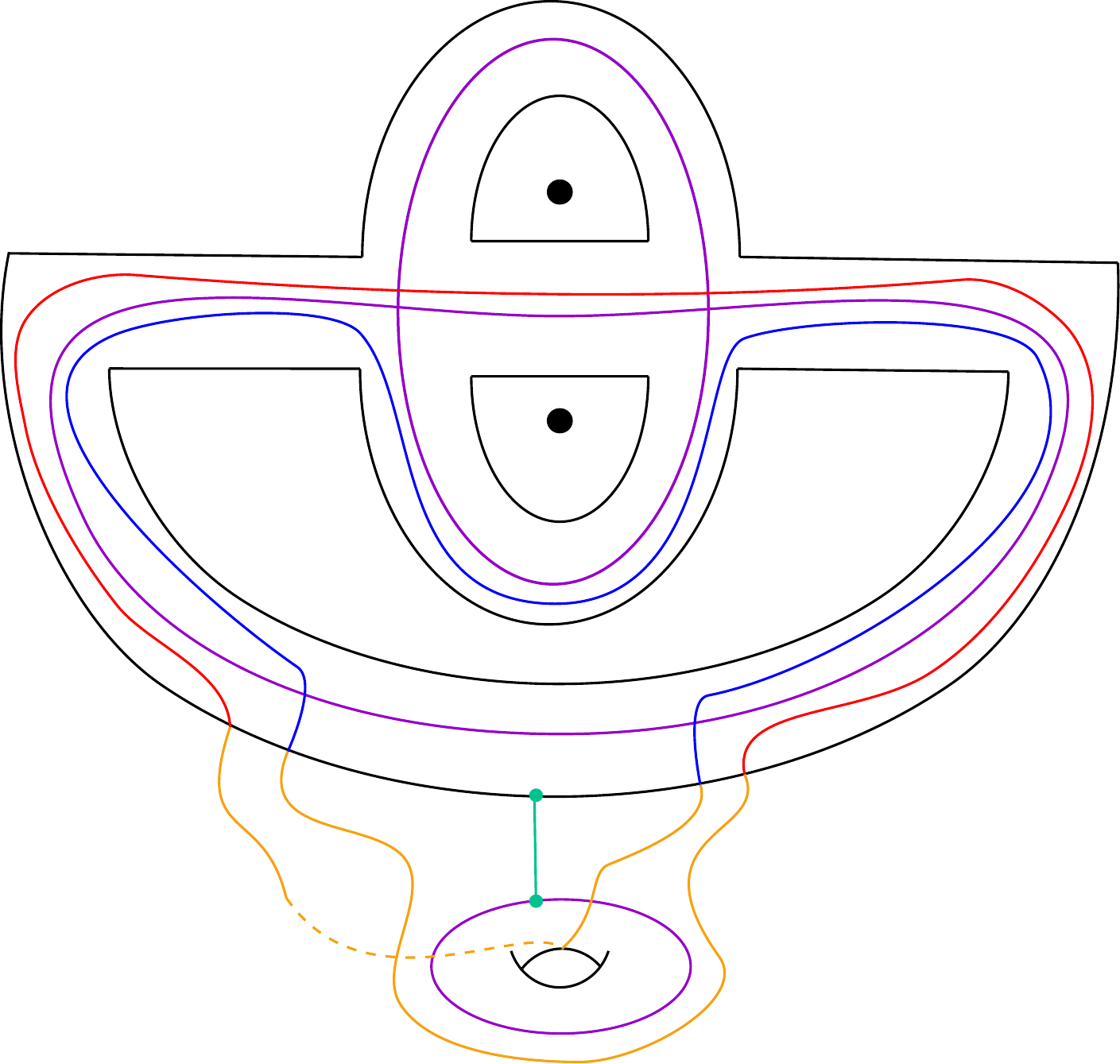
\caption{A diagram of $S_1$ and the construction of $\alpha$ in the event that $S_2$ has positive genus.}
\label{fig:Section5-non-planar}
\end{figure}

\begin{figure}[htb!]
\centering
\def\svgwidth{2in}
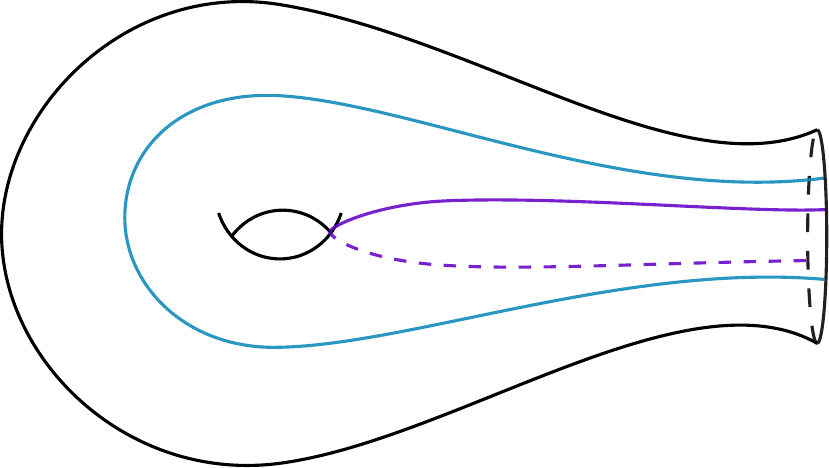
\caption{In any surface with at least one boundary component and genus $g \geq 1$, given an arc $\eta$ constructed from a nonseparating curve as described above, there exists an arc $\eta'$ disjoint from $\eta$ and whose endpoints link with those of $\eta$.}
\label{fig:Section5-linking}
\end{figure}

Now we specify the construction of $\alpha \cap S_1 = \{r_1, r_2 \}$ as shown in Figure \ref{fig:Lemma4-1-Fig1} ($r_1$ is the blue arc and $r_2$ is the red arc). Choose $r_1, r_2$ so that the endpoints of $r_1$ coincide with one endpoint of $\eta$ and with one of $\eta'$, and similarly for $r_2$. Then the concatenation $\eta \cdot r_1 \cdot \eta' \cdot r_2$ yields a simple closed curve $\alpha$ which is necessarily non-separating; indeed, by construction it intersects $c$ exactly once. Moreover, $\alpha$ intersects both $\gamma^+$ and $\gamma^-$ more than twice.

By applying to $\alpha$ a high power of a mapping class fixing $S_1$ and acting as a pseudo-Anosov on $S_2$, we obtain a second non-separating curve $\beta$ whose projection to $S_2$ is arbitrarily far away from the projection of $\alpha$ to $S_2$ in $\mathcal{AC}_1(S_2)$. The lemma now follows because both $\alpha$ and $\beta$ are at least distance $2$ from $\{\gamma^+, \gamma^-\}$, the only two curves in $L(u,v)$ that project trivially to $S_2$. \end{proof}

\begin{lemma} \label{2-edge-3}
Suppose $g \geq 1$ and that $(g,p) \neq (1,3)$ . If $(u,v)$ is a $2$-edge in $\mathcal{SC}(S)$ with $u$ nonseparating and $v$ separating, then $\diam(L(u,v)) \geq 3.$ Furthermore, whenever $d_{L}(\alpha, \beta)=3$, there exists a path of length $3$ in $L(u,v)$ from $\alpha$ to $\beta$ that passes through $\left\{b_1, b_2 \right\}$.
\end{lemma}

\begin{proof} We note that the last sentence is implied by the proof of Lemma \ref{lemma:leq3}: in each of the cases outlined there, there exists a path of length $3$ connecting $\alpha$ to $\beta$ which passes through $b_{1}$ or $b_{2}$.

Without loss of generality, we may assume $v$ is a separating curve bounding two punctures of $S$. Let $S_1$ denote the regular neighborhood of $u \cup v$. See Figure~\ref{fig:Proposition5.2-u-v-configs}. We claim that $S_{1}$ is topologically a sphere with $4$ boundary components. Indeed, $S_{1}$ has at least three boundary components, two of which correspond to the pair of punctures on one side of $v$. There are three boundary components if and only if $S_{1}$ has positive genus, but in that case, its Euler characteristic is $2-2-3 =-3$. Each complementary region of $u \cup v$ on $S_{1}$ is either a disk or a punctured disk and as a graph on $S_{1}$ with vertices corresponding to intersection points, $u \cup v$ is $4$-valent and therefore has twice as many edges as vertices. It follows that \[ \chi(S_1) = -3= i(u,v) - 2 i (u,v) + D, \] where $D \geq 0$ is unknown. Hence, $3 \leq i(u,v)$, a contradiction. So $S_1$ is a $4$-holed sphere with two boundary components corresponding to boundary components of $S$ and two boundary components, $b_1$ and $b_2$ corresponding to boundaries of $S_2 = S \setminus S_1$. The curves $b_1$ and $b_2$ are necessarily nonseparating as shown in the proof of Lemma~\ref{lemma:leq3}.

\begin{figure}[htb!]
\centering
\def\svgwidth{4.25in}
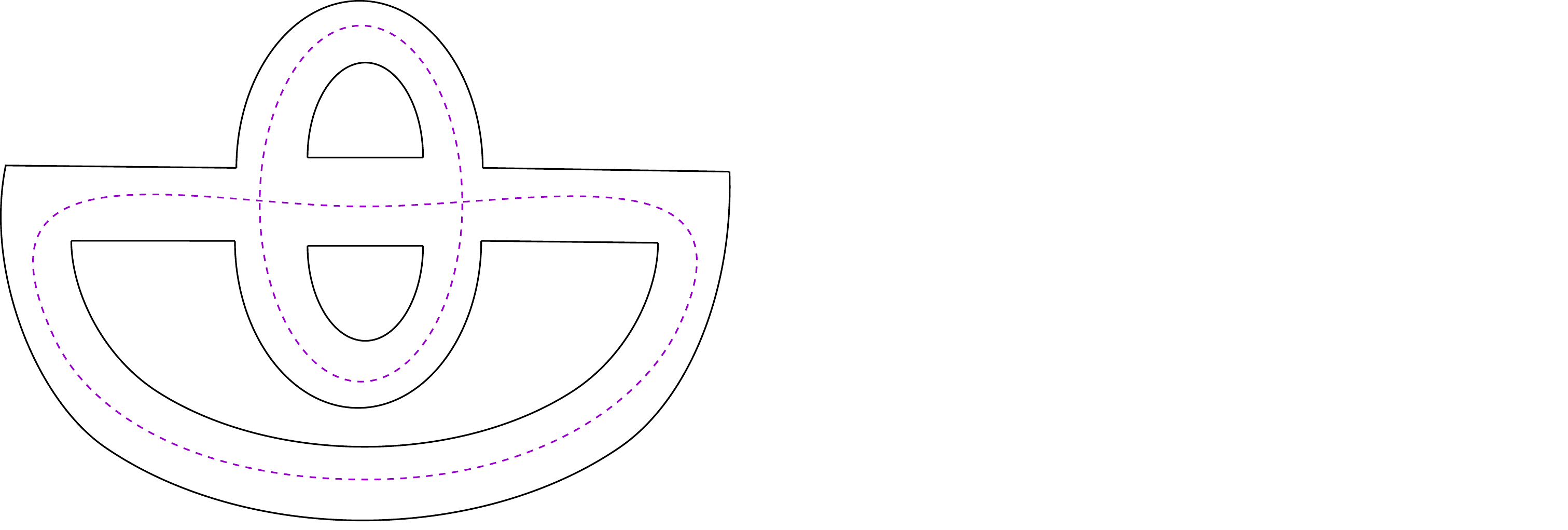
\caption{The possible configurations for $u,v,$ and $S_1$ with $u$ and $v$ shown in dotted lines. On the left, $S_1$ is a sphere with $4$ boundary components. On the right, $S_1$ is a torus with $3$ boundary components, which can be ruled out by Euler characteristic considerations.}
\label{fig:Proposition5.2-u-v-configs}
\end{figure}


We will construct a pair of curves $\alpha$ and $\beta$ so that any path from one to the other such that each vertex projects non-trivially to $S_{2}$ must be longer than $3$. On the other hand, $b_{1}, b_{2}$ are the only two curves in $L(u,v)$ that do not project to $S_{2}$. Note that $\gamma^{\pm}$ are \textit{not} in $L(u,v)$ since both intersect $u$ twice and if either one of them is separating, it can not bound a $3$-holed sphere on either of its sides by the topological assumptions made on $S$ in the statement of the lemma. 

Let $\alpha$ be a $2$-curve in configuration (iii) and let $\alpha \cap S_2 = \{\eta, \eta' \}$. Letting $\rho$ denote a  homeomorphism $S_2$ taking $b_1$ to $b_2$, let $\beta = \rho(\alpha)$. Therefore, both $\beta$ and $\alpha$ are non-separating; $\alpha$ (respectively $\beta$) intersects $b_{1}$ (respectively $b_{2}$) three times; and $\alpha$ (respectively $\beta$) intersects $b_{2}$ (respectively $b_{1}$) exactly once. 

We claim there is $K>0$ so that if 
\[ d_{\mathcal{AC}(S_{2})}(\pi_{S_2}(\alpha), \pi_{S_2}(\beta)) > K, \]
then if $\Gamma= \left\{\alpha, v_1,..., v_{n}, \beta \right\}$ is any path in $L(u,v)$ so that $v_{i}$ projects to $S_2$ non-trivially, then $n > 2$. Indeed, $i(v_{j}, v_{j+1}) \leq 2$ and thus $\Gamma$ may not be a path in  $\mathcal{AC}(S_{2})$, but it constitutes a sequence that makes uniformly bounded jumps. The existence of $\Gamma$ therefore implies an upper bound on the distance in $\mathcal{AC}(S_{2})$ between the projections of $\alpha$ and $\beta$ that depends only on $n$. 

Now, assume that the projections of $\alpha, \beta$ to $S_{2}$ are at an $\mathcal{AC}(S_{2})$ distance of at least $K$, where $K$ is as above. It follows that a shortest path from $\alpha$ to $\beta$ must pass through at least one of $b_{1}, b_{2}$ since the $L(u,v)$ distance between them is no more than $3$. Since $d_{L}(\alpha, b_{1})$ and $d_{L}(\beta, b_{2})$ are both at least $2$, a shortest path from $\alpha$ to $\beta$ can not have length $2$, and therefore must be of length $3$.  \end{proof}

In the context of the previous lemma, the curves $\left\{b_1, b_2 \right\}$ play the role of a \textit{shortcut set}, a notion we will introduce formally in Section \ref{section:k-curve}. Furthermore, the proof of Lemma \ref{k jointly separating} applies in our setting here exactly as written, and shows: 

\begin{lemma} \label{no shortcut} If $(u,v)$ is a $0$-edge such that both $u$ and $v$ are non-separating, then the diameter of $L(u,v)$ is either infinite, or it is $3$. In the latter case, there can not exist a pair of curves $b_1,b_2$ in $L(u,v)$ so that any shortest path of length $3$ in $L(u,v)$ passes through either $b_1$ or $b_2$.  
\end{lemma}

\subsection{Proof of Schaller's conjecture.}

We are now ready prove Conjecture~\ref{SchallerC}.


\vspace{2 mm}

\noindent \textbf{Theorem \ref{Schaller}} \textit{The automorphism group of $\mathcal{SC}(S)$ is isomorphic to $\Mod^\pm(S).$}

\vspace{2 mm}


\begin{proof} We begin by showing that automorphisms of $\mathcal{SC}(S)$ preserve $0$-edges. Let $(u,v)$ be a $0$-edge, and assume first that $S$ has genus at least $1$.

If $u$ and $v$ are both nonseparating, either $\diam(L(u,v)) = \infty$ or $\diam(L(u,v)) = 3$ and in this case, $L(u,v)$ does not possess a shortcut set-- curves for which any shortest path of length $3$ must pass through at least one of them. So, $(u,v)$ cannot be sent to a $1$-edge, since links of $1$-edges have diameter $4$, nor can it be sent to a 2-edge (since links of $2$-edges either have diameter $4$ or diameter $3$ and possess the curves $b_1, b_2$). 

If one or both of $u$ and $v$ are separating, then $\diam(L(u,v)) = \infty$. So again, $(u,v)$ cannot be sent to a $1$- or $2$-edge. 

If $S$ has genus $0$, then all curves are separating and we only need to distinguish between $0$-edges and $2$-edges. Note that links of $0$-edges have infinite diameter and links of $2$-edges (with both curves separating) have diameter at most $4$. Thus, automorphisms of $\mathcal{SC}(S)$ preserve $0$-edges, as desired. 

To complete the proof it remains to show that any automorphism $\psi$ of the subgraph of $\mathcal{SC}(S)$ consisting of all vertices but only $0$-edges is induced by a mapping class. If $g = 0$, this amounts to saying that the graph with vertices corresponding to separating curves that bound a $3$-holed sphere on one side and with edges corresponding to disjointness, has automorphism group isomorphic to the extended mapping class group. This follows readily from McLeay's extension \cite{McLeay} of Brendle--Margalit's work on \textit{complexes of regions} \cite{BrendleMargalit}, and we outline the idea as follows. 
A \textit{complex of regions} is a simplicial complex whose vertices are essential subsurfaces chosen from some specific subset of mapping class group orbits of all subsurfaces, and edges are determined by disjointness. When $g=0$, we can interpret the subgraph of $\mathcal{SC}(S)$ corresponding only to $0$-edges as a complex of regions, where each vertex represents a $3$-holed sphere. Theorem $2$ of \cite{McLeay} states that the automorphism group of a complex of regions is isomorphic to the extended mapping class group when every \textit{minimal} vertex (roughly speaking, a subsurface which does not properly contain another subsurface representing a vertex of that complex of regions) is \textit{small}, a technical assumption that amounts to saying that the topology of the entire surface is sufficiently complicated relative to the topology of the subsurface. In our context, every vertex will be minimal, and every vertex will be small so long as $p \geq 7$. The only remaining genus $0$ case is $p=6$; since the argument involves ideas from \cite{McLeay} that are not relevant to the rest of the remaining cases of Theorem \ref{Schaller}, we cover this in Appendix \ref{p6}.

If $g \neq 0$, then $\psi$ induces an automorphism of $\mathcal{N}(S)$. Indeed, any non-separating curve is involved in a $0$-edge of diameter $3$, but a $0$-edge involving a separating curve has infinite diameter. Thus there is a mapping class $f$ so that $\psi$ coincides with $f$ when restricted to the non-separating vertices. Consider $\phi \circ f^{-1}$. This is an automorphism of $\mathcal{SC}(S)$ which pointwise fixes each vertex corresponding to a non-separating curve. 

Let $u$ be a vertex of $\mathcal{SC}(S)$ corresponding to a separating curve, let $v = \phi \circ f^{-1}(u)$, and suppose $v \neq u$. 

\medskip

\noindent {\bf Claim:} There exists a nonseparating curve $\gamma$ that is disjoint from $u$ and intersects $v$ at least $3$ times. 

\begin{proof}[Proof of Claim:]

Suppose that $u$ and $v$ are disjoint. Note that since $S$ has genus at least $1$, then the complement of $u$ containing $v$ has genus at least $1$. So there exists a nonseparating curve $\gamma$ disjoint from $u$ which intersects $v$ arbitrarily many times.

Suppose $u$ and $v$ intersect. Once again, the component of the complement of $u$ which is not a $3$-holed sphere, call it $S'$, has genus at least $1$. Note that $v \cap S'$ is an essential multi-arc, call it $\nu$. We can then pick any component of $\nu$ and find a nonseparating curve $\gamma$ in $S'$ which intersects it arbitrarily many times. \end{proof}


By the claim there exists a nonseparating curve $\gamma$ that is distance $1$ from $u$ but distance $> 1$ from $v$. This is a contradiction since $\phi \circ f^{-1}(\gamma) = \gamma$. Thus, $u=v$, and therefore $\phi \circ f^{-1}$ is the identity on the full vertex set of $\mathcal{SC}(S)$, as desired.  \end{proof}

\section{The k-Curve Graph}
\label{section:k-curve}

We are now ready to show that the automorphism group of the k-curve graph is the extended mapping class group for $|\chi(S)|$ sufficiently large with respect to $k$. Throughout this section we will assume that $S$ is a connected, oriented surface with negative Euler characteristic. As before, we call edges in $\mathcal{C}_k(S)$ that connect vertices admitting disjoint representatives \textit{0-edges}. We call all other edges \textit{non-zero edges}. Distinguishing between $0$-edges and non-zero edges in $\mathcal{C}_k(S)$ is a more delicate process than distinguishing between $0$ and $1$-edges in $\N_1(S)$ and $\mathcal{SC}(S)$. In addition to the diameter we will also record two other properties of the edge links. First, we will consider the cardinality of the edge links, namely whether the edge link contains a finite or infinite number of vertices. Second, we will define a finite collection of curves associated to an edge link called a \textit{shortcut set}, whose existence (or nonexistence) will be our final tool for distinguishing between edge links. Throughout this section, we assume $S$ satisfies $|\chi(S)| \geq k + 512$ unless specified otherwise. See the appendix (particularly Remark \ref{512}) for an explanation of the relevance of this inequality.

\begin{figure}[htb!]
\centering
\def\svgwidth{2in}
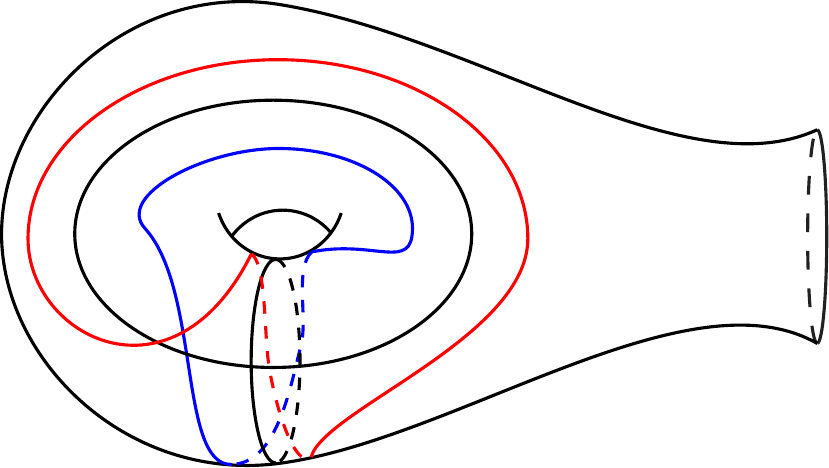
\caption{A shortcut set is a generalization of $\gamma^{\pm}$ in the case $k=1$. Indeed, when $i(u,v) =1$, then $\diam(L(u,v)) = 4$, and as demonstrated in Lemma \ref{leq4} there is a path of length at most $4$ between them that passes through $\{\gamma^+, \gamma^-\}$.}
\label{fig:Section6-shortcut-set}
\end{figure}

The goal of this section will be to prove the partition of edge types shown in Table \ref{table:distinguish} based on the three characteristics that we have just outlined. 

\begin{table}[htb!]
\centering
    \begin{tabular}{|l|c|c|c|}
    \hline
      Edge-type of $(u,v)$ & $|L(u,v)|$ & $\diam(L(u,v))$ & Shortcut Set \\
    \hline
    \hline
       $0$-edge, $u, v$ jointly non-separating  &  $\infty$ & $\infty$ & N/A \\
    \hline
        $0$-edge, $u, v$ jointly separating & & &\\
    \hline
        \enspace i) No component of $S\setminus (u\cup v)$ is a 3-holed sphere & $\infty$ & 3 & No \\
    \hline
        \enspace ii) A component of $S\setminus (u\cup v)$ is a 3-holed sphere & $\infty$ & $\infty$ & N/A\\
    \hline
       non-zero edge, $u,v$ not filling  &  $\infty$ & $1, 2, \text{ or } 4$ & N/A\\
    \hline
       non-zero edge, $u,v$ not filling &  $\infty$ & 3 & Yes \\
    \hline
    \end{tabular}
    \bigskip
    \caption{Strategy for distinguishing between $0$-edges and non-zero edges. N/A indicates that the existence of a shortcut set was not checked for these edge links. Note that we do not include the case of a non-zero edge with $u,v$ filling since if $i(u,v) \leq k$ and $|\chi(S)| \geq k + 512$, then they necessarily do not fill.}
    \label{table:distinguish}
\end{table}

\subsection{Diameter and cardinality of links.} We will now compute the diameters of various types of edges. Let $(u,v)$ be an edge in $\mathcal{C}_k(S)$. We will begin by considering the case when $u$ and $v$ are a filling pair.

\begin{lemma}\label{fillwholesurface}
If $(u,v)$ is an edge in $\mathcal C_k(S)$ such that $u \cup v$ fills $S$, then there are finitely many vertices in $L(u,v)$. 
\end{lemma}

\begin{proof} Since $u$ and $v$ fill, their union gives rise to the $1$-skeleton of a polygonal decomposition of $S$ where some polygons may be once-punctured. Any other essential curve $\gamma$ can be isotoped to be in minimal position with respect to $u \cup v$ and so it defines an equivalence class of cyclically ordered sequences each of length~$i(\gamma, u \cup v)$; one simply reads off the edges of the polygonal decomposition in accordance with the order in which $\gamma$ meets them. However, this does not yield a uniquely defined cyclic sequence because $\gamma$ can be homotoped over a vertex of one polygon and into another. We will consider any two sequences related in this way to be equivalent. 

There are at most finitely many sequences of edges in the polygonal decomposition of length at most $k$, and therefore there are at most finitely many (equivalence classes of) cyclic sequences of length at most $k$. This implies that there are at most finitely many curves which intersect both $u$ and $v$ at most $k$ times. \end{proof}

We will next consider the links of non-zero edges $(u,v)$ when $u$ and $v$ are not a filling pair. 

\begin{lemma}\label{k-edge finite}
If $(u,v)$ is a non-zero edge in $\mathcal{C}_k(S)$ such that $u$ and $v$ do not fill $S$, then \[\diam(L(u,v))\leq 4.\]
\end{lemma}

\begin{proof}
Let $\alpha \in L(u,v)$ and let $F(u,v)$ be the subsurface of $S$ filled by $u$ and $v$ (potentially with some complementary disks glued back in so as to make $F$ essential). Note that $u \cup v$ can be thought of as a $4$-valent graph $\Gamma$ with vertices in $u \cap v$ and edges given by arcs of either $u$ or $v$ running between intersection points. Note that $\Gamma$ has exactly twice as many edges as vertices. Since $F(u,v)$ is a thickening of $\Gamma$, we get
\[ \chi(F(u,v)) = - i (u,v).\] 

It follows that $|\chi(F(u,v))| \leq k,$ and hence
\begin{equation} \label{cardinality}
|\chi(S \setminus F(u,v))| \geq 512.
\end{equation}

There exists $\gamma \in L(u,v)$ such that $\gamma \subset F(u,v)$. This can be seen, for example, by surgering along intersections of $u$ and $v$ as in Hempel's argument (see \cite[Lemma 2.1]{Hempel}). If $\alpha \subset S\setminus F(u,v)$, then $d_L(\alpha, \gamma)=1$. Otherwise, if $\alpha \subset F(u,v),$ there exists a simple closed curve $\beta\subset S\setminus F(u,v)$ so that $\beta \in L(u,v)$, since $u$ and $v$ do not fill $S$. This yields a path between $\alpha$ and $\gamma$ in $L(u,v)$ of length 2. Hence $d_L(\alpha, \gamma)\leq 2$.

Lastly, if $\alpha$ nontrivially intersects both $F(u,v)$ and its complement, then we consider the multi-arc formed by $\alpha \cap \big(S\setminus F(u,v)\big)$. Abusing notation slightly, we will denote this multi-arc by $\alpha$. Lemma \ref{multi-arcint} implies that there exists some essential simple closed curve $\eta \subset S\setminus F(u,v)$ such that $i(\alpha, \eta)\leq k.$ Thus, $d_L(\alpha, \eta)=1$ and so $d_L(\alpha, \gamma) \leq 2$, since $d_L(\eta, \gamma) = 1$. 

It follows from the above cases that the diameter of $L(u,v)$ is at most 4. \end{proof}

Next we consider the diameter of a $0$-edge $(u,v)$ when $u$ and $v$ are jointly non-separating. 

\begin{lemma} \label{lemma:jnsep}
Let $u, v \in \mathcal{C}_k(S)$. If $(u,v)$ is a $0$-edge such that $u$ and $v$ are jointly non-separating, then $L(u,v)$ has infinite diameter.
\end{lemma}

In order to prove Lemma \ref{lemma:jnsep} we first need to establish the following quasi-isometry between $\mathcal{AC}(S)$ and $\mathcal{AC}_k(S)$.

\begin{proposition}\label{lemma:ACkqi}
Let $S$ be a surface of genus at least 2. Consider $\mathcal{AC}_k(S)$, the graph with the same vertex set as $\mathcal{AC}(S)$ and with edges connecting arcs and curves that intersect essentially at most $k$ times. Then $\mathcal{AC}(S) \qi \mathcal{AC}_k(S)$.
\end{proposition} 

\begin{proof}
Let $\alpha$ and $\beta$ be two vertices in $\mathcal{AC}_k (S)$. We use $d_{\mathcal{AC} _k}(\cdot, \cdot)$ to denote the distance in $\mathcal{AC}_k (S)$ and $d_{\mathcal{AC}}(\cdot, \cdot)$ the distance in $\mathcal{AC} (S)$.  Let $\phi: \mathcal{AC}_k (S) \rightarrow \mathcal{AC} (S)$ be the identity map on the vertices. Since every edge in $\mathcal{AC} (S)$ is also present in $\mathcal{AC}_k (S)$, we have that $$d_{\mathcal{AC}_k}(\alpha, \beta) \leq d_{\mathcal{AC}}(\phi(\alpha), \phi(\beta))_.$$ 

On the other hand, for any $\alpha, \beta$ connected by a non-zero edge in $\mathcal{AC}_k(S)$, by surgering along the intersections of $\alpha$ and $\beta$ as in Hempel's argument (see \cite{Hempel} Lemma 2.1), there is a path between $\alpha$ and $\beta$ consisting of only $0$-edges with length at most $2 \log_2(k)+ 2$ in $\mathcal{AC}_k(S)$. This path is mapped bijectively into $\mathcal{AC}(S)$ by $\phi$. Thus, \[d_{\mathcal{AC}}(\phi(\alpha), \phi(\beta)) \leq (2 \log_2(k) + 2) \cdot d_{\mathcal{AC}_k}(\alpha, \beta).\qedhere\] \end{proof}

We can now prove Lemma \ref{lemma:jnsep}.

\begin{proof}[Proof of Lemma \ref{lemma:jnsep}]
Let $u$ and $v$ be jointly non-separating disjoint curves on $S,$ and consider $S'= S\setminus (u \cup v).$ By Propositions \ref{qi} and \ref{lemma:ACkqi}, we know that $\mathcal{AC}_k(S') \qi \mathcal{AC}(S') \qi C(S').$ So we can consider a coarsely well-defined projection \[\tau: L(u, v)\to \mathcal{AC}_k(S')\] defined as follows. If $\alpha\in L(u,v)$ and $i(\alpha,u)=i(\alpha,v)=0,$ then $\tau(\alpha)=\alpha\in \mathcal{AC}_{k}(S')$. Otherwise, send $\alpha$ to the multi-arc representing its intersection with $S'$, which is a simplex in $\mathcal{AC}_k(S)$. It follows from the definition of $\tau$ that 

\begin{equation} \label{link-arc-curve} d_L(\alpha,\beta) \geq d_{\mathcal{AC}_k(S')}(\tau(\alpha),\tau(\beta)).\end{equation}

Consider a non-separating curve $\gamma \in \mathcal{C}(S'),$ and let $\phi: S \to S$ be a map fixing $u$ and $v$ pointwise and which restricts to a pseudo-Anosov on $S'$. Then for any $N \in \mathbb{N}$, there exists $n$ so that \[d_{\mathcal{C}(S')}(\gamma,\phi ^n(\gamma))\geq N.\] Since $\mathcal{AC}_k(S')$ is quasi-isometric to $\mathcal{C}(S'),$ we can choose appropriate $n$ to make  $d_{\mathcal{AC}(S')}(\gamma, \phi^n(\gamma))$ in $\mathcal{AC}_k(S')$ arbitrarily large. By inequality \ref{link-arc-curve} above, the diameter of $L(u,v)$ is infinite. \end{proof}

\subsection{Shortcut sets.} We will now make precise the definition of a \textit{shortcut set}. We then use it to distinguish between the remaining cases.

\begin{definition} Given $L(u,v)$ with $\diam(L(u,v)) = R< \infty$, a \emph{shortcut set} for $L(u,v)$ is a finite set of curves $\{\gamma_0,\ldots,\gamma_n\}$ with the following properties:
\begin{enumerate}
\item $\gamma_i \in L(u,v)$ for all $i$ and,
\item given any $\alpha,\beta\in L(u,v)$ with $d_L(\alpha, \beta) = R$,  there exists a path of length R between $\alpha$ and $\beta$ that passes through at least one of the $\gamma_i$'s.
\end{enumerate} 
\end{definition}

We can now prove the assertion given in the fifth row of Table \ref{table:distinguish}.

\begin{proposition}\label{diam3}
Given $u, v \in \mathcal{C}_k (S)$ so that $u,v$ do not fill $S$, if $(u,v)$ is a non-zero edge with $\diam(L(u,v))=3$, then there exists a shortcut set for $L(u,v)$.
\end{proposition}

\begin{proof}

Let $\Gamma= \{\gamma_{0},..., \gamma_{n}\}$ be the set of curves in $L(u,v)$ entirely contained in $F(u,v)$, which is finite by Lemma \ref{fillwholesurface}. We claim that they form a shortcut set. 

Let $\alpha, \beta \in L(u,v)$ with $d_L(\alpha, \beta)=3$. We will construct a path of length $3$ between $\alpha$ and $\beta$ which contains at least one $\gamma_i \in \Gamma$. This is trivially true if either $\alpha$ or $\beta$ is contained in the subsurface $F(u,v)$. 

\medskip

\noindent {\bf Claim:} Either $\alpha$ or $\beta$ has distance $1$ from $\Gamma$. 

\begin{proof} The claim is clear if $\alpha$ or $\beta$ is contained in $S\setminus F(u,v)$. Consider when $\alpha$ and $\beta$ intersect both $F(u,v)$ and its complement nontrivially. Assume by contradiction that both $\alpha$ and $\beta$ are $L(u,v)$-distance at least $2$ from every curve in $\Gamma$. We can then replace $\alpha$ with its image $\alpha'$ under a high power of a mapping class $\phi$ which restricts to the identity on $F(u,v)$ and acts as a pseudo-Anosov on $S \setminus F(u,v)$. By choosing a sufficiently large power, we can assume that
\[ d_{\mathcal{AC}(S \setminus F(u,v))}(\alpha', \beta) > 4.\] Observe that $\alpha'$ is still at least distance $2$ from $\Gamma$ because $\alpha \cap F(u,v)= \alpha' \cap F(u,v)$. 

Let $\left\{ \alpha', v_{0}, v_{1}, \beta \right\}$ be a path in $L(u,v)$ from $\alpha'$ to $\beta$, where $v_{0}$ and $v_{1}$ are not necessarily distinct (such a path always exists since $\diam(L(u, v))=3$ ). By assumption $\beta$ is at least $L(u,v)$-distance $2$ from every curve in $\Gamma$, so $v_0, v_1 \not \in \Gamma$. Then both $v_0$ and $v_1$ project non-trivially to $S\setminus F(u,v)$, which yields a path of length at most $3$ in $\mathcal{AC}(S \setminus F(u,v))$ between the projections of $\alpha'$ and $\beta$. This is a contradiction, which finishs the proof of the claim. \end{proof}

Now, without loss of generality, assume $\alpha$ is adjacent to some $\gamma _i$. Then $\beta \cap F(u,v) \neq \emptyset$. Otherwise $\beta$ is disjoint from $\gamma_i$ and $d_L(\alpha, \beta) \le 2$, a contradiction. By Lemma \ref{multi-arcint} there exists a simple closed curve $\eta \in S \setminus F(u,v)$ that intersects $\beta$ no more than $k$ times. Thus $\eta$ is adjacent to both $\gamma _i$ and $\beta$ in the link and we have a desired path of length $3$ from $\alpha$ to $\beta$ passing through the shortcut set. \end{proof}

The following two propositions together with Lemma \ref{lemma:jnsep} establishes the assertions in the first, third, and fourth rows of Table~1.

\begin{proposition} \label{k jointly separating}
Let $u, v \in \mathcal{C}_k(S)$ be non-separating. If $(u,v)$ is a $0$-edge such that $u$ and $v$ jointly separate $S$, then:
\begin{enumerate}
\item [(i)] If neither component of $S \setminus (u \cup v)$ is a $3$-holed sphere, then $L(u,v)$ has diameter $3$ and does not have a shortcut set.
\item [(ii)] Otherwise, $L(u,v)$ has infinite diameter.
\end{enumerate}
\end{proposition}

\begin{proof}
Denote the two components of $S \setminus (u \cup v)$ by $S_1$ and $S_2.$ Let $\alpha, \beta \in L(u,v)$ be distinct.

We will begin by considering (i).  By assumption we know that neither $S_1$ nor $S_2$ is a three-holed sphere. Note that if $\alpha$ and $\beta$ are both contained in the same component of $S\setminus(u \cup v)$, say $S_1$, then $d_L(\alpha, \beta)\leq 2$, since there is an essential curve contained in $S_2$ and thus, disjoint from both $\alpha$ and $\beta$. If $\alpha$ and $\beta$ are contained in different components of $S\setminus (u \cup v)$, then $d_L(\alpha,\beta)=1$ since they are disjoint.

Next, suppose that $\alpha \subset S_1$ and that $\beta$ intersects both $S_1$ and $S_2$ non-trivially. Then by Lemma \ref{multi-arcint}, there is an essential curve $\eta$ contained entirely in either $S_1$ or $S_2$ so that $i(\eta, \beta) \leq k$. If $\eta \subset S_2$, it follows that $d_{L}(\alpha, \beta) \leq 2$. If $\eta \subset S_1$, let $\rho \subset S_2$ be any essential curve, which exists because $S_2$ is not a $3$-holed sphere. Then $\left\{\alpha, \rho, \eta, \beta\right\}$ is a length $3$ path in $L(u,v)$ from $\alpha$ to $\beta$. 

Finally, suppose $\alpha$ and $\beta$ intersect both $S_1$ and $S_2$ nontrivially. 

We now construct a path $\{\alpha, \rho,\eta,\beta\}$ of length $3$  between in $L(u,v)$. There are three possibilities:

\begin{enumerate}
\item the projection $\pi_{S_1}(\alpha)$ consists of a single weighted arc. In this case, we let $\rho$ be a curve on $S_1$ disjoint from $\alpha$. This is always possible since $S_1$ is not a $3$-holed sphere. See Figure~\ref{fig:proposition6.7-1}. 

\begin{figure}[htb!]
\centering
\def\svgwidth{4in}
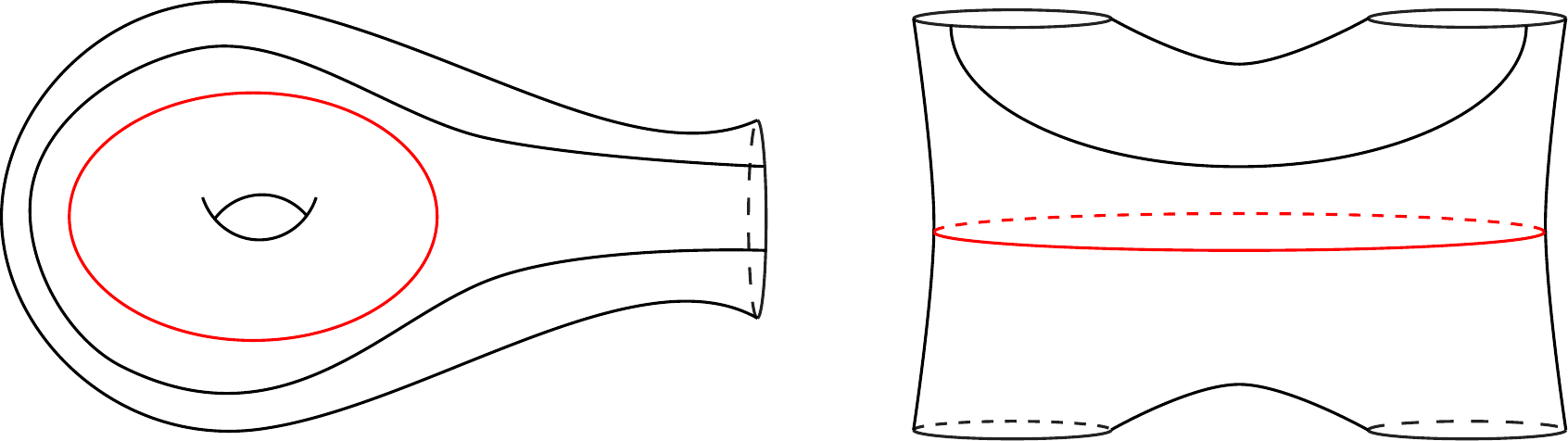
\caption{Since $S_1$ is not a $3$-holed sphere, it either has genus or at least $4$ punctures. In either case, given any arc there always exists an essential curve disjoint from it.}
\label{fig:proposition6.7-1}
\end{figure}

\item $\pi_{S_1}(\alpha)$ consists of multiple non-homotopic essential arcs, but every arc in $\pi_{S_1}(\alpha)$ begins and ends at two different boundary components (i.e., each arc intersects $u$ and $v$ only once). In this case, we construct $\rho$ in the following manner: take two non-homotopic arcs $c_1, c_2$ in $\pi_{S_1}(\alpha)$ whose endpoints on $u$ are adjacent to each other among all arcs in $\pi_{S_1}(\alpha)$. Concatenate $c_1$ and $c_2$ first with the subarc of $u$ that contains no other arcs' endpoints, and then with a subarc of $v$ that connects the two other endpoints of $c_1$ and $c_2$. The concatenation $\rho$ is an essential simple closed curve since we assume $c_1$ and $c_2$ are non-homotopic. See Figure~\ref{fig:proposition6.7-2}.

\begin{figure}[htb!]
\centering
\def\svgwidth{2.5in}
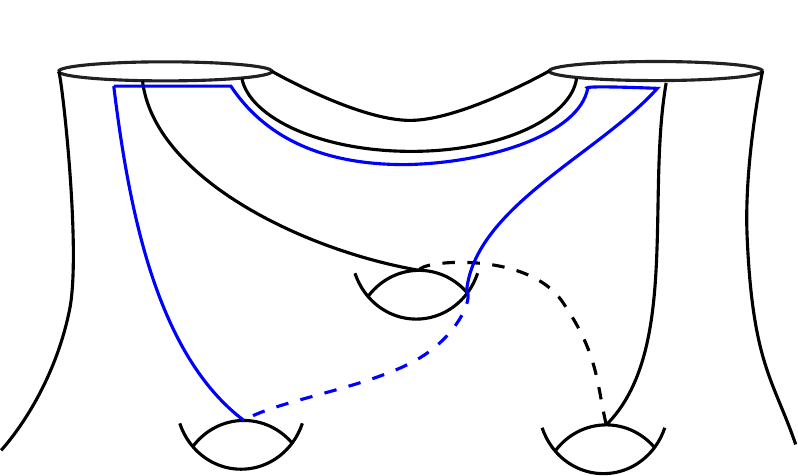
\caption{Construction of $\rho$ in the case that each arc of $\pi_{S_1}(\alpha)$ has endpoints on distinct boundary components and there exists at least 2 arcs in the projection that are not parallel.}
\label{fig:proposition6.7-2}
\end{figure}

\item $\pi_{S_1}(\alpha)$ consists of multiple non-isotopic essential arcs, but some arc in $\pi_{S_1}(\alpha)$ begins and ends at the same component. Without loss of generality, we assume there exists an arc $c \in \pi_{S_1}(\alpha)$ intersecting $u$ twice. In this case, we construct $\rho$ by concatenating $c$ with the subarc of $u$ that makes $\rho$ essential. This is always possible since $S_1$ is not a three-holed sphere as in Figure~\ref{fig:proposition6.7-1}. Since $\rho$ is disjoint from both $u$ and $v$, $\rho\in L(u,v)$. Meanwhile, $d_{L(u,v)}(\rho, \alpha)=1$ since $i(\alpha, \rho)\leq i(\alpha, u)\leq k.$ See Figure~\ref{fig:proposition6.7-3}.


Note that the curve $\rho$ will not be peripheral except in the case that $S_1$ is a punctured annulus bounded by $u$ and $v$. In this case, we choose $\rho$ so that it differs from $u$ by a single puncture on the interior of $S_1$. Since $S_1$ is not a three-holed sphere, $u$ and $v$ are not homotopic. 

\begin{figure}[htb!]
\centering
\def\svgwidth{2.5in}
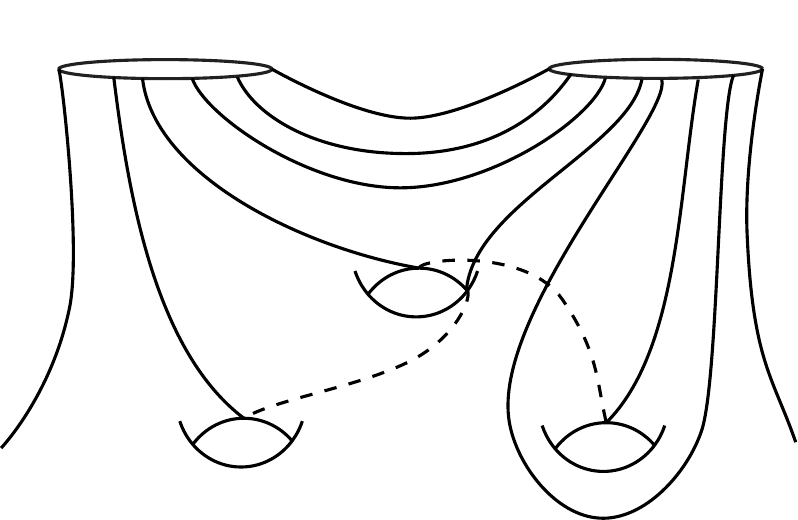
\caption{In the event that at least one arc in $\pi_{S_1}(\alpha)$ has both endpoints on $u$, we concatenate it with the red sub-arc of $u$ to form the desired $\rho$.}
\label{fig:proposition6.7-3}
\end{figure}

\end{enumerate}

We therefore have constructed a curve $\rho$ that is completely contained in $S_1$ and adjacent to $\alpha$ in $L(u,v).$ A similar curve $\eta \subset S_2$ adjacent to $\beta$ may be constructed. This gives a path of length $3$ between $\alpha$ and $\beta$ in $L(u,v)$ as desired. 

Now consider partial pseudo-Anosovs $\varphi_i$ which act as a pseudo-Anosov on $S_i$ and as the identity on $S\setminus S_i$ for $i = 1,2$. Let $\alpha' = \varphi_2^{n_2}(\varphi_1^{n_1}(\alpha))$ where $n_{1}, n_{2}$ are chosen sufficiently large so as to guarantee $d_{\mathcal{AC}_{k}(S_{i})}(\alpha', \beta)>3$. We can do this since the diameter of $\mathcal{AC}_k(S_i)$ is infinite for $i = 1, 2$ by Proposition \ref{lemma:ACkqi}. 

Note that we can construct a path $\{\alpha', \rho, \eta, \beta\}$ between $\alpha'$ and $\beta$ in the same way as above. We will now show that such a path is minimal in $L(u,v)$. Suppose to the contrary that there exists a path of length $2$ between $\alpha'$ and $\beta$, say $\{\alpha', \psi, \beta\}.$ Without loss of generality, $\pi_{S_1}(\psi)$ is nontrivial, so $d_{\pi_{S_1}(L)}(\alpha', \beta)=2,$ a contradiction. A path of lenth $1$ is ruled out for the same reason. Thus, the shortest path between $\alpha'$ and $\beta$ has length 3 and $\diam(L(u,v))=3.$

It remains to show that there does not exist a shortcut set in $L(u,v)$. Let $\Gamma = \{\gamma_1, \gamma_2, \ldots, \gamma_n\}$ be any finite set of curves in $L(u,v)$. Since $\Gamma$ is a finite set, its image in $\mathcal{AC}_k(S_i)$ has finite diameter. 

Let $\rho_i$ be an essential arc in $S_i$ that has one end point on $u$ and the other on $v$. Fix $\beta _i = \varphi_i^n(\rho_i)$ with $n \in \mathbb N$ large enough so that $d_{\mathcal{AC}_k(S_i)}(\beta_i, \pi_i(\Gamma)) > 3$, where $\varphi_i$ are defined as above. Now let $\alpha _i = \varphi_i^{m}(\rho_i)$ where $m \in \mathbb N$ is chosen sufficiently large so that $d_{\mathcal{AC}_k(S_i)}(\alpha_i, \beta_i) > 6$. Note that this implies $d_{\mathcal{AC}_k(S_i)}(\alpha_i,\pi_i(\Gamma) ) >3$. 

Let $\alpha$ and $\beta$ be the curves obtained by concatenating $\alpha_1$ with $\alpha_2$ and $\beta_1$ with $\beta_2$ respectively. Note that $\alpha$ and $\beta$ intersect each of $u$ and $v$ exactly once and are therefore contained in $L(u,v)$ and $d_{L(u,v)}(\alpha, \beta) = 3$, by construction. Additionally, if $\alpha$ or $\beta$ were adjacent to some $\gamma_j \in \Gamma$ in $L(u,v)$ then one of the $\alpha_i$ or $\beta_i$ would be adjacent to $\pi _ i (\gamma _j)$ in $\mathcal{AC}_k(S_i)$, a contradiction. Therefore, $\alpha$ and $\beta$ are at least distance $2$ from $\Gamma$ and therefore there is no path of length $3$ from $\alpha$ to $\beta$ that goes through $\Gamma$. We conclude that there does not exist a shortcut set on $S$ in $L(u,v)$. 

We now finish the proof by considering (ii). Suppose exactly one of the connected components of $S\setminus (u\cup v)$ is a three-holed sphere, say $S_1$. Then we can choose essential simple closed curves $\alpha, \beta$ contained in $S_2$ such that $d_{\mathcal{AC}_k(S_2)}(\alpha,\beta)$ is arbitrarily large. Let $P = \{\alpha,v_1,\dots,v_n,\beta\}$ be any path of shortest length between $\alpha$ and $\beta$ in $L(u,v).$ Then each $v_i$ for $1\leq i\leq n$ projects nontrivially to $S_2.$ Therefore $\pi_{S_2}(P) = \{\pi_{S_2}(\alpha),\pi_{S_2}(v_1),\dots,\pi_{S_2}(v_n),\pi_{S_2}(\beta)\}$ is a path in $\mathcal{AC}_k(S_2).$ Observe that $\pi_{S_2}(\alpha) = \alpha$ and $\pi_{S_2}(\beta) = \beta$, which implies $|P| \geq d_{\mathcal{AC}_k(S_2)}(\alpha, \beta)$. So there exists curves $\alpha, \beta \in L(u,v)$ that are arbitrarily far apart in $L(u,v)$ and so $L(u,v)$ has infinite diameter. \end{proof}

It remains to consider the possibility that $(u,v)$ is a $0$ edge with at least one of $u,v$ separating: 

\begin{proposition} \label{k-separating}
Let $u, v \in \mathcal{C}_k(S)$ such that at least one of them is separating. If $(u,v)$ is a $0$-edge, then:
\begin{enumerate}
\item [(i)] If at least two components of $S \setminus (u \cup v)$ are not $3$-holed spheres, then $L(u,v)$ has diameter $3$ and does not have a shortcut set.
\item [(ii)] Otherwise, $L(u,v)$ has infinite diameter.
\end{enumerate}
\end{proposition}

\begin{proof} This argument will follow the logic of the proof of Proposition \ref{k jointly separating} very closely; we include it in full for the ease of the reader.

Suppose that at least two components of $S \setminus (u \cup v)$ are not $3$-holed spheres. Note that there can be at most three components of $S \setminus (u \cup v)$; if there are exactly two components which are not $3$-holed spheres, denote them by $S_{1}, S_{2}$. In this scenario, the exact argument used in (i) of Proposition \ref{k jointly separating} applies here.  

We also note that if we are in scenario (ii), the argument used above in Proposition \ref{k jointly separating} applies here, verbatim. 

It remains to consider the possibility that there are three components of $S \setminus (u \cup v)$ all of which are not $3$-holed spheres. Denote these components by $S_{1}, S_{2}, S_{3}$. Let $\alpha, \beta$ be curves in $L(u,v)$. Suppose that $\alpha$ is disjoint from some $S_{i}$. Using Lemma \ref{multi-arcint}, there exists some $\eta$ intersecting $\beta$ at most $k$ times and contained entirely in some $S_{j}$. If $j= i$, then $d_{L}(\alpha, \beta) \leq 2$. Otherwise, choose $\rho \subset S_{i}$ arbitrarily and note that $\left\{\alpha, \rho, \eta, \beta \right\}$ is a length $3$-path. 

Assume next that $\alpha, \beta$ both intersect each $S_{i}$. This is analogous to the case in the previous proof corresponding to $\alpha, \beta$ intersecting both $S_{1}, S_{2}$. The argument used in this setting in the previous proposition yielded a curve $\rho$ contained entirely inside $S_{1}$ and disjoint from $\alpha$. The exact same argument can be used here to produce a curve contained entirely inside either $S_{1}, S_{2}, S_{3}$ and disjoint from $\alpha$. By symmetry, we can also produce a curve disjoint from $\beta$ and contained entirely in one of the other components.  

For the lower bound on diameter, choose partial pseudo-Anosovs $\phi_{i}$ on $S_{i}$ and choose some $\alpha \in L(u,v)$ intersecting each $S_{i}$. Then, as in the previous proof, choose $n_{i}, i=1,2,3$ sufficiently large so as to guarantee that 
\[ d_{\mathcal{AC}_{k}(S_{i})}(\alpha, \phi_{1}^{n_{1}}(\phi_{2}^{n_{2}}(\phi_{3}^{n_{3}}(\alpha))) > 3. \]
Then as in the previous argument, no length $2$ path can exist between $\alpha$ and $\beta:= \phi_{1}^{n_{1}}(\phi_{2}^{n_{2}}(\phi_{3}^{n_{3}}(\alpha)))$ since any curve $\psi$ realizing this path has to project non-trivially to at least one $S_{i}$ and this would yield a path of length $2$ between the projections of $\alpha$ and $\beta$ in that subsurface. 

Finally we will show that there does not exist a shortcut set in $L(u,v)$. Consider any finite set of curves $\Gamma = \{\gamma_1, \gamma_2, \ldots \gamma_n\}$ with each $\gamma_i \in L(u,v)$. Note that $\Gamma$ has finite diameter in $\mathcal AC_k(S_i)$. The idea is to construct curves $\alpha, \beta \in L(u,v)$ which are distance exactly $3$ from each other, but are also distance at least $2$ from $\Gamma$. This rules out any curve in $\Gamma$ from appearing in any length of path $3$ between $\alpha$ and $\beta$ and we can conclude that there is no shortcut set for $L(u,v)$. In fact, the proof can be carried out in the exact same way as in the previous proposition, since we are assuming that each component of $S \setminus (u \cup v)$ are not $3$-holed spheres. 

\end{proof}

We are now in a position to prove our main result. 

\vspace{2 mm}

\noindent \textbf{Theorem \ref{k-curve}. }\textit{Suppose $|\chi(S)| \geq k+ 512$. Then the natural map 
\[ \mbox{Mod}^{\pm}(S) \rightarrow \mbox{Aut}(\mathcal{C}_k(S)) \]
is an isomorphism.}

\vspace{2 mm}

\begin{proof} 
For any edge $(u,v)$ in $\mathcal{C}_k(S)$, note that any automorphism of $\mathcal{C}_k(S)$ preserves the cardinality of $L(u,v)$, the diameter of $L(u,v)$, and the existence of a shortcut set for $L(u,v)$. Therefore by Lemma \ref{k-edge finite}, Lemma \ref{lemma:jnsep}, Proposition \ref{diam3}, and \ref{k jointly separating}, we obtain all the assertions made in Table \ref{table:distinguish}. This implies that every automorphism of $\mathcal{C}_k(S)$ sends $0$-edges to $0$-edges and non-zero edges to non-zero edges. Hence any automorphism of $\mathcal{C}_{k}(S)$ induces an automorphism of the curve graph $\mathcal{C}(S)$ and therefore corresponds to a mapping class. The other direction of the isomorphism is clear and hence the theorem follows. \end{proof}

\appendix
\section{Where does 512 come from?}

\label{appendix}

The following result is necessary for the proofs of Proposition \ref{diam3} and \ref{k jointly separating} and introduces the restriction of $|\chi(S)| \geq k + 512$ that appears in Theorem \ref{k-curve} (see Remark \ref{512}).

\begin{lemma} \label{multi-arcint}
There exists some constant $D>0$ satisfying the following. If $Y\subset S$ is an essential subsurface (i.e, all boundary components are essential in $S$) satisfying $|\chi(S\setminus Y)|>D$ and $\alpha$ a simple closed curve on $S$ in minimal position with $\partial Y$ and with $i(\alpha, \partial Y) \leq 4k$ for some $k \in \mathbb N$, then there exists an essential simple closed curve $\beta$ on $S\setminus Y$ such that $i(\alpha, \beta) \le k$. 
\end{lemma}

Lemma \ref{multi-arcint} appears (in a different context and stated in a different way) as Proposition 3.1 of \cite{Aougab2}. For the sake of completeness, we include a sketch of proof here, which requires the following fact: 

\begin{lemma}[Lemma 3.2, \cite{Fiorini}] \label{girth}
Let $\varepsilon > 0$. There exists a decreasing function $f:(0,1) \longrightarrow \mathbb{R}_+$ so that if $G=(V,E)$ is any graph with average degree greater than $2+\varepsilon$, then $G$ has girth no larger than $g(\varepsilon)\cdot log_2(|V|)$.
\end{lemma}

\begin{remark} \label{512} The proof of Lemma $3.2$ of \cite{Fiorini} gives the bound of $18$ for $f(1/2)$. This implies that any choice of $D > 296$ is sufficient for the statement of Theorem \ref{k-curve} to hold. However, we choose $D = 512$ to simplify our computation.
\end{remark}

We are now ready to sketch the proof of Lemma \ref{multi-arcint}.

\begin{proof}[Proof of Lemma \ref{multi-arcint}]
If there exists an essential simple curve in $S\setminus Y$ disjoint from $\alpha,$ we are done. 
If no such curve exists,  $\alpha \cap (S \setminus Y)$ is a filling weighted multi-arc on $S \setminus Y$, and therefore the complement of $\alpha$ in $S\setminus Y$ is a collection of polygons. Abusing notation slightly, we refer to this weighted multi-arc $\alpha \cap (S \setminus Y)$ as $\alpha$. Note that 
\begin{equation} \label{linear above and below}
 |\chi(S \setminus Y)| \leq |\alpha| \leq 3|\chi(S \setminus Y)|. 
\end{equation}

Let $\alpha'$ be the collection of all arcs in $\alpha$ with weight $\ge \dfrac{2k}{\sqrt{|\chi(S \setminus Y)|}},$ which we will call {\bf large mass arcs}. Since $i(\alpha, \partial Y) \leq 4k$, there are at most $\sqrt{|\chi(S \setminus Y)|}$ such arcs in $\alpha.$ 

Let $S'$ be the complement of $\alpha'$ in $S \setminus Y.$ Cutting along any arc can decrease the absolute value of the Euler characteristic by at most $2$, and therefore 
\begin{equation} \label{sqrt}
|\chi(S')| \geq |\chi(S \setminus Y)| - 2\sqrt{|\chi(S \setminus Y)|},
\end{equation}
which is positive so long as $|\chi(S \setminus Y)| >4$. 

Let $G$ denote the dual graph to $\alpha$ on $S \setminus Y$: one vertex for each complementary polygon, two of which are connected by an edge when the corresponding polygons share a boundary edge. The average degree $\bar{d}(G)$ of $G$ is at least 3, since otherwise two arcs in $\alpha$ would be homotopic. Hence, 
\[ \bar{d}(G) = \dfrac{2|E(G)|}{|V(G)|}\geq 3. \]
We now define $G'$ to be the dual graph on $S'$ to the collection of arcs in $\alpha$ which do not have large mass. Using (\ref{linear above and below}) and (\ref{sqrt}), calculating the average degree of a vertex yields: 
\[\dfrac{2E(G')}
{V(G')}=\dfrac{2(|\alpha|-|\alpha'|)}{V(G')} \geq 2.5,\]
where the last inequality holds so long as $|\chi(S \setminus Y)| >36$.

By Lemma \ref{girth}, there exists a cycle, $\beta,$ on $G'$ with edge-length at most 
\[ f\left(\dfrac{1}{2}\right)\cdot \log_2(|V(G')|) \leq f\left(\dfrac{1}{2}\right)\cdot \log_2(2|\chi(S \setminus Y)|).\]

Without loss of generality, $\beta$ is simple (otherwise, there exists a shorter cycle). To prove that $\beta$ is essential, one shows that inessential intersections between $\beta$ and arcs of $\alpha$ imply the existence of inessential intersections between $\alpha$ and $\partial Y$ (see \cite{Aougab2} for details).

Therefore, 
\[i(\beta,\alpha) \leq \dfrac{f \left(\frac{1}{2}\right)\cdot \log_2(2|\chi(S \setminus Y)|)\cdot 2k}{\sqrt{|2\chi(S \setminus Y)|.}}\]
Hence it suffices to choose $D> 36$ and sufficiently large so that 
\[ \frac{1}{2 \cdot f\left(\frac{1}{2}\right)} > \frac{\log_2(D)}{D}. \qedhere \]
\end{proof}

We note that the conclusion of Lemma \ref{multi-arcint} is not implied by the argument of Hempel \cite{Hempel}. Indeed, the more standard surgery argument will only reduce the intersection number by a factor of $2$, as opposed to $4$.  

\section{The p= 6 case of Schaller's Conjecture} \label{p6}

To resolve Schaller's conjecture when $(g,p)=(0,6)$, we use an argumentation motivated by the tools in Section $5$ of \cite{McLeay}. In particular, we will prove the following.

\begin{proposition}\label{prop:appendixB}
An automorphism of the $0$-edge subgroup of $\mathcal{SC}(S_{0,6})$ can be uniquely extended to an automorphism of the full curve graph. 
\end{proposition}

Thus, by Proposition~\ref{prop:appendixB}, we obtain an injection of the automorphism group of $\mathcal{SC}(S)$ into the extended mapping class group, as desired. To this end, recall that a \textit{join} is a graph formed by two collections $\mathcal{V}, \mathcal{U}$ of vertices such that a pair of vertices $x,y \in \mathcal{V} \cup \mathcal{U}$ span an edge if an only if either $x \in \mathcal{V}, y \in \mathcal{U}$ or $x \in \mathcal{U}, y \in \mathcal{V}$. 

Using the language of \cite{McLeay}, a join is $2$-\textit{sided} whenever $\mathcal{V}, \mathcal{U}$ are both infinite sets, and a join is \textit{maximal} when (i) neither $\mathcal{V}$ nor $\mathcal{U}$ can be replaced with proper super-sets while maintaining the join property and (ii) there does not exist a vertex $x$ simultaneously disjoint from every vertex in $\mathcal{V} \cup \mathcal{U}$. It is immediate that a graph automorphism must preserve the collection of maximal $2$-sided joins. 

Let $\mathcal{J} = \left\{\mathcal{V}, \mathcal{U} \right\}$ be a maximal $2$-sided join in the $0$-edge subgraph of $\mathcal{SC}(S_{0,6})$. Given a vertex $v$, let $b_{1}(v), b_{2}(v)$ denote the two boundary components of $S_{0,6}$ cut off by $v$. It follows that if $v \in \mathcal{V}$ and $u \in \mathcal{U}$, 
\[ \left\{b_{1}(v), b_{2}(v) \right\} \cap \left\{b_{1}(u), b_{2}(u) \right\} = \emptyset, \]
since otherwise the curves corresponding to $u$ and $v$ would intersect. Therefore, $\mathcal{J}$ determines two disjoint subsets $T_{1}, T_{2}$ of the $6$ boundary components of $S$, such that each vertex $v \in \mathcal{V}$ (respectively $u \in \mathcal{U}$) has the property the $b_{i}(v) \in T_{1}$ (respectively, $b_{i}(u) \in T_{2}$). 

We claim that $|T_{1}|= |T_{2}|= 3$, for suppose that $|T_{1}|= 2$. Then maximality of $\mathcal{J}$ implies that there must be a boundary component $b$ that is neither in $T_{1}$, nor in $T_{2}$. Indeed, consider a pair of separating curves that jointly encircle the same two boundary components but which are arranged on opposite sides of a third boundary component that they jointly block from the rest of the surface, as in Figure~\ref{fig:AppendixB-1}. If no such pair is present in $\mathcal{V}$, then either 2-sidedness or maximality is contradicted, and as soon as such a pair exists, the boundary component they jointly block can not be in $T_{2}$. However, this also contradicts maximality, because we can then consider a curve that encircles the blocked boundary component and one of the two boundary components in $T_{1}$; such a curve can be added to $\mathcal{V}$. 

\begin{figure}[htb!]
\centering
\def\svgwidth{2in}
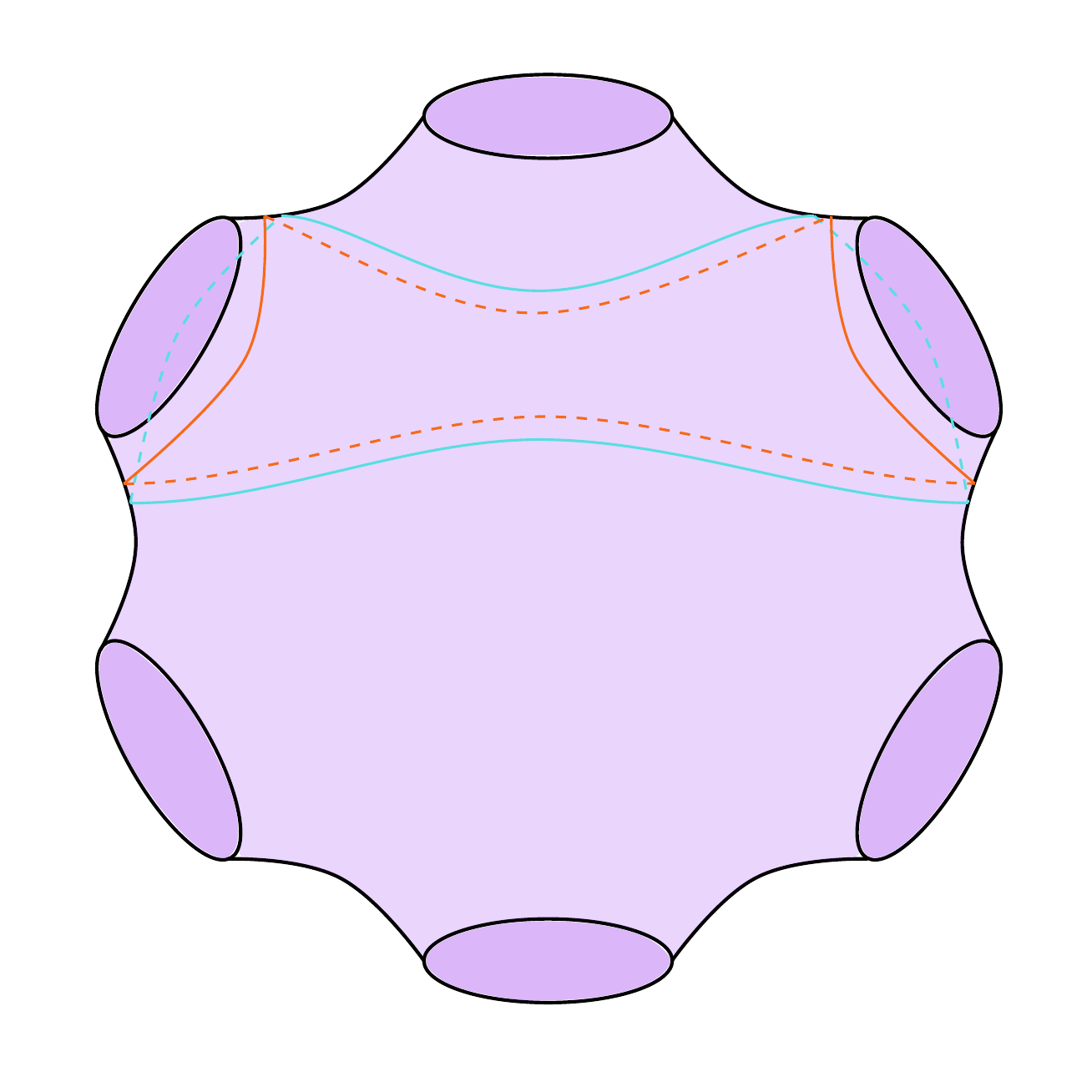
\caption{The orange and blue curves both encircle boundary components 2 and 6, while jointly blocking boundary component 1 from the rest of the surface.}
\label{fig:AppendixB-1}
\end{figure}

It follows that a maximal $2$-sided join $\mathcal{J}$ corresponds to two infinite collections $\mathcal{V}, \mathcal{U}$ and a partitioning of the $6$ boundary components of $S_{0,6}$ into two subsets $T_{1}, T_{2}$ of three boundary components each such that each vertex in $\mathcal{V}$ encircles two boundary components of $T_{1}$ (and respectively for $u \in \mathcal{U}$ and $T_{2}$). It follows that exactly one component of $S \setminus \left\{\mathcal{V} \cup \mathcal{J}\right\}$ is an essential annulus separating $T_{1}$ and $T_{2}$, and thus there is a unique (up to isotopy) simple closed curve $\delta_{\mathcal{J}}$ determined by $\mathcal{J}$. See Figure~\ref{fig:AppendixB-2}

\begin{figure}[htb!]
\centering
\def\svgwidth{2in}
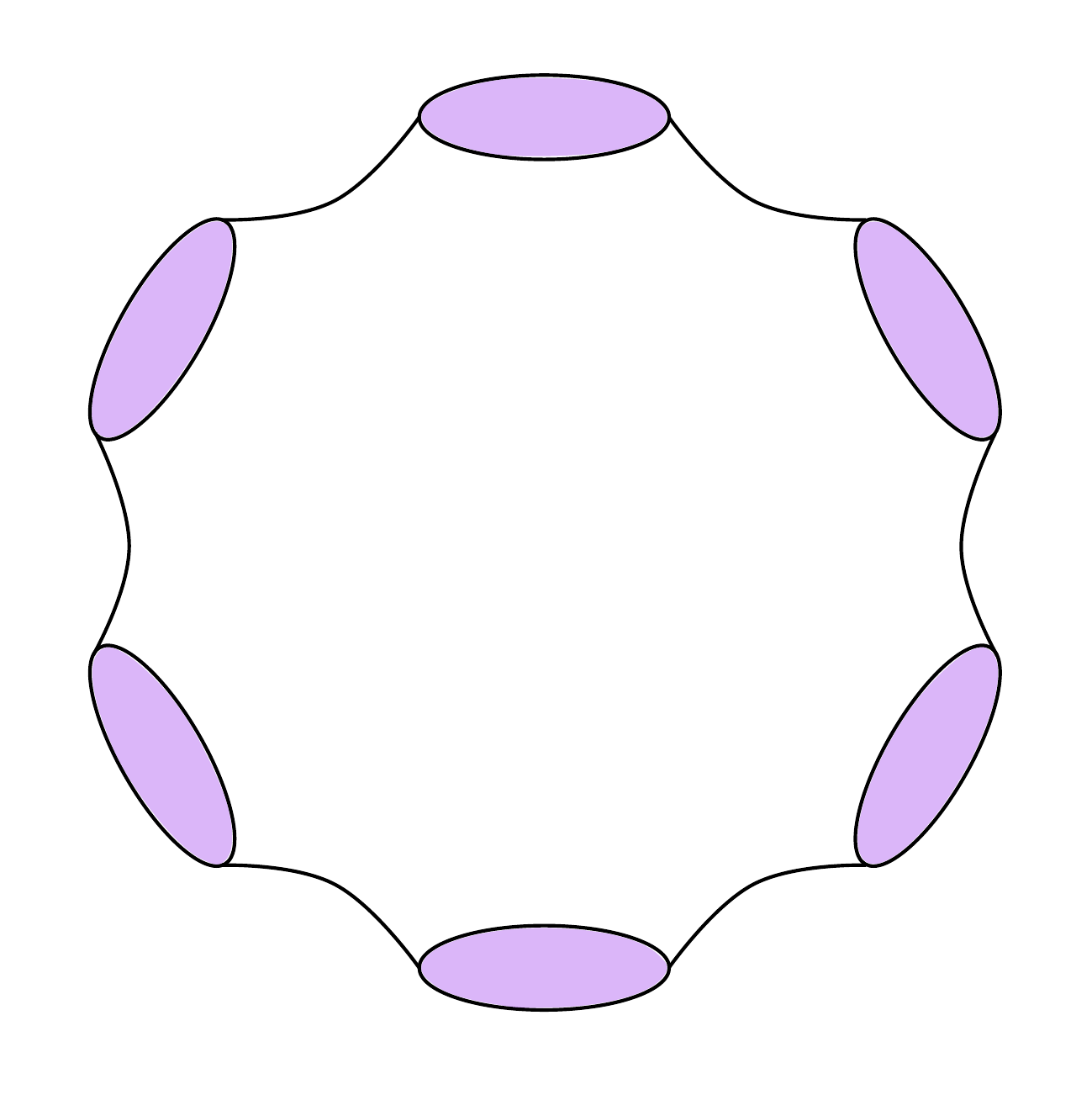
\caption{The simple closed curve $\delta_{\mathcal J}$ separating $T_1 = \{1, 2, 3\}$ and $T_2 = \{4, 5, 6\}$ is the core curve of the annulus shown in light blue.}
\label{fig:AppendixB-2}
\end{figure}

Now, let $\Phi$ be an automorphism of the $0$-edge subgraph of $\mathcal{SC}(S_{0,6})$, and let $\alpha$ be a vertex of the curve graph which is not a vertex of $\mathcal{SC}(S_{0,6})$. It follows that $\alpha$ separates $S$ into two subsurfaces that each possess three of the original boundary components. There then corresponds to $\alpha$ a $2$-sided maximal join $\mathcal{J}_{\alpha}= \left\{\mathcal{V}_{\alpha}, \mathcal{U}_{\alpha} \right\}$ of the $0$-edge subgraph of $\mathcal{SC}(S)$: for instance, $\mathcal{V}_{\alpha}$ consists of all vertices of $\mathcal{SC}(S)$ disjoint from $\alpha$ and to the left of $\alpha$. The automorphism $\Phi$ then sends $\mathcal{J}_{\alpha}$ to a maximal $2$-sided join $\Phi(\mathcal{J})$, which, by the above paragraphs, uniquely determines a separating curve which we define to be $\Phi(\alpha)$. 

This extends $\Phi$ to the entire curve graph, and it remains only to check that $\Phi$ is an automorphism. If $v \in \mathcal{SC}(S)$, $\alpha \notin \mathcal{SC}(S)$, and $i(v, \alpha)= 0$, then $i(\Phi(v), \Phi(\alpha))= 0$ since $\Phi(\alpha)$ will be the curve corresponding to a join for which $v$ lives in one of the two vertex sets. If $\alpha, \beta$ are both not in $\mathcal{SC}(S)$, they must intersect. This implies that the corresponding $2$-sided maximal joins $\mathcal{J}_{\alpha}, \mathcal{J}_{\beta}$ are distinct, and that therefore $\Phi(\mathcal{J}_{\alpha}) \neq \Phi(\mathcal{J}_{\beta})$, which in turn implies that $\Phi(\alpha) \neq \Phi(\beta)$ and so these curves must also intersect.

\bibliographystyle{plain}
\bibliography{paper}

\end{document}

%% file: essential-arcs-curves.pdf_tex
\begingroup%
  \makeatletter%
  \providecommand\color[2][]{%
    \errmessage{(Inkscape) Color is used for the text in Inkscape, but the package 'color.sty' is not loaded}%
    \renewcommand\color[2][]{}%
  }%
  \providecommand\transparent[1]{%
    \errmessage{(Inkscape) Transparency is used (non-zero) for the text in Inkscape, but the package 'transparent.sty' is not loaded}%
    \renewcommand\transparent[1]{}%
  }%
  \providecommand\rotatebox[2]{#2}%
  \ifx\svgwidth\undefined%
    \setlength{\unitlength}{656.53660717bp}%
    \ifx\svgscale\undefined%
      \relax%
    \else%
      \setlength{\unitlength}{\unitlength * \real{\svgscale}}%
    \fi%
  \else%
    \setlength{\unitlength}{\svgwidth}%
  \fi%
  \global\let\svgwidth\undefined%
  \global\let\svgscale\undefined%
  \makeatother%
  \begin{picture}(1,0.3382605)%
    \put(0,0){\includegraphics[width=\unitlength,page=1]{essential-arcs-curves.pdf}}%
    \put(0.14391036,0.18224128){\color[rgb]{0,0,0}\makebox(0,0)[rb]{\smash{$\eta$}}}%
    \put(0.5824652,0.269958){\color[rgb]{0,0,0}\makebox(0,0)[rb]{\smash{$\beta$}}}%
    \put(0.30529518,0.12177848){\color[rgb]{0,0,0}\makebox(0,0)[b]{\smash{$\gamma_1$}}}%
    \put(0.80345637,0.16496892){\color[rgb]{0,0,0}\makebox(0,0)[lb]{\smash{$\alpha$}}}%
    \put(0,0){\includegraphics[width=\unitlength,page=2]{essential-arcs-curves.pdf}}%
    \put(0.63460307,0.12082979){\color[rgb]{0,0,0}\makebox(0,0)[b]{\smash{$\gamma_2$}}}%
  \end{picture}%
\endgroup%

%% file: boundary-slide.pdf_tex
\begingroup%
  \makeatletter%
  \providecommand\color[2][]{%
    \errmessage{(Inkscape) Color is used for the text in Inkscape, but the package 'color.sty' is not loaded}%
    \renewcommand\color[2][]{}%
  }%
  \providecommand\transparent[1]{%
    \errmessage{(Inkscape) Transparency is used (non-zero) for the text in Inkscape, but the package 'transparent.sty' is not loaded}%
    \renewcommand\transparent[1]{}%
  }%
  \providecommand\rotatebox[2]{#2}%
  \ifx\svgwidth\undefined%
    \setlength{\unitlength}{429.09210218bp}%
    \ifx\svgscale\undefined%
      \relax%
    \else%
      \setlength{\unitlength}{\unitlength * \real{\svgscale}}%
    \fi%
  \else%
    \setlength{\unitlength}{\svgwidth}%
  \fi%
  \global\let\svgwidth\undefined%
  \global\let\svgscale\undefined%
  \makeatother%
  \begin{picture}(1,0.52242565)%
    \put(0,0){\includegraphics[width=\unitlength,page=1]{boundary-slide.pdf}}%
    \put(0.39170419,0.13359537){\color[rgb]{0,0,0}\makebox(0,0)[rb]{\smash{$\beta$}}}%
    \put(0.53201486,0.29985597){\color[rgb]{0,0,0}\makebox(0,0)[rb]{\smash{$\gamma$}}}%
    \put(0.30075765,0.39701458){\color[rgb]{0,0,0}\makebox(0,0)[rb]{\smash{$\alpha$}}}%
    \put(0,0){\includegraphics[width=\unitlength,page=2]{boundary-slide.pdf}}%
  \end{picture}%
\endgroup%

%% file: Section2.2.pdf_tex
\begingroup%
  \makeatletter%
  \providecommand\color[2][]{%
    \errmessage{(Inkscape) Color is used for the text in Inkscape, but the package 'color.sty' is not loaded}%
    \renewcommand\color[2][]{}%
  }%
  \providecommand\transparent[1]{%
    \errmessage{(Inkscape) Transparency is used (non-zero) for the text in Inkscape, but the package 'transparent.sty' is not loaded}%
    \renewcommand\transparent[1]{}%
  }%
  \providecommand\rotatebox[2]{#2}%
  \newcommand*\fsize{\dimexpr\f@size pt\relax}%
  \newcommand*\lineheight[1]{\fontsize{\fsize}{#1\fsize}\selectfont}%
  \ifx\svgwidth\undefined%
    \setlength{\unitlength}{1548.14066978bp}%
    \ifx\svgscale\undefined%
      \relax%
    \else%
      \setlength{\unitlength}{\unitlength * \real{\svgscale}}%
    \fi%
  \else%
    \setlength{\unitlength}{\svgwidth}%
  \fi%
  \global\let\svgwidth\undefined%
  \global\let\svgscale\undefined%
  \makeatother%
  \begin{picture}(1,0.158547)%
    \lineheight{1}%
    \setlength\tabcolsep{0pt}%
    \put(0,0){\includegraphics[width=\unitlength,page=1]{Section2.2.pdf}}%
    \put(0.39031016,0.09680152){\color[rgb]{0,0,0}\makebox(0,0)[t]{\lineheight{0.53333336}\smash{\begin{tabular}[t]{c}$\gamma$\end{tabular}}}}%
    \put(0.52336953,0.09621015){\color[rgb]{0,0,0}\makebox(0,0)[t]{\lineheight{0.53333336}\smash{\begin{tabular}[t]{c}$\eta$\end{tabular}}}}%
    \put(0.58528036,0.05166593){\color[rgb]{0,0,0}\makebox(0,0)[lt]{\lineheight{0.53333336}\smash{\begin{tabular}[t]{l}$\rho$\end{tabular}}}}%
    \put(0.28643158,0.03103837){\color[rgb]{0,0,0}\makebox(0,0)[rt]{\lineheight{0.53333336}\smash{\begin{tabular}[t]{r}$\omega$\end{tabular}}}}%
    \put(0.2224388,0.1057268){\color[rgb]{0,0,0}\makebox(0,0)[t]{\lineheight{0.53333336}\smash{\begin{tabular}[t]{c}$\alpha$\end{tabular}}}}%
    \put(0.17306067,0.10834265){\color[rgb]{0,0,0}\makebox(0,0)[rt]{\lineheight{0.53333336}\smash{\begin{tabular}[t]{r}$\beta$\end{tabular}}}}%
    \put(0.13187127,0.08044027){\color[rgb]{0,0,0}\makebox(0,0)[rt]{\lineheight{0.53333336}\smash{\begin{tabular}[t]{r}$\varepsilon$\end{tabular}}}}%
    \put(0,0){\includegraphics[width=\unitlength,page=2]{Section2.2.pdf}}%
    \put(0.79688374,0.07213511){\color[rgb]{0,0,0}\makebox(0,0)[rt]{\lineheight{0.53333336}\smash{\begin{tabular}[t]{r}$\varepsilon$\end{tabular}}}}%
    \put(0.86390857,0.14363495){\color[rgb]{0,0,0}\makebox(0,0)[t]{\lineheight{0.53333336}\smash{\begin{tabular}[t]{c}$\alpha$\end{tabular}}}}%
    \put(0.94807279,0.13425507){\color[rgb]{0,0,0}\makebox(0,0)[t]{\lineheight{0.53333336}\smash{\begin{tabular}[t]{c}$\gamma$\end{tabular}}}}%
    \put(0.92906498,0.08638551){\color[rgb]{0,0,0}\makebox(0,0)[lt]{\lineheight{0.53333336}\smash{\begin{tabular}[t]{l}$\eta$\end{tabular}}}}%
    \put(0.91335915,0.01839906){\color[rgb]{0,0,0}\makebox(0,0)[lt]{\lineheight{0.53333336}\smash{\begin{tabular}[t]{l}$\rho$\end{tabular}}}}%
    \put(0.81205365,0.00341577){\color[rgb]{0,0,0}\makebox(0,0)[rt]{\lineheight{0.53333336}\smash{\begin{tabular}[t]{r}$\beta$\end{tabular}}}}%
    \put(0,0){\includegraphics[width=\unitlength,page=3]{Section2.2.pdf}}%
  \end{picture}%
\endgroup%

%% file: NewFigure5.pdf_tex
\begingroup%
  \makeatletter%
  \providecommand\color[2][]{%
    \errmessage{(Inkscape) Color is used for the text in Inkscape, but the package 'color.sty' is not loaded}%
    \renewcommand\color[2][]{}%
  }%
  \providecommand\transparent[1]{%
    \errmessage{(Inkscape) Transparency is used (non-zero) for the text in Inkscape, but the package 'transparent.sty' is not loaded}%
    \renewcommand\transparent[1]{}%
  }%
  \providecommand\rotatebox[2]{#2}%
  \ifx\svgwidth\undefined%
    \setlength{\unitlength}{911.45178646bp}%
    \ifx\svgscale\undefined%
      \relax%
    \else%
      \setlength{\unitlength}{\unitlength * \real{\svgscale}}%
    \fi%
  \else%
    \setlength{\unitlength}{\svgwidth}%
  \fi%
  \global\let\svgwidth\undefined%
  \global\let\svgscale\undefined%
  \makeatother%
  \begin{picture}(1,0.25566744)%
    \put(0,0){\includegraphics[width=\unitlength,page=1]{NewFigure5.pdf}}%
    \put(0.21844027,0.19094484){\color[rgb]{0,0,0}\makebox(0,0)[b]{\smash{$u$}}}%
    \put(0.05777154,0.12743602){\color[rgb]{0,0,0}\makebox(0,0)[rb]{\smash{$v$}}}%
    \put(0,0){\includegraphics[width=\unitlength,page=2]{NewFigure5.pdf}}%
    \put(0.39339368,0.23066011){\color[rgb]{0,0,0}\makebox(0,0)[lb]{\smash{$\lambda_u$}}}%
    \put(0.42569613,0.02357069){\color[rgb]{0,0,0}\makebox(0,0)[lb]{\smash{$\lambda_v$}}}%
    \put(0,0){\includegraphics[width=\unitlength,page=3]{NewFigure5.pdf}}%
    \put(0.73848973,0.18765339){\color[rgb]{0,0,0}\makebox(0,0)[b]{\smash{$u$}}}%
    \put(0.83831139,0.1570591){\color[rgb]{0,0,0}\makebox(0,0)[b]{\smash{$v$}}}%
    \put(0,0){\includegraphics[width=\unitlength,page=4]{NewFigure5.pdf}}%
  \end{picture}%
\endgroup%

%% file: NewFigure6.pdf_tex
\begingroup%
  \makeatletter%
  \providecommand\color[2][]{%
    \errmessage{(Inkscape) Color is used for the text in Inkscape, but the package 'color.sty' is not loaded}%
    \renewcommand\color[2][]{}%
  }%
  \providecommand\transparent[1]{%
    \errmessage{(Inkscape) Transparency is used (non-zero) for the text in Inkscape, but the package 'transparent.sty' is not loaded}%
    \renewcommand\transparent[1]{}%
  }%
  \providecommand\rotatebox[2]{#2}%
  \ifx\svgwidth\undefined%
    \setlength{\unitlength}{909.96728138bp}%
    \ifx\svgscale\undefined%
      \relax%
    \else%
      \setlength{\unitlength}{\unitlength * \real{\svgscale}}%
    \fi%
  \else%
    \setlength{\unitlength}{\svgwidth}%
  \fi%
  \global\let\svgwidth\undefined%
  \global\let\svgscale\undefined%
  \makeatother%
  \begin{picture}(1,0.44452896)%
    \put(0,0){\includegraphics[width=\unitlength,page=1]{NewFigure6.pdf}}%
    \put(0.10675296,0.33785547){\color[rgb]{0,0,0}\makebox(0,0)[lb]{\smash{$u$}}}%
    \put(0.34592322,0.36782811){\color[rgb]{0,0,0}\makebox(0,0)[rb]{\smash{$v$}}}%
    \put(0,0){\includegraphics[width=\unitlength,page=2]{NewFigure6.pdf}}%
    \put(0.82692884,0.22046266){\color[rgb]{0,0,0}\makebox(0,0)[lb]{\smash{$u$}}}%
    \put(0.74490688,0.02231016){\color[rgb]{0,0,0}\makebox(0,0)[b]{\smash{$v$}}}%
  \end{picture}%
\endgroup%

%% file: subsurface-proj-1.pdf_tex
\begingroup%
  \makeatletter%
  \providecommand\color[2][]{%
    \errmessage{(Inkscape) Color is used for the text in Inkscape, but the package 'color.sty' is not loaded}%
    \renewcommand\color[2][]{}%
  }%
  \providecommand\transparent[1]{%
    \errmessage{(Inkscape) Transparency is used (non-zero) for the text in Inkscape, but the package 'transparent.sty' is not loaded}%
    \renewcommand\transparent[1]{}%
  }%
  \providecommand\rotatebox[2]{#2}%
  \ifx\svgwidth\undefined%
    \setlength{\unitlength}{1029.17229314bp}%
    \ifx\svgscale\undefined%
      \relax%
    \else%
      \setlength{\unitlength}{\unitlength * \real{\svgscale}}%
    \fi%
  \else%
    \setlength{\unitlength}{\svgwidth}%
  \fi%
  \global\let\svgwidth\undefined%
  \global\let\svgscale\undefined%
  \makeatother%
  \begin{picture}(1,0.66784756)%
    \put(0,0){\includegraphics[width=\unitlength,page=1]{subsurface-proj-1.pdf}}%
    \put(0.75137688,0.4331625){\color[rgb]{0,0,0}\makebox(0,0)[b]{\smash{$W$}}}%
    \put(0.58222211,0.24225279){\color[rgb]{0,0,0}\makebox(0,0)[lb]{\smash{$Z$}}}%
    \put(0.18092736,0.4331625){\color[rgb]{0,0,0}\makebox(0,0)[b]{\smash{$Y$}}}%
    \put(0.43429516,0.46947289){\color[rgb]{0,0,0}\makebox(0,0)[rb]{\smash{$X$}}}%
    \put(0,0){\includegraphics[width=\unitlength,page=2]{subsurface-proj-1.pdf}}%
    \put(0.94681045,0.49148775){\color[rgb]{0,0,0}\makebox(0,0)[lb]{\smash{$S$}}}%
    \put(0.76038476,0.25145822){\color[rgb]{0,0,0}\makebox(0,0)[lb]{\smash{$i_W$}}}%
    \put(0.45179331,0.23440696){\color[rgb]{0,0,0}\makebox(0,0)[lb]{\smash{$i_Z$}}}%
    \put(0.1719549,0.25014654){\color[rgb]{0,0,0}\makebox(0,0)[lb]{\smash{$i_Y$}}}%
    \put(0.45539234,0.55060463){\color[rgb]{0,0,0}\makebox(0,0)[rb]{\smash{$i_X$}}}%
    \put(0.29847799,0.36011834){\color[rgb]{0,0,0}\makebox(0,0)[rb]{\smash{$\alpha$}}}%
    \put(0,0){\includegraphics[width=\unitlength,page=3]{subsurface-proj-1.pdf}}%
  \end{picture}%
\endgroup%

%% file: subsurface-proj-2.pdf_tex
\begingroup%
  \makeatletter%
  \providecommand\color[2][]{%
    \errmessage{(Inkscape) Color is used for the text in Inkscape, but the package 'color.sty' is not loaded}%
    \renewcommand\color[2][]{}%
  }%
  \providecommand\transparent[1]{%
    \errmessage{(Inkscape) Transparency is used (non-zero) for the text in Inkscape, but the package 'transparent.sty' is not loaded}%
    \renewcommand\transparent[1]{}%
  }%
  \providecommand\rotatebox[2]{#2}%
  \ifx\svgwidth\undefined%
    \setlength{\unitlength}{1102.3944532bp}%
    \ifx\svgscale\undefined%
      \relax%
    \else%
      \setlength{\unitlength}{\unitlength * \real{\svgscale}}%
    \fi%
  \else%
    \setlength{\unitlength}{\svgwidth}%
  \fi%
  \global\let\svgwidth\undefined%
  \global\let\svgscale\undefined%
  \makeatother%
  \begin{picture}(1,0.57703078)%
    \put(0,0){\includegraphics[width=\unitlength,page=1]{subsurface-proj-2.pdf}}%
    \put(0.2977585,0.5134968){\color[rgb]{0,0,0}\makebox(0,0)[b]{\smash{$\gamma$}}}%
    \put(0.31894768,0.5563549){\color[rgb]{0,0,0}\makebox(0,0)[b]{\smash{$Y$}}}%
    \put(0.70978311,0.55513038){\color[rgb]{0,0,0}\makebox(0,0)[b]{\smash{$Z$}}}%
    \put(0,0){\includegraphics[width=\unitlength,page=2]{subsurface-proj-2.pdf}}%
    \put(0.71941169,0.3209941){\color[rgb]{0,0,0}\makebox(0,0)[lb]{\smash{$p_Z$}}}%
    \put(0.71940075,0.13900495){\color[rgb]{0,0,0}\makebox(0,0)[lb]{\smash{$\text{compactify}$}}}%
    \put(0.29312817,0.3209941){\color[rgb]{0,0,0}\makebox(0,0)[lb]{\smash{$\pi_Y$}}}%
    \put(0,0){\includegraphics[width=\unitlength,page=3]{subsurface-proj-2.pdf}}%
  \end{picture}%
\endgroup%

%% file: figure1.pdf_tex
\begingroup%
  \makeatletter%
  \providecommand\color[2][]{%
    \errmessage{(Inkscape) Color is used for the text in Inkscape, but the package 'color.sty' is not loaded}%
    \renewcommand\color[2][]{}%
  }%
  \providecommand\transparent[1]{%
    \errmessage{(Inkscape) Transparency is used (non-zero) for the text in Inkscape, but the package 'transparent.sty' is not loaded}%
    \renewcommand\transparent[1]{}%
  }%
  \providecommand\rotatebox[2]{#2}%
  \ifx\svgwidth\undefined%
    \setlength{\unitlength}{776.12809457bp}%
    \ifx\svgscale\undefined%
      \relax%
    \else%
      \setlength{\unitlength}{\unitlength * \real{\svgscale}}%
    \fi%
  \else%
    \setlength{\unitlength}{\svgwidth}%
  \fi%
  \global\let\svgwidth\undefined%
  \global\let\svgscale\undefined%
  \makeatother%
  \begin{picture}(1,0.32949917)%
    \put(0,0){\includegraphics[width=\unitlength,page=1]{figure1.pdf}}%
    \put(0.24210594,0.30013164){\color[rgb]{0,0,0}\makebox(0,0)[b]{\smash{$S_1$}}}%
    \put(0.07055003,0.13779104){\color[rgb]{0,0,0}\makebox(0,0)[rb]{\smash{$u$}}}%
    \put(0.28940793,0.22142104){\color[rgb]{0,0,0}\makebox(0,0)[b]{\smash{$v$}}}%
    \put(0.48076523,0.07506852){\color[rgb]{0,0,0}\makebox(0,0)[b]{\smash{$\alpha$}}}%
  \end{picture}%
\endgroup%

%% file: figure2.pdf_tex
\begingroup%
  \makeatletter%
  \providecommand\color[2][]{%
    \errmessage{(Inkscape) Color is used for the text in Inkscape, but the package 'color.sty' is not loaded}%
    \renewcommand\color[2][]{}%
  }%
  \providecommand\transparent[1]{%
    \errmessage{(Inkscape) Transparency is used (non-zero) for the text in Inkscape, but the package 'transparent.sty' is not loaded}%
    \renewcommand\transparent[1]{}%
  }%
  \providecommand\rotatebox[2]{#2}%
  \ifx\svgwidth\undefined%
    \setlength{\unitlength}{776.12809457bp}%
    \ifx\svgscale\undefined%
      \relax%
    \else%
      \setlength{\unitlength}{\unitlength * \real{\svgscale}}%
    \fi%
  \else%
    \setlength{\unitlength}{\svgwidth}%
  \fi%
  \global\let\svgwidth\undefined%
  \global\let\svgscale\undefined%
  \makeatother%
  \begin{picture}(1,0.32949917)%
    \put(0,0){\includegraphics[width=\unitlength,page=1]{figure2.pdf}}%
    \put(0.24210594,0.30013164){\color[rgb]{0,0,0}\makebox(0,0)[b]{\smash{$S_1$}}}%
    \put(0.07055003,0.13779104){\color[rgb]{0,0,0}\makebox(0,0)[rb]{\smash{$u$}}}%
    \put(0.31277514,0.20666282){\color[rgb]{0,0,0}\makebox(0,0)[b]{\smash{$v$}}}%
    \put(0,0){\includegraphics[width=\unitlength,page=2]{figure2.pdf}}%
    \put(0.42510939,0.18524245){\color[rgb]{0,0,0}\makebox(0,0)[lb]{\smash{$\alpha$}}}%
    \put(0.41768789,0.06768939){\color[rgb]{0,0,0}\makebox(0,0)[lb]{\smash{$\alpha$}}}%
  \end{picture}%
\endgroup%

%% file: NewFigure4-1.pdf_tex
\begingroup%
  \makeatletter%
  \providecommand\color[2][]{%
    \errmessage{(Inkscape) Color is used for the text in Inkscape, but the package 'color.sty' is not loaded}%
    \renewcommand\color[2][]{}%
  }%
  \providecommand\transparent[1]{%
    \errmessage{(Inkscape) Transparency is used (non-zero) for the text in Inkscape, but the package 'transparent.sty' is not loaded}%
    \renewcommand\transparent[1]{}%
  }%
  \providecommand\rotatebox[2]{#2}%
  \newcommand*\fsize{\dimexpr\f@size pt\relax}%
  \newcommand*\lineheight[1]{\fontsize{\fsize}{#1\fsize}\selectfont}%
  \ifx\svgwidth\undefined%
    \setlength{\unitlength}{692.75563019bp}%
    \ifx\svgscale\undefined%
      \relax%
    \else%
      \setlength{\unitlength}{\unitlength * \real{\svgscale}}%
    \fi%
  \else%
    \setlength{\unitlength}{\svgwidth}%
  \fi%
  \global\let\svgwidth\undefined%
  \global\let\svgscale\undefined%
  \makeatother%
  \begin{picture}(1,0.32140678)%
    \lineheight{1}%
    \setlength\tabcolsep{0pt}%
    \put(0,0){\includegraphics[width=\unitlength,page=1]{NewFigure4-1.pdf}}%
  \end{picture}%
\endgroup%

%% file: NewFigure4-2.pdf_tex
\begingroup%
  \makeatletter%
  \providecommand\color[2][]{%
    \errmessage{(Inkscape) Color is used for the text in Inkscape, but the package 'color.sty' is not loaded}%
    \renewcommand\color[2][]{}%
  }%
  \providecommand\transparent[1]{%
    \errmessage{(Inkscape) Transparency is used (non-zero) for the text in Inkscape, but the package 'transparent.sty' is not loaded}%
    \renewcommand\transparent[1]{}%
  }%
  \providecommand\rotatebox[2]{#2}%
  \newcommand*\fsize{\dimexpr\f@size pt\relax}%
  \newcommand*\lineheight[1]{\fontsize{\fsize}{#1\fsize}\selectfont}%
  \ifx\svgwidth\undefined%
    \setlength{\unitlength}{659.34730228bp}%
    \ifx\svgscale\undefined%
      \relax%
    \else%
      \setlength{\unitlength}{\unitlength * \real{\svgscale}}%
    \fi%
  \else%
    \setlength{\unitlength}{\svgwidth}%
  \fi%
  \global\let\svgwidth\undefined%
  \global\let\svgscale\undefined%
  \makeatother%
  \begin{picture}(1,0.34193527)%
    \lineheight{1}%
    \setlength\tabcolsep{0pt}%
    \put(0,0){\includegraphics[width=\unitlength,page=1]{NewFigure4-2.pdf}}%
  \end{picture}%
\endgroup%

%% file: Lemma4-1.pdf_tex
\begingroup%
  \makeatletter%
  \providecommand\color[2][]{%
    \errmessage{(Inkscape) Color is used for the text in Inkscape, but the package 'color.sty' is not loaded}%
    \renewcommand\color[2][]{}%
  }%
  \providecommand\transparent[1]{%
    \errmessage{(Inkscape) Transparency is used (non-zero) for the text in Inkscape, but the package 'transparent.sty' is not loaded}%
    \renewcommand\transparent[1]{}%
  }%
  \providecommand\rotatebox[2]{#2}%
  \ifx\svgwidth\undefined%
    \setlength{\unitlength}{446.14621629bp}%
    \ifx\svgscale\undefined%
      \relax%
    \else%
      \setlength{\unitlength}{\unitlength * \real{\svgscale}}%
    \fi%
  \else%
    \setlength{\unitlength}{\svgwidth}%
  \fi%
  \global\let\svgwidth\undefined%
  \global\let\svgscale\undefined%
  \makeatother%
  \begin{picture}(1,0.58524354)%
    \put(0,0){\includegraphics[width=\unitlength,page=1]{Lemma4-1.pdf}}%
  \end{picture}%
\endgroup%

%% file: Section5Fix.pdf_tex
\begingroup%
  \makeatletter%
  \providecommand\color[2][]{%
    \errmessage{(Inkscape) Color is used for the text in Inkscape, but the package 'color.sty' is not loaded}%
    \renewcommand\color[2][]{}%
  }%
  \providecommand\transparent[1]{%
    \errmessage{(Inkscape) Transparency is used (non-zero) for the text in Inkscape, but the package 'transparent.sty' is not loaded}%
    \renewcommand\transparent[1]{}%
  }%
  \providecommand\rotatebox[2]{#2}%
  \newcommand*\fsize{\dimexpr\f@size pt\relax}%
  \newcommand*\lineheight[1]{\fontsize{\fsize}{#1\fsize}\selectfont}%
  \ifx\svgwidth\undefined%
    \setlength{\unitlength}{1143.89139117bp}%
    \ifx\svgscale\undefined%
      \relax%
    \else%
      \setlength{\unitlength}{\unitlength * \real{\svgscale}}%
    \fi%
  \else%
    \setlength{\unitlength}{\svgwidth}%
  \fi%
  \global\let\svgwidth\undefined%
  \global\let\svgscale\undefined%
  \makeatother%
  \begin{picture}(1,0.41897834)%
    \lineheight{1}%
    \setlength\tabcolsep{0pt}%
    \put(0,0){\includegraphics[width=\unitlength,page=1]{Section5Fix.pdf}}%
  \end{picture}%
\endgroup%

%% file: Lemma5.2-final-eta.pdf_tex
\begingroup%
  \makeatletter%
  \providecommand\color[2][]{%
    \errmessage{(Inkscape) Color is used for the text in Inkscape, but the package 'color.sty' is not loaded}%
    \renewcommand\color[2][]{}%
  }%
  \providecommand\transparent[1]{%
    \errmessage{(Inkscape) Transparency is used (non-zero) for the text in Inkscape, but the package 'transparent.sty' is not loaded}%
    \renewcommand\transparent[1]{}%
  }%
  \providecommand\rotatebox[2]{#2}%
  \ifx\svgwidth\undefined%
    \setlength{\unitlength}{446.14621629bp}%
    \ifx\svgscale\undefined%
      \relax%
    \else%
      \setlength{\unitlength}{\unitlength * \real{\svgscale}}%
    \fi%
  \else%
    \setlength{\unitlength}{\svgwidth}%
  \fi%
  \global\let\svgwidth\undefined%
  \global\let\svgscale\undefined%
  \makeatother%
  \begin{picture}(1,0.9159826)%
    \put(0,0){\includegraphics[width=\unitlength,page=1]{Lemma5.2-final-eta.pdf}}%
    \put(0.23557483,0.24678324){\color[rgb]{0,0,0}\makebox(0,0)[rb]{\smash{$\eta$}}}%
  \end{picture}%
\endgroup%

%% file: Lemma5.2-u-v-config.pdf_tex
\begingroup%
  \makeatletter%
  \providecommand\color[2][]{%
    \errmessage{(Inkscape) Color is used for the text in Inkscape, but the package 'color.sty' is not loaded}%
    \renewcommand\color[2][]{}%
  }%
  \providecommand\transparent[1]{%
    \errmessage{(Inkscape) Transparency is used (non-zero) for the text in Inkscape, but the package 'transparent.sty' is not loaded}%
    \renewcommand\transparent[1]{}%
  }%
  \providecommand\rotatebox[2]{#2}%
  \ifx\svgwidth\undefined%
    \setlength{\unitlength}{489.33738788bp}%
    \ifx\svgscale\undefined%
      \relax%
    \else%
      \setlength{\unitlength}{\unitlength * \real{\svgscale}}%
    \fi%
  \else%
    \setlength{\unitlength}{\svgwidth}%
  \fi%
  \global\let\svgwidth\undefined%
  \global\let\svgscale\undefined%
  \makeatother%
  \begin{picture}(1,0.68702469)%
    \put(0,0){\includegraphics[width=\unitlength,page=1]{Lemma5.2-u-v-config.pdf}}%
    \put(0.23417122,0.01445873){\color[rgb]{0,0,0}\makebox(0,0)[b]{\smash{$b_1$}}}%
    \put(-0.72565284,1.97763405){\color[rgb]{0,0,0}\makebox(0,0)[lb]{\smash{}}}%
    \put(0.86175855,0.01755517){\color[rgb]{0,0,0}\makebox(0,0)[b]{\smash{$b_2$}}}%
    \put(0.10926366,0.32720434){\color[rgb]{0,0,0}\makebox(0,0)[rb]{\smash{$v$}}}%
    \put(0.53846465,0.57182717){\color[rgb]{0,0,0}\makebox(0,0)[b]{\smash{$u$}}}%
    \put(0.23417122,0.64044544){\color[rgb]{0,0,0}\makebox(0,0)[b]{\smash{$p_1$}}}%
    \put(0.86175855,0.64044535){\color[rgb]{0,0,0}\makebox(0,0)[b]{\smash{$p_2$}}}%
    \put(0,0){\includegraphics[width=\unitlength,page=2]{Lemma5.2-u-v-config.pdf}}%
  \end{picture}%
\endgroup%

%% file: AlphaConfigIncluded.pdf_tex
\begingroup%
  \makeatletter%
  \providecommand\color[2][]{%
    \errmessage{(Inkscape) Color is used for the text in Inkscape, but the package 'color.sty' is not loaded}%
    \renewcommand\color[2][]{}%
  }%
  \providecommand\transparent[1]{%
    \errmessage{(Inkscape) Transparency is used (non-zero) for the text in Inkscape, but the package 'transparent.sty' is not loaded}%
    \renewcommand\transparent[1]{}%
  }%
  \providecommand\rotatebox[2]{#2}%
  \newcommand*\fsize{\dimexpr\f@size pt\relax}%
  \newcommand*\lineheight[1]{\fontsize{\fsize}{#1\fsize}\selectfont}%
  \ifx\svgwidth\undefined%
    \setlength{\unitlength}{1606.18326842bp}%
    \ifx\svgscale\undefined%
      \relax%
    \else%
      \setlength{\unitlength}{\unitlength * \real{\svgscale}}%
    \fi%
  \else%
    \setlength{\unitlength}{\svgwidth}%
  \fi%
  \global\let\svgwidth\undefined%
  \global\let\svgscale\undefined%
  \makeatother%
  \begin{picture}(1,0.22099852)%
    \lineheight{1}%
    \setlength\tabcolsep{0pt}%
    \put(0,0){\includegraphics[width=\unitlength,page=1]{AlphaConfigIncluded.pdf}}%
    \put(0.85970241,0.00375744){\color[rgb]{0,0,0}\makebox(0,0)[t]{\lineheight{1.25}\smash{\begin{tabular}[t]{c}$(iii)$\end{tabular}}}}%
    \put(0,0){\includegraphics[width=\unitlength,page=2]{AlphaConfigIncluded.pdf}}%
    \put(0.13747037,0.00375744){\color[rgb]{0,0,0}\makebox(0,0)[t]{\lineheight{1.25}\smash{\begin{tabular}[t]{c}$(i)$\end{tabular}}}}%
    \put(0,0){\includegraphics[width=\unitlength,page=3]{AlphaConfigIncluded.pdf}}%
    \put(0.4985864,0.00375744){\color[rgb]{0,0,0}\makebox(0,0)[t]{\lineheight{1.25}\smash{\begin{tabular}[t]{c}$(ii)$\end{tabular}}}}%
  \end{picture}%
\endgroup%

%% file: Section5.pdf_tex
\begingroup%
  \makeatletter%
  \providecommand\color[2][]{%
    \errmessage{(Inkscape) Color is used for the text in Inkscape, but the package 'color.sty' is not loaded}%
    \renewcommand\color[2][]{}%
  }%
  \providecommand\transparent[1]{%
    \errmessage{(Inkscape) Transparency is used (non-zero) for the text in Inkscape, but the package 'transparent.sty' is not loaded}%
    \renewcommand\transparent[1]{}%
  }%
  \providecommand\rotatebox[2]{#2}%
  \ifx\svgwidth\undefined%
    \setlength{\unitlength}{676.08223126bp}%
    \ifx\svgscale\undefined%
      \relax%
    \else%
      \setlength{\unitlength}{\unitlength * \real{\svgscale}}%
    \fi%
  \else%
    \setlength{\unitlength}{\svgwidth}%
  \fi%
  \global\let\svgwidth\undefined%
  \global\let\svgscale\undefined%
  \makeatother%
  \begin{picture}(1,0.94999593)%
    \put(0,0){\includegraphics[width=\unitlength,page=1]{Section5.pdf}}%
    \put(0.63119785,0.13876427){\color[rgb]{0,0,0}\makebox(0,0)[lb]{\smash{$\eta$}}}%
    \put(0.2457852,0.13876427){\color[rgb]{0,0,0}\makebox(0,0)[rb]{\smash{$\eta'$}}}%
    \put(0.46489414,0.19453213){\color[rgb]{0,0,0}\makebox(0,0)[rb]{\smash{$\lambda$}}}%
  \end{picture}%
\endgroup%

%% file: Section5-linking.pdf_tex
\begingroup%
  \makeatletter%
  \providecommand\color[2][]{%
    \errmessage{(Inkscape) Color is used for the text in Inkscape, but the package 'color.sty' is not loaded}%
    \renewcommand\color[2][]{}%
  }%
  \providecommand\transparent[1]{%
    \errmessage{(Inkscape) Transparency is used (non-zero) for the text in Inkscape, but the package 'transparent.sty' is not loaded}%
    \renewcommand\transparent[1]{}%
  }%
  \providecommand\rotatebox[2]{#2}%
  \ifx\svgwidth\undefined%
    \setlength{\unitlength}{397.58560936bp}%
    \ifx\svgscale\undefined%
      \relax%
    \else%
      \setlength{\unitlength}{\unitlength * \real{\svgscale}}%
    \fi%
  \else%
    \setlength{\unitlength}{\svgwidth}%
  \fi%
  \global\let\svgwidth\undefined%
  \global\let\svgscale\undefined%
  \makeatother%
  \begin{picture}(1,0.56382504)%
    \put(0,0){\includegraphics[width=\unitlength,page=1]{Section5-linking.pdf}}%
    \put(0.25562375,0.45225491){\color[rgb]{0,0,0}\makebox(0,0)[rb]{\smash{$\eta$}}}%
    \put(0.47258764,0.32775336){\color[rgb]{0,0,0}\makebox(0,0)[rb]{\smash{$\eta'$}}}%
  \end{picture}%
\endgroup%

%% file: Proposition5.2-u-v-configs.pdf_tex
\begingroup%
  \makeatletter%
  \providecommand\color[2][]{%
    \errmessage{(Inkscape) Color is used for the text in Inkscape, but the package 'color.sty' is not loaded}%
    \renewcommand\color[2][]{}%
  }%
  \providecommand\transparent[1]{%
    \errmessage{(Inkscape) Transparency is used (non-zero) for the text in Inkscape, but the package 'transparent.sty' is not loaded}%
    \renewcommand\transparent[1]{}%
  }%
  \providecommand\rotatebox[2]{#2}%
  \ifx\svgwidth\undefined%
    \setlength{\unitlength}{1450.92491538bp}%
    \ifx\svgscale\undefined%
      \relax%
    \else%
      \setlength{\unitlength}{\unitlength * \real{\svgscale}}%
    \fi%
  \else%
    \setlength{\unitlength}{\svgwidth}%
  \fi%
  \global\let\svgwidth\undefined%
  \global\let\svgscale\undefined%
  \makeatother%
  \begin{picture}(1,0.33776319)%
    \put(0,0){\includegraphics[width=\unitlength,page=1]{Proposition5.2-u-v-configs.pdf}}%
    \put(0.23148007,0.14954774){\color[rgb]{0,0,0}\makebox(0,0)[b]{\smash{$2$}}}%
    \put(0.23148007,0.25548894){\color[rgb]{0,0,0}\makebox(0,0)[b]{\smash{$1$}}}%
    \put(0.35573655,0.14565497){\color[rgb]{0,0,0}\makebox(0,0)[b]{\smash{$3$}}}%
    \put(0.35480618,0.25636923){\color[rgb]{0,0,0}\makebox(0,0)[b]{\smash{$4$}}}%
    \put(0,0){\includegraphics[width=\unitlength,page=2]{Proposition5.2-u-v-configs.pdf}}%
    \put(0.76551369,0.14396552){\color[rgb]{0,0,0}\makebox(0,0)[b]{\smash{$2$}}}%
    \put(0.76551369,0.24990671){\color[rgb]{0,0,0}\makebox(0,0)[b]{\smash{$1$}}}%
    \put(0.88977016,0.14007274){\color[rgb]{0,0,0}\makebox(0,0)[b]{\smash{$3$}}}%
    \put(0.88883979,0.25078701){\color[rgb]{0,0,0}\makebox(0,0)[b]{\smash{$3$}}}%
  \end{picture}%
\endgroup%

%% file: Section6-shortcut-set.pdf_tex
\begingroup%
  \makeatletter%
  \providecommand\color[2][]{%
    \errmessage{(Inkscape) Color is used for the text in Inkscape, but the package 'color.sty' is not loaded}%
    \renewcommand\color[2][]{}%
  }%
  \providecommand\transparent[1]{%
    \errmessage{(Inkscape) Transparency is used (non-zero) for the text in Inkscape, but the package 'transparent.sty' is not loaded}%
    \renewcommand\transparent[1]{}%
  }%
  \providecommand\rotatebox[2]{#2}%
  \ifx\svgwidth\undefined%
    \setlength{\unitlength}{397.58560936bp}%
    \ifx\svgscale\undefined%
      \relax%
    \else%
      \setlength{\unitlength}{\unitlength * \real{\svgscale}}%
    \fi%
  \else%
    \setlength{\unitlength}{\svgwidth}%
  \fi%
  \global\let\svgwidth\undefined%
  \global\let\svgscale\undefined%
  \makeatother%
  \begin{picture}(1,0.56382504)%
    \put(0,0){\includegraphics[width=\unitlength,page=1]{Section6-shortcut-set.pdf}}%
  \end{picture}%
\endgroup%

%% file: Proposition6.7-1.pdf_tex
\begingroup%
  \makeatletter%
  \providecommand\color[2][]{%
    \errmessage{(Inkscape) Color is used for the text in Inkscape, but the package 'color.sty' is not loaded}%
    \renewcommand\color[2][]{}%
  }%
  \providecommand\transparent[1]{%
    \errmessage{(Inkscape) Transparency is used (non-zero) for the text in Inkscape, but the package 'transparent.sty' is not loaded}%
    \renewcommand\transparent[1]{}%
  }%
  \providecommand\rotatebox[2]{#2}%
  \ifx\svgwidth\undefined%
    \setlength{\unitlength}{811.81910428bp}%
    \ifx\svgscale\undefined%
      \relax%
    \else%
      \setlength{\unitlength}{\unitlength * \real{\svgscale}}%
    \fi%
  \else%
    \setlength{\unitlength}{\svgwidth}%
  \fi%
  \global\let\svgwidth\undefined%
  \global\let\svgscale\undefined%
  \makeatother%
  \begin{picture}(1,0.28103898)%
    \put(0,0){\includegraphics[width=\unitlength,page=1]{Proposition6.7-1.pdf}}%
  \end{picture}%
\endgroup%

%% file: Proposition6.7-2.pdf_tex
\begingroup%
  \makeatletter%
  \providecommand\color[2][]{%
    \errmessage{(Inkscape) Color is used for the text in Inkscape, but the package 'color.sty' is not loaded}%
    \renewcommand\color[2][]{}%
  }%
  \providecommand\transparent[1]{%
    \errmessage{(Inkscape) Transparency is used (non-zero) for the text in Inkscape, but the package 'transparent.sty' is not loaded}%
    \renewcommand\transparent[1]{}%
  }%
  \providecommand\rotatebox[2]{#2}%
  \ifx\svgwidth\undefined%
    \setlength{\unitlength}{382.64731925bp}%
    \ifx\svgscale\undefined%
      \relax%
    \else%
      \setlength{\unitlength}{\unitlength * \real{\svgscale}}%
    \fi%
  \else%
    \setlength{\unitlength}{\svgwidth}%
  \fi%
  \global\let\svgwidth\undefined%
  \global\let\svgscale\undefined%
  \makeatother%
  \begin{picture}(1,0.59609703)%
    \put(0,0){\includegraphics[width=\unitlength,page=1]{Proposition6.7-2.pdf}}%
    \put(0.82772392,0.53653051){\color[rgb]{0,0,0}\makebox(0,0)[b]{\smash{$u$}}}%
    \put(0.49755709,0.35358505){\color[rgb]{0,0,0}\makebox(0,0)[b]{\smash{$c_1$}}}%
    \put(0.67213325,0.33140986){\color[rgb]{0,0,0}\makebox(0,0)[lb]{\smash{$c_2$}}}%
  \end{picture}%
\endgroup%

%% file: Proposition6.7-3.pdf_tex
\begingroup%
  \makeatletter%
  \providecommand\color[2][]{%
    \errmessage{(Inkscape) Color is used for the text in Inkscape, but the package 'color.sty' is not loaded}%
    \renewcommand\color[2][]{}%
  }%
  \providecommand\transparent[1]{%
    \errmessage{(Inkscape) Transparency is used (non-zero) for the text in Inkscape, but the package 'transparent.sty' is not loaded}%
    \renewcommand\transparent[1]{}%
  }%
  \providecommand\rotatebox[2]{#2}%
  \ifx\svgwidth\undefined%
    \setlength{\unitlength}{382.64731925bp}%
    \ifx\svgscale\undefined%
      \relax%
    \else%
      \setlength{\unitlength}{\unitlength * \real{\svgscale}}%
    \fi%
  \else%
    \setlength{\unitlength}{\svgwidth}%
  \fi%
  \global\let\svgwidth\undefined%
  \global\let\svgscale\undefined%
  \makeatother%
  \begin{picture}(1,0.65204755)%
    \put(0,0){\includegraphics[width=\unitlength,page=1]{Proposition6.7-3.pdf}}%
    \put(0.83706143,0.59109504){\color[rgb]{0,0,0}\makebox(0,0)[b]{\smash{$u$}}}%
    \put(0.20922602,0.59248103){\color[rgb]{0,0,0}\makebox(0,0)[b]{\smash{$v$}}}%
    \put(0,0){\includegraphics[width=\unitlength,page=2]{Proposition6.7-3.pdf}}%
  \end{picture}%
\endgroup%

%% file: SepCurves.pdf_tex
\begingroup%
  \makeatletter%
  \providecommand\color[2][]{%
    \errmessage{(Inkscape) Color is used for the text in Inkscape, but the package 'color.sty' is not loaded}%
    \renewcommand\color[2][]{}%
  }%
  \providecommand\transparent[1]{%
    \errmessage{(Inkscape) Transparency is used (non-zero) for the text in Inkscape, but the package 'transparent.sty' is not loaded}%
    \renewcommand\transparent[1]{}%
  }%
  \providecommand\rotatebox[2]{#2}%
  \newcommand*\fsize{\dimexpr\f@size pt\relax}%
  \newcommand*\lineheight[1]{\fontsize{\fsize}{#1\fsize}\selectfont}%
  \ifx\svgwidth\undefined%
    \setlength{\unitlength}{627.82105549bp}%
    \ifx\svgscale\undefined%
      \relax%
    \else%
      \setlength{\unitlength}{\unitlength * \real{\svgscale}}%
    \fi%
  \else%
    \setlength{\unitlength}{\svgwidth}%
  \fi%
  \global\let\svgwidth\undefined%
  \global\let\svgscale\undefined%
  \makeatother%
  \begin{picture}(1,0.99976398)%
    \lineheight{1}%
    \setlength\tabcolsep{0pt}%
    \put(0,0){\includegraphics[width=\unitlength,page=1]{SepCurves.pdf}}%
    \put(0.08765902,0.72548111){\color[rgb]{0,0,0}\makebox(0,0)[rt]{\lineheight{1.25}\smash{\begin{tabular}[t]{r}$6$\end{tabular}}}}%
    \put(0.50183131,0.96299245){\color[rgb]{0,0,0}\makebox(0,0)[t]{\lineheight{1.25}\smash{\begin{tabular}[t]{c}$1$\end{tabular}}}}%
    \put(0.91208432,0.72548111){\color[rgb]{0,0,0}\makebox(0,0)[lt]{\lineheight{1.25}\smash{\begin{tabular}[t]{l}$2$\end{tabular}}}}%
    \put(0.91208432,0.24819636){\color[rgb]{0,0,0}\makebox(0,0)[lt]{\lineheight{1.25}\smash{\begin{tabular}[t]{l}$3$\end{tabular}}}}%
    \put(0.50183131,0.00842292){\color[rgb]{0,0,0}\makebox(0,0)[t]{\lineheight{1.25}\smash{\begin{tabular}[t]{c}$4$\end{tabular}}}}%
    \put(0.08765902,0.24819636){\color[rgb]{0,0,0}\makebox(0,0)[rt]{\lineheight{1.25}\smash{\begin{tabular}[t]{r}$5$\end{tabular}}}}%
  \end{picture}%
\endgroup%

%% file: SepAnnulus.pdf_tex
\begingroup%
  \makeatletter%
  \providecommand\color[2][]{%
    \errmessage{(Inkscape) Color is used for the text in Inkscape, but the package 'color.sty' is not loaded}%
    \renewcommand\color[2][]{}%
  }%
  \providecommand\transparent[1]{%
    \errmessage{(Inkscape) Transparency is used (non-zero) for the text in Inkscape, but the package 'transparent.sty' is not loaded}%
    \renewcommand\transparent[1]{}%
  }%
  \providecommand\rotatebox[2]{#2}%
  \newcommand*\fsize{\dimexpr\f@size pt\relax}%
  \newcommand*\lineheight[1]{\fontsize{\fsize}{#1\fsize}\selectfont}%
  \ifx\svgwidth\undefined%
    \setlength{\unitlength}{623.44117269bp}%
    \ifx\svgscale\undefined%
      \relax%
    \else%
      \setlength{\unitlength}{\unitlength * \real{\svgscale}}%
    \fi%
  \else%
    \setlength{\unitlength}{\svgwidth}%
  \fi%
  \global\let\svgwidth\undefined%
  \global\let\svgscale\undefined%
  \makeatother%
  \begin{picture}(1,1.00678765)%
    \lineheight{1}%
    \setlength\tabcolsep{0pt}%
    \put(0,0){\includegraphics[width=\unitlength,page=1]{SepAnnulus.pdf}}%
    \put(0.08827486,0.73057785){\color[rgb]{0,0,0}\makebox(0,0)[rt]{\lineheight{1.25}\smash{\begin{tabular}[t]{r}$6$\end{tabular}}}}%
    \put(0.49833151,0.96975779){\color[rgb]{0,0,0}\makebox(0,0)[t]{\lineheight{1.25}\smash{\begin{tabular}[t]{c}$1$\end{tabular}}}}%
    \put(0.91146669,0.73057785){\color[rgb]{0,0,0}\makebox(0,0)[lt]{\lineheight{1.25}\smash{\begin{tabular}[t]{l}$2$\end{tabular}}}}%
    \put(0.91146669,0.24994002){\color[rgb]{0,0,0}\makebox(0,0)[lt]{\lineheight{1.25}\smash{\begin{tabular}[t]{l}$3$\end{tabular}}}}%
    \put(0.49833151,0.00848209){\color[rgb]{0,0,0}\makebox(0,0)[t]{\lineheight{1.25}\smash{\begin{tabular}[t]{c}$4$\end{tabular}}}}%
    \put(0.08827486,0.24994002){\color[rgb]{0,0,0}\makebox(0,0)[rt]{\lineheight{1.25}\smash{\begin{tabular}[t]{r}$5$\end{tabular}}}}%
    \put(0,0){\includegraphics[width=\unitlength,page=2]{SepAnnulus.pdf}}%
  \end{picture}%
\endgroup%

%% file: paper.bbl
\begin{thebibliography}{10}

\bibitem{APP}
James~W. Anderson, Hugo Parlier, and Alexandra Pettet.
\newblock Small filling sets of curves on a surface.
\newblock {\em Topology Appl.}, 158(1):84--92, 2011.

\bibitem{Aougab2}
Tarik Aougab.
\newblock Uniform hyperbolicity of the graphs of curves.
\newblock {\em Geom. Topol.}, 17(5):2855--2875, 2013.

\bibitem{Aramayona}
Javier Aramayona.
\newblock Simplicial embeddings between pants graphs.
\newblock {\em Geom. Dedicata}, 144:115--128, 2010.

\bibitem{AramayonaKoberdaParlier}
Javier Aramayona, Thomas Koberda, and Hugo Parlier.
\newblock Injective maps between flip graphs.
\newblock {\em Ann. Inst. Fourier (Grenoble)}, 65(5):2037--2055, 2015.

\bibitem{AramayonaLeininger}
Javier Aramayona and Christopher~J. Leininger.
\newblock Finite rigid sets in curve complexes.
\newblock {\em J. Topol. Anal.}, 5(2):183--203, 2013.

\bibitem{BrendleMargalit}
Tara~E. Brendle and Dan Margalit.
\newblock Normal subgroups of mapping class groups and the metaconjecture of
  {I}vanov.
\newblock {\em J. Amer. Math. Soc.}, 32(4):1009--1070, 2019.

\bibitem{Disarlo}
Valentina Disarlo.
\newblock Combinatorial rigidity of arc complexes.
\newblock {\em J. Amer. Math. Soc.}, 32(4):1009--1070, 2015.

\bibitem{Farb}
Benson Farb, editor.
\newblock {\em Problems on mapping class groups and related topics}, volume~74
  of {\em Proceedings of Symposia in Pure Mathematics}.
\newblock American Mathematical Society, Providence, RI, 2006.

\bibitem{primer}
Benson Farb and Dan Margalit.
\newblock {\em A primer on mapping class groups}, volume~49 of {\em Princeton
  Mathematical Series}.
\newblock Princeton University Press, Princeton, NJ, 2012.

\bibitem{Fiorini}
Samuel Fiorini, Gwena\"{e}l Joret, Dirk~Oliver Theis, and David~R. Wood.
\newblock Small minors in dense graphs.
\newblock {\em European J. Combin.}, 33(6):1226--1245, 2012.

\bibitem{Hatcher}
Allen Hatcher.
\newblock On triangulations of surfaces.
\newblock {\em Topology Appl.}, 40(2):189--194, 1991.

\bibitem{Hempel}
John Hempel.
\newblock 3-manifolds as viewed from the curve complex.
\newblock {\em Topology}, 40(3):631--657, 2001.

\bibitem{Irmak2}
Elmas Irmak.
\newblock Superinjective simplicial maps of complexes of curves and injective
  homomorphisms of subgroups of mapping class groups.
\newblock {\em Topology}, 43(3):513--541, 2004.

\bibitem{Irmak}
Elmas Irmak.
\newblock Complexes of nonseparating curves and mapping class groups.
\newblock {\em Michigan Math. J.}, 54(1):81--110, 2006.

\bibitem{IrmakKorkmaz}
Elmas Irmak and Mustafa Korkmaz.
\newblock Automorphisms of the {H}atcher-{T}hurston complex.
\newblock {\em Israel J. Math.}, 162:183--196, 2007.

\bibitem{IrmakMcCarthy}
Elmas Irmak and John~D. McCarthy.
\newblock Injective simplicial maps of the arc complex.
\newblock {\em Turkish J. Math.}, 34(3):339--354, 2010.

\bibitem{Ivanov}
Nikolai~V. Ivanov.
\newblock Automorphism of complexes of curves and of {T}eichm\"{u}ller spaces.
\newblock {\em Internat. Math. Res. Notices}, (14):651--666, 1997.

\bibitem{KorkmazPapa}
Mustafa Korkmaz and Athanase Papadopoulos.
\newblock On the arc and curve complex of a surface.
\newblock {\em Math. Proc. Cambridge Philos. Soc.}, 148(3):473--483, 2010.

\bibitem{KorkmazPapa2}
Mustafa Korkmaz and Athanase Papadopoulos.
\newblock On the ideal triangulation graph of a punctured surface.
\newblock {\em Ann. Inst. Fourier (Grenoble)}, 62(4):1367--1382, 2012.

\bibitem{Lickorish}
W.~B.~R. Lickorish.
\newblock A representation of orientable combinatorial {$3$}-manifolds.
\newblock {\em Ann. of Math. (2)}, 76:531--540, 1962.

\bibitem{Luo}
Feng Luo.
\newblock Automorphisms of the complex of curves.
\newblock {\em Topology}, 39(2):283--298, 2000.

\bibitem{Margalit}
Dan Margalit.
\newblock Problems, questions, and conjectures about mapping class groups.
\newblock In {\em Breadth in contemporary topology}, volume 102 of {\em Proc.
  Sympos. Pure Math.}, pages 157--186. Amer. Math. Soc., Providence, RI, 2019.

\bibitem{MasurMinsky}
Howard~A. Masur and Yair~N. Minsky.
\newblock Geometry of the complex of curves. {I}. {H}yperbolicity.
\newblock {\em Invent. Math.}, 138(1):103--149, 1999.

\bibitem{MasurMinsky2}
Howard~A. Masur and Yair~N. Minsky.
\newblock Geometry of the complex of curves. {II}. {H}ierarchical structure.
\newblock {\em Geom. Funct. Anal.}, 10(4):902--974, 2000.

\bibitem{McCarthyPapa}
John~D. McCarthy and Athanase Papadopoulos.
\newblock Simplicial actions of mapping class groups.
\newblock In {\em Handbook of {T}eichm\"{u}ller theory. {V}olume {III}},
  volume~17 of {\em IRMA Lect. Math. Theor. Phys.}, pages 297--423. Eur. Math.
  Soc., Z\"{u}rich, 2012.

\bibitem{McLeay}
Alan McLeay.
\newblock Geometric normal subgroups in mapping class groups of punctured
  surfaces.
\newblock {\em New York J. Math.}, 25:839--888, 2019.

\bibitem{RafiSchleimer}
Kasra Rafi and Saul Schleimer.
\newblock Curve complexes are rigid.
\newblock {\em Duke Math. J.}, 158(2):225--246, 2011.

\bibitem{Schleimer}
Saul Schleimer.
\newblock Notes on the complex of curves.
\newblock http://homepages.warwick.ac.uk/~masgar/Maths/notes.pdf.

\bibitem{Schaller}
Paul Schmutz~Schaller.
\newblock Mapping class groups of hyperbolic surfaces and automorphism groups
  of graphs.
\newblock {\em Compositio Math.}, 122(3):243--260, 2000.

\end{thebibliography}
